\documentclass[]{amsart}

\usepackage{amssymb}
\usepackage{amsfonts}
\usepackage{amsmath}
\usepackage{amsthm}
\usepackage{mathtools}
\usepackage{graphicx}
\usepackage{mathrsfs}
\usepackage{dsfont}
\usepackage{amscd}
\usepackage{multirow}
\usepackage[all]{xy} \normalfont
\usepackage[T1]{fontenc}
\usepackage{calligra}
\usepackage{verbatim}
\usepackage[usenames,dvipsnames]{color}
\usepackage[colorlinks=true,linkcolor=Blue,citecolor=Violet]{hyperref}
\usepackage{enumerate}
\usepackage{nicematrix}

\newcommand{\SiepG}{{Sierpi{\'n}ski gasket}}
\newcommand{\SG}[1]{\mathcal{SG}_{#1}}

\newcommand{\N}{{\mathds{N}}}
\newcommand{\Z}{{\mathds{Z}}}

\newcommand{\R}{{\mathds{R}}}
\newcommand{\C}{{\mathds{C}}}
\newcommand{\T}{{\mathds{T}}}

\newcommand{\D}{{\mathfrak{D}}}
\newcommand{\A}{{\mathfrak{A}}}
\newcommand{\B}{{\mathfrak{B}}}

\newcommand{\Lip}{{\mathsf{L}}}
\newcommand{\TLip}{{\mathsf{S}}}
\newcommand{\Hilbert}{{\mathscr{H}}}

\newcommand{\dpropinquity}[1]{{\mathsf{\Lambda}^\ast_{#1}}}

\newcommand{\dmodpropinquity}[1]{{\mathsf{\Lambda}^{\ast\mathsf{mod}}_{{#1},{#1}_{\mathsf{inner}}}}}
\newcommand{\dmetpropinquity}[1]{{\mathsf{\Lambda}^{\ast\mathsf{met}}_{{#1},{#1}_{\mathsf{inner}},{#1}_{\mathsf{mod}}}}}
\newcommand{\dcovmetpropinquity}[1]{{\mathsf{\Lambda}^{\ast\mathsf{met,cov}}_{{#1},{#1}_{\mathsf{inner}},{#1}_{\mathsf{mod}}}}}
\newcommand{\dcovmodpropinquity}[1]{{\mathsf{\Lambda}^{\ast\mathsf{mod,cov}}_{{#1},{#1}_{\mathsf{inner}}}}}
\newcommand{\spectralpropinquity}[1]{{\mathsf{\Lambda}^{\mathsf{spec}}_{{#1}, {#1}_{\mathsf{inner}},{#1}_{\mathsf{mod}} }}}
\newcommand{\lspectralpropinquity}{{\mathsf{\Lambda}^{\mathsf{spec}}}}

\newcommand{\Kantorovich}[1]{{\mathsf{mk}_{#1}}}
\newcommand{\KantorovichMod}[1]{{\mathsf{k}_{#1}}}

\newcommand{\KantorovichAlt}[1]{{\mathrm{mk^{\mathrm{alt}}_{#1}}}}

\newcommand{\Haus}[1]{{\mathsf{Haus}_{#1}}}

\newcommand{\StateSpace}{{\mathscr{S}}}
\newcommand{\ModStateSpace}{\widetilde{\mathscr{S}}}
\newcommand{\FalseDual}[1]{{\mathscr{D}\left({#1}\right)}}

\newcommand{\MongeKant}{{Mon\-ge-Kan\-to\-ro\-vich metric}}

\newcommand{\Lqcms}{{\JLL} quantum compact metric space}
\newcommand{\Qqcms}[1]{{$#1$}--\gQqcms}
\newcommand{\gQqcms}{Leibniz
  quantum compact metric space}

\newcommand{\gQVB}{metrized quantum vector bundle}
\newcommand{\QVB}[1]{$\left({#1},{#1}_{\mathsf{inner}}\right)$--\gQVB}
\newcommand{\LMVB}{Leibniz \gMVB}
\newcommand{\gMVB}{metrical
  $C^\ast$-correspondence}
\newcommand{\MVB}[1]{$\left({#1},{#1}_{\mathsf{inner}},{#1}_{\mathsf{mod}}\right)$--\gMVB}

\newcommand{\mvb}[3]{{\mathrm{mcc}\left({#1},{#2},{#3}\right)}}
\newcommand{\umvb}[3]{{\mathrm{umcc}\left({#1},{#2},{#3}\right)}}
\newcommand{\qcms}{quantum compact metric space}
\newcommand{\lcqms}{quantum locally compact metric space}

\newcommand{\unit}{1}

\newcommand{\sa}[1]{{\mathfrak{sa}\left({#1}\right)}}

\newcommand{\UIso}[4]{{\mathsf{UIso}_{#1}\left({#2}\rightarrow{#3}\middle\vert{#4}\right)}}

\newcommand{\inner}[3]{{\left<{#1},{#2}\right>_{#3}}}
\newcommand{\JLL}{Lei\-bniz}

\newcommand{\dom}[1]{{\operatorname*{dom}\left({#1}\right)}}
\newcommand{\codom}[1]{{\operatorname*{codom}\left({#1}\right)}}
\newcommand{\diam}[2]{{\mathrm{diam}\left({#1},{#2}\right)}}

\newcommand{\norm}[2]{\left\|{#1}\right\|_{#2}}
\newcommand{\tunnelset}[4]{{\text{\calligra
      Tunnels}\,\left[{#1}\xrightarrow{#3} {#2}\middle\vert
      {#4} \right]}}
\newcommand{\tunnelsetltd}[3]{{\text{\calligra
      Tunnels}\,\left[{#1}\stackrel{#3}{\longrightarrow}{#2} \right]}}

\newcommand{\targetsettunnel}[3]{{\mathfrak{t}_{#1}\left({#2}\middle\vert{#3}\right)}}

 \newcommand{\CDN}{{\mathsf{D}}}
\newcommand{\TDN}{{\mathsf{T}}}

\newcommand{\worknote}[1]{} 
\newcommand{\opnorm}[3]{{\left|\mkern-1.5mu\left|\mkern-1.5mu\left|
          {#1}
        \right|\mkern-1.5mu\right|\mkern-1.5mu\right|_{#3}^{#2}}}

\newcommand{\tunnelreach}[2]{{\rho\left({#1}\middle\vert{#2}\right)}}
\newcommand{\tunnelmodreach}[2]{{\rho_m\left({#1}\middle\vert{#2}\right)}}

\newcommand{\tunnelmagnitude}[2]{{\mu\left({#1}\middle\vert{#2}\right)}}
\newcommand{\tunnelmodmagnitude}[2]{{\mu_m\left({#1}\middle\vert{#2}\right)}}

\newcommand{\tunnelextent}[1]{{\chi\left({#1}\right)}}

\newcommand{\alg}[1]{{\mathfrak{#1}}}

\newcommand{\module}[1]{{\mathscr{#1}}}

\newcommand{\resolvent}[2]{\mathcal{R}\left({#1} ; {#2} \right)}

\newcommand{\CMS}[8]{ \left( \begin{array}{lllll}
                               #1 & #2 & #3 & #4 \\
                               #5 & #6 & #7 & #8
                             \end{array}\right)}
                        
\newcommand{\AdRep}[1]{{\mathrm{Ad}_{#1}}}

\theoremstyle{plain}
\newtheorem{theorem}{Theorem}[section]

\newtheorem{proposition}[theorem]{Proposition}

\newtheorem{theorem-definition}[theorem]{Theorem-Definition}

\theoremstyle{definition}
\newtheorem{definition}[theorem]{Definition}

\newtheorem{convention}[theorem]{Convention}

\theoremstyle{remark}

\newtheorem{remark}[theorem]{Remark}

\newtheorem{notation}[theorem]{Notation}

\renewcommand{\geq}{\geqslant}
\renewcommand{\leq}{\leqslant}

\numberwithin{equation}{section}

\allowdisplaybreaks[4]

\begin{document}

\title{The Gromov-Hausdorff propinquity for
  metric Spectral Triples}
\author{Fr\'{e}d\'{e}ric
  Latr\'{e}moli\`{e}re}
\email{frederic@math.du.edu}
\urladdr{http://www.math.du.edu/\symbol{126}frederic}
\address{Department of Mathematics \\
  University of Denver \\ Denver CO 80208}

\date{\today} 
\thanks{This work is part of the project
  supported by the grant
  H2020-MSCA-RISE-2015-691246-QUANTUM
  DYNAMICS and grant \#3542/H2020/2016/2 of
  the Polish Ministry of Science and Higher
  Education.}

\begin{abstract}
  We define a metric on the class of metric spectral triples, which is
  null exactly between unitarily equivalent spectral triples. This
  metric dominates the propinquity, and thus implies metric
  convergence of the {\qcms s} induced by metric spectral triples. In
  the process of our construction, we also introduce the covariant
  modular propinquity, as a key component for the definition of the
  spectral propinquity.
\end{abstract}
\maketitle



\section{Introduction}

The primary purpose of our research is to bring forth a new approach to problems from mathematical physics and noncommutative geometry by constructing an analytic framework around Gromov-Hausdorff-like hypertopologies on
classes of quantum spaces. The central themes of this
project have been the construction of noncommutative analogues of the
Gromov-Hausdorff distance \cite{Gromov81,Gromov} adapted to
C*-algebras \cite{Latremoliere05, Latremoliere13, Latremoliere13b,
  Latremoliere13c, Latremoliere14, Latremoliere15, Latremoliere15b,
  Latremoliere15c} and the initiation and advancement of a theory of
metric convergence for various structures over $C^\ast$-algebras, such
as modules \cite{Latremoliere16c,Latremoliere17a,Latremoliere18a} and
group actions
\cite{Latremoliere17c,Latremoliere18b,Latremoliere18c}. The present
work introduces a distance on the space of \emph{metric spectral
  triples}, strongly motivated by potential applications to
mathematical physics, such as the convergence of matrix models to some
limit \cite{Latremoliere21a}.

Our motivation for this project emerges from four connected
observations. First, a recurrent theme in mathematical physics is the
construction of quantum models as limits of some discrete, often even
finite models, when some metric on the spaces are involved. Second,
certain approaches to quantum cosmology and quantum gravity involve an
as-of-yet not fully understood geometry on the space of all
space-times \cite{Wheeler68}. Notably, the first occurrence and study
of the Gromov-Hausdorff distance was actually due to Edwards
\cite{Edwards75}, motivated by Wheeler's superspace approach to
quantum gravity. Third, a set of converging ideas in quantum physics
suggests the possibility that at the Planck scale, space-time may be
best described as a noncommutative space \cite{Doplicher95}, and
metric considerations have become a component of this research,
including many references to our work
\cite{Wallet12,Wallet15,dandrea13}. Fourth, remarkable new
developments in geometry arose from the use of the Gromov-Hausdorff
distance and the metric properties of manifolds and related spaces. We
thus aim at developing a theory which allow us to formalize physics
problems and problems from noncommutative geometry at the level of
hyperspaces of quantum metric spaces and spectral triples, so as to
apply to them new analytic techniques.

\emph{Spectral triples}, as introduced by Connes
\cite{Connes89,Connes} as early as 1985 in his lectures at the
Coll{\`e}ge de France, have emerged as the preferred means to
generalize Riemannian geometry to the noncommutative realm. Their
importance lies in their well-established power in generalizing, in
particular, spectral geometry to various new situations, from the
study of the spaces of leaves of foliations, to defining a geometry on quantum
tori, quantum spheres, and other quantum spaces, which, in turn, have
found applications in mathematical physics. Our perspective on
spectral triples provides a new direction for investigation, by
focusing on the metric aspects of noncommutative geometry, and
studying \emph{spaces of spectral triples}.

The importance of our work in this paper is to be found in the
applications it opens. Our present work puts a topology on the class
of all metric spectral triples. Therefore, it becomes possible to
address questions such as perturbations of metric within an analytical
framework --- quantifying the scale of perturbations, including the
effects of changes of underlying topologies, and studying topological
properties of classes of quantum spaces obtained from perturbations,
such as compactness
\cite{Latremoliere15c,Latremoliere15d,Latremoliere15}. We can also
discuss approximations of spectral triples by other spectral triples,
for instance spectral triples on matrix models approximating spectral
triples on infinite dimensional C*-algebras \cite{Latremoliere13c} ---
for instance, physically motivated models over fuzzy tori converging
to quantum torus \cite{Latremoliere21a}. We can also discuss time
evolution of quantum geometries, or any other dynamical process or
flows where both the quantum metric and the quantum topology are
allowed to change, all within a natural framework based on metric
space theory. While approximations of differential structures is
generally delicate and at times rigid, the flexibility offered by
by spectral triples and by introducing noncommutative spaces open new
possibilities for interesting research. Our project even opens new directions for research within classical metric geometry, such as in the study of fractals, as seen in
\cite{Latremoliere20a}.

\medskip

Connes' original introduction of spectral triples \cite{Connes89} was
actually instrumental in his introduction of compact quantum metric
spaces. Spectral triples are far-reaching abstractions of the Dirac
operator acting on the smooth sections of a vector bundle over a
Riemann spin manifold. There are varying definitions of spectral
triples in the literature, and for our purpose, we start with what
seems to be a good common core met by almost all definitions of which we are
aware.

\begin{definition}[\cite{Connes}]\label{spectral-triple-def}
  A \emph{spectral triple} $(\A,\Hilbert,D)$ consists of a unital
  C*-algebra $\A$, a Hilbert space $\Hilbert$, and a self-adjoint
  operator $D$ defined on a dense linear subspace $\dom{D}$ of
  $\Hilbert$, such that there exists a unital faithful *-representation of
  $\A$ on $\Hilbert$ (we will identify $\A$ with a C*-subalgebra of
  the algebra $\B(\Hilbert)$ of bounded linear operators on
  $\Hilbert$), and
  \begin{enumerate}
  \item $D+i$ has a compact inverse,
  \item the set of $a\in\A$ such that:
    \begin{equation*}
      a \cdot \dom{D} \subseteq \dom{D} 
    \end{equation*}
    and
    \begin{equation*}
      [D,a] \text{ is closeable, with bounded closure}
    \end{equation*}
    is dense in $\A$.
  \end{enumerate}
\end{definition}

Note that if $T$ is the inverse of $D+i$, then $T$ is compact if and
only if $T^\ast T$ is compact. Thus $D+i$ has compact inverse if and
only if $(1 + D^2)$ has a compact inverse.

\begin{remark}
  We follow the convention in the literature on spectral triples not
  to introduce a notation for the representation of the C*-algebra
  $\A$ on the Hilbert space $\Hilbert$ in a spectral triple
  $(\A,\Hilbert,D)$ --- this may at times require some care in reading
  some of our statements but it also is the standard adopted in the
  field.

  Moreover, whenever no confusion may arise, we will identify a bounded operator $T$ from $\dom{D}$ to $\Hilbert$ with its unique uniformly continuous extension to $\Hilbert$.
\end{remark}

Spectral triples induce an extended pseudo-metric on the state space
of their underlying C*-algebras, called the Connes metric. Of prime
interest in noncommutative geometry are the spectral triples whose
Connes' metric induces the weak* topology. There are several ways to
understand this focus, including the facts that, if the Connes metric
is indeed an extended metric, then the weak* topology is the weakest
topology it can induce, and it is the only compact one; moreover the
weak* topology is the natural topology on the state space from a
physical perspective. Spectral triples whose Connes' metric metrizes
the weak* topology will be called \emph{metric spectral triples}.

Now, as we shall see, this additional topological requirement on
spectral triples means that such metric spectral triples induce a
structure of \emph{\qcms s}. A {\qcms} is a noncommutative analogue of
the algebra of Lipschitz functions over a compact metric space, and is
the basic object of study of our project in noncommutative
geometry. The definition of {\qcms s} has evolved from Connes'
original proposition \cite{Connes89}, motivated by spectral triples,
to the current version we now state, owing mostly to Rieffel's
observation \cite{Rieffel98a,Rieffel99} that the {\MongeKant} on
quantum metric spaces should share a key topological property with the
original {\MongeKant} in the classical picture. Our contribution to
the following definition, from \cite{Latremoliere13,Latremoliere15},
is to impose a form of a Leibniz relation, as a key property for our
work on the propinquity, and a notion of {\lcqms}
\cite{Latremoliere12b}. We start by introducing the notion of a
{\qcms}.

\begin{notation}
  If $E$ is some normed vector space, then we denote its norm by
  $\norm{\cdot}{E}$ unless otherwise specified. For a C*-algebra $\A$,
  we write $\sa{\A}$ for the subspace of self-adjoint elements in
  $\A$, and $\StateSpace(\A)$ for the state space of $\A$. If $\A$ is
  unital, then its unit is denoted by $\unit_\A$.

  If $a\in\A$, with $\A$ a C*-algebra, then
  $\Re a = \frac{a+a^\ast}{2} \in \sa{\A}$ and
  $\Im a = \frac{a-a^\ast}{2i} \in \sa{\A}$.
\end{notation}

\begin{definition}
  We endow $[0,\infty)^4$ with the product order defined by setting, for all $(x_1,x_2,x_3,x_4)$,$(x'_1,x'_2,x'_3,x'_4)$ in $[0,\infty)^4$,
  \begin{equation*}
    (x_1,x_2,x_3,x_4) \leq (x'_1,x'_2,x'_3,x'_4) \iff \forall j\in\{1,2,3,4\} \quad x_j\leq x'_j \text.
  \end{equation*}
  
  A function $F : [0,\infty)^4\rightarrow[0,\infty)$ is
  \emph{permissible} when $F$ is weakly increasing from the product
  order on $[0,\infty)^4$ and, for all $x,y,l_x,l_y\geq 0$ we have
  $F(x,y,l_x,l_y) \geq x l_y + y l_x$.
\end{definition}

\begin{definition}[\cite{Connes89,Rieffel98a,Rieffel99,Rieffel10c,Latremoliere13,Latremoliere15}]\label{qcms-def}
  For a permissible function $F$, an \emph{\Qqcms{F}} $(\A,\Lip)$ is a
  unital C*-algebra $\A$ and a seminorm $\Lip$ defined on a dense
  Jordan-Lie subalgebra $\dom{\Lip}$ of $\sa{\A}$ such that:
  \begin{enumerate}
  \item
    $\left\{ a \in \dom{\Lip} : \Lip(a) = 0 \right\} = \R\unit_\A$,
  \item the {\MongeKant} $\Kantorovich{\Lip}$ defined between any two
    states $\varphi,\psi \in \StateSpace(\A)$ by:
    \begin{equation*}
      \Kantorovich{\Lip}(\varphi,\psi) = \sup\left\{ \left|\varphi(a) - \psi(a)\right| : a\in\dom{\Lip}, \Lip(a) \leq 1 \right\}
    \end{equation*}
    metrizes the weak* topology on $\StateSpace(\A)$,
  \item $\Lip$ is lower semi-continuous with respect to
    $\norm{\cdot}{\A}$, i.e. $\{ a\in\dom{\Lip} : \Lip(a) \leq 1\}$ is
    closed for $\norm{\cdot}{\A}$,
  \item if $a,b \in \dom{\Lip}$, then
    $\frac{ab+ba}{2},\frac{ab-ba}{2i} \in\dom{\Lip}$. and
    \begin{equation*}
      \Lip\left(\frac{a b + b a}{2}\right), \Lip\left(\frac{a b - b a}{2i}\right) \leq F(\norm{a}{\A},\norm{b}{\A},\Lip(a),\Lip(b))\text.
    \end{equation*}
  \end{enumerate}

  A {\Lqcms} $(\A,\Lip)$ is a {\Qqcms{F}} for
  $F : x,y,l_x,l_y \mapsto x l_y + y l_x$, i.e. for all
  $a,b \in \dom{\Lip}$, we have
  $\Lip\left(\Re(a b)\right), \Lip\left(\Im(a b)\right) \leq
  \norm{a}{\A}\Lip(b) + \norm{b}{\A} \Lip(a)$.
\end{definition}

We will use the following common condition, applied to L-seminorms and
other seminorms, as we did in
\cite{Latremoliere13,Latremoliere13b,Latremoliere14,Latremoliere16c,Latremoliere18d}.
\begin{convention}\label{seminorm-convention}
  If $E$ is a vector space, and if $L$ is a seminorm defined on a
  subspace $\dom{L}$ of $E$, then we set $L(x) = \infty$ for all
  $x\in E\setminus \dom{L}$. In particular,
  $\dom{L} = \{ x \in E : L(x) < \infty \}$. We use the usual
  conventions used in measure theory when dealing with $\infty$ here,
  i.e. $\infty+x=x+\infty=\infty$ for all $x\in[0,\infty]$,
  $x\infty=\infty$ for $x\in(0,\infty)$, $0\infty = 0$, and
  $x < \infty$ for all $x\in(0,\infty)$.
\end{convention}

Now, a metric spectral triple is formally defined as follows. Our
focus for this paper will be the geometry of the space of metric
spectral triples.

\begin{notation}
  We denote the norm of a linear map $T : E \rightarrow F$ between
  normed vector spaces $E$ and $F$ by $\opnorm{T}{E}{F}$, or simply
  $\opnorm{T}{}{F}$ if $E=F$. 
\end{notation}

\begin{definition}\label{metric-spectral-triple-def}
  A \emph{metric} spectral triple $(\A,\Hilbert,D)$ is a spectral
  triple such that, if we set:
  \begin{multline*}
    \forall \varphi,\psi  \in \StateSpace(\A) \quad \Kantorovich{D}(\varphi,\psi) = \sup\Big\{ \left|\varphi(a) - \psi(a)\right| : \\
    a\in\sa{\A}, a\cdot\dom{D}\subseteq\dom{D},
    \opnorm{[D,a]}{}{\Hilbert} \leq 1 \Big\}
  \end{multline*}
  then the metric $\Kantorovich{D}$ metrizes the weak* topology on the
  state space $\StateSpace(\A)$ of $\A$.
\end{definition}

Metric spectral triples do give rise to {\qcms s} in a natural
fashion, which was the original prescription of Connes \cite{Connes89}. To any
spectral triple, we can associate a seminorm which will be our L-seminorm canonically induced by a metric spectral triple. As the precise definition of {\qcms} has evolved, we include the full proof of the following proposition, and we note that some other propositions in the same vein can be found in \cite[Proposition 3.7]{Rieffel99}, \cite[Lemma 2.3,2.4]{Kaad18}, \cite{Christensen15}.

\begin{notation}\label{LipD-notation}
  If $(\A,\Hilbert,D)$ is a spectral triple, then we set
  \begin{equation*}
    \dom{\Lip_D} = \left\{ a \in \sa{\A} : a\cdot\dom{D}\subseteq\dom{D} \text{ and }[D,a]\text{ is bounded} \right\}\text,
  \end{equation*}
  and we denote by $\Lip_{D}$ the seminorm defined for all
  $a\in \dom{\Lip_D}$ by
  \begin{equation*}
    \Lip_{D}(a) = \opnorm{[D,a]}{}{\Hilbert}\text.
  \end{equation*}
  Note that with our Convention (\ref{seminorm-convention}),
  $\Lip_\D(a) = \infty$ whenever $a\in\sa{\A}\setminus\dom{\Lip_D}$.
\end{notation}

\begin{proposition}\label{spectral-metric-prop}
  Let $(\A,\Hilbert,D)$ be a spectral triple. The spectral triple
  $(\A,\Hilbert,D)$ is metric if and only if $(\A,\Lip_{D})$ is a
  {\Lqcms}.
\end{proposition}

\begin{proof}
  If $(\A,\Lip_D)$ is a {\Lqcms}, then by Definition
  (\ref{metric-spectral-triple-def}), the spectral triple
  $(\A,\Hilbert,D)$ is metric.

  Let us now assume that $(\A,\Hilbert,D)$ is a metric spectral
  triple. By Notation (\ref{LipD-notation}), the domain of $\Lip_D$ is:
  \begin{equation*}
    \left\{ a \in \sa{\A} : a\cdot\dom{D} \subseteq\dom{D} \text{ and }\opnorm{[D,a]}{}{\Hilbert}<\infty \right\}\text{.}
  \end{equation*}
  By Definition (\ref{spectral-triple-def}), the set:
  \begin{equation*}
    \mathscr{D} = \left\{ a \in \A : a\cdot\dom{D} \subseteq\dom{D} \text{ and }\opnorm{[D,a]}{}{\Hilbert}<\infty \right\}
  \end{equation*}
  is norm dense in $\A$. If $a\in\sa{\A}$, then there exists
  $(a_n)_{n\in\N}$ in $\mathscr{D}^\N$ converging to $a$ in norm. Now,
  we prove that if $b \in \mathscr{D}$ then $b^\ast \in \mathscr{D}$
  as well. Let $b\in\mathscr{D}$. if $\xi,\zeta\in \dom{D}$, then:
  \begin{align*}
    \inner{b^\ast\xi}{D\zeta}{\Hilbert} 
    &= \inner{\xi}{bD\zeta}{\Hilbert} \\
    &= \inner{\xi}{Db\zeta}{\Hilbert} - \inner{\xi}{[D,b]\zeta}{\Hilbert}\\
    &= \inner{D\xi}{b\zeta}{\Hilbert} - \inner{\xi}{[D,b]\zeta}{\Hilbert} \text{.}
  \end{align*} 
  Now, since $\xi\in\dom{D}$, the linear map
  $\zeta\in\dom{D}\mapsto \inner{D\xi}{b\zeta}{\Hilbert}$ is continuous, and
  since $[D,b]$ is bounded, the linear map
  $\zeta\in\dom{D} \mapsto \inner{\xi}{[D,b]\zeta}{\Hilbert}$ is also
  continuous. Hence
  $\zeta\in\Hilbert\mapsto \inner{b^\ast\xi}{D\zeta}{\Hilbert}$ is
  continuous, and thus $b^\ast \xi \in \dom{D^\ast}=\dom{D}$. Now, on
  $\dom{D}$, we observe that
  $[D,b^\ast] = Db^\ast - b^\ast D= (b D - D b)^\ast = (-[D,b])^\ast$
  as $D$ is self-adjoint, so $b^\ast \in \mathscr{D}$.

  It is immediate to check that $\mathscr{D}$ is a linear space, and
  thus in particular, for all $n\in\N$, we have
  $\Re a_n = \frac{a_n + a_n^\ast}{2} \in \dom{\Lip_D}$, and of course
  as $a\in\sa{\A}$, we have by continuity of $\Re$ that
  $a = \Re a = \lim_{n\rightarrow\infty}\Re a_n$, thus proving that
  $\dom{\Lip_D}$ is dense in $\sa{\A}$.

  By Definition (\ref{metric-spectral-triple-def}), the {\MongeKant}
  $\Kantorovich{\Lip_D}$ metrizes the weak* topology. In particular,
  as a metric , it is finite between any two states of $\A$. Let
  $a\in\sa{\A}$ with $\Lip_D(a) = 0$. Let
  $\varphi,\psi \in \StateSpace(\A)$. We have, by Definition
  (\ref{qcms-def}):
  \begin{equation*}
    0 \leq \left| \varphi(a) - \psi(a) \right| \leq \Lip_D(a) \Kantorovich{\Lip_D}(\varphi,\psi) = 0
  \end{equation*}
  and thus $\varphi(a - \psi(a)\unit_\A) = 0$ for all
  $\varphi,\psi \in \StateSpace(\A)$. Thus (as $a\in\sa{\A}$), if we
  fix $\psi \in \StateSpace(\A)$:
  \begin{equation*}
    \norm{a-\psi(a)\unit_\A}{\A} = \sup_{\varphi\in\StateSpace(\A)}|\varphi(a-\psi(a)\unit_\A)| = 0
  \end{equation*}
  so $a = \psi(a)\unit_\A$, i.e.
  $\{a\in\sa{\A} : \Lip_D(a) = 0 \} \subseteq \R\unit_\A$. On the
  other hand, $\Lip_D(\unit_\A) = 0$ by construction, so
  $\{a\in\sa{\A}:\Lip_D(a) = 0\} = \R\unit_\A$, as desired.

  We now check that $\Lip_D$ is lower semicontinuous. Let
  $(a_n)_{n\in\N}$ be a sequence in $\dom{\Lip_D}$ with
  $\Lip_D(a_n) \leq 1$ converging in norm to $a\in\sa{\A}$. Let
  $\xi\in\dom{D}$ and let $\zeta\in\dom{D}$. For any $n\in\N$:
  \begin{align*}
    \inner{a_n\xi}{D\zeta}{\Hilbert}
    &= \inner{\xi}{a_nD\zeta}{\Hilbert} \\
    &= \inner{\xi}{D a_n  \zeta}{\Hilbert} - \inner{\xi}{[D, a_n]\zeta}{\Hilbert} \\
    &= \inner{D\xi}{a_n\zeta}{\Hilbert} - \inner{\xi}{[D,a_n]\zeta}{\Hilbert} \\
  \end{align*}
  and therefore:
  \begin{align*}
    \left|\inner{a\xi}{D\zeta}{\Hilbert}\right| 
    &= \lim_{n\rightarrow\infty}\left|\inner{a_n\xi}{D\zeta}{\Hilbert}\right|\\
    &\leq \limsup_{n\rightarrow\infty} \left( \left| \inner{D\xi}{a_n\zeta}{\Hilbert}\right| + \left|\inner{\xi}{[D,a_n]\zeta}{\Hilbert}\right|\right)\\
    &\leq \left|\inner{D\xi}{a\zeta}{\Hilbert}\right| + \norm{\xi}{\Hilbert}\norm{\zeta}{\Hilbert} \\
    &\leq \norm{\zeta}{\Hilbert}\left( \norm{D\xi}{\Hilbert}\norm{a}{\A} + \norm{\xi}{\Hilbert}  \right)\text{.}
  \end{align*} 
  So the function $\zeta\in\dom{D} \mapsto \inner{a\xi}{D\zeta}{\Hilbert}$ is
  continuous, and thus $a\xi\in\dom{D}$. Thus
  $a\cdot\dom{D} \subseteq\dom{D}$ as $\xi \in \dom{D}$ was
  arbitrary. We can therefore apply \cite[Proposition 3.7]{Rieffel99},
  whose argument we now briefly recall. If $\xi,\zeta\in\dom{D}$ with
  $\norm{\xi}{\Hilbert}\leq 1$ and $\norm{\zeta}{\Hilbert}\leq 1$,
  then:
  \begin{multline}\label{metric-spectral-eq-1}
      1\geq \left|\inner{[D,a_n]\xi}{\zeta}{\Hilbert} \right|
      = \left|\inner{a_n\xi}{D\zeta}{\Hilbert} - \inner{D\xi}{a_n\zeta}{\Hilbert} \right| \\
      \xrightarrow{n\rightarrow\infty} \left| \inner{a\xi}{D\zeta}{\Hilbert} - \inner{D\xi}{a\zeta}{\Hilbert}\right| = \left|\inner{[D,a]\xi}{\zeta}{\Hilbert}\right|\text{.}
    \end{multline}
  Since $\dom{D}$ is dense in $\Hilbert$ and since, by Expression (\ref{metric-spectral-eq-1}), for all $\xi,\zeta\in\dom{D}$, we have proven that $\left|\inner{[D,a]\xi}{\zeta}{\Hilbert}\right| \leq \norm{\xi}{\Hilbert}\norm{\zeta}{\Hilbert}$, we conclude that $[D,a]$ is bounded with norm $1$ on $\dom{D}$, and thus extends to a bounded operator of norm at most $1$ on $\Hilbert$.
  
  Thus $\{a\in\sa{\A} : \Lip_D(a) \leq 1\}$ is indeed normed closed. As
  $\Lip_D$ is a seminorm, this implies that it is lower semi-continuous with respect to $\norm{\cdot}{\A}$.

  Last, $\Lip_D$ satisfies the Leibniz inequality since it is the norm of a
  derivation. First, we note that $\mathscr{D}$ is indeed an
  algebra. If $a,b \in \mathscr{D}$ then, first, since
  $b\cdot\dom{D}\subseteq\dom{D}$, we also have
  $ab\cdot\dom{D} \subseteq a\cdot\dom{D} \subseteq\dom{D}$. Moreover,
  if $\xi,\zeta \in \dom{D}$, then:
  \begin{align*}
    \inner{Dab\xi - abD\xi}{\zeta}{\Hilbert} 
    &= \inner{D ab \xi - aDb\xi}{\zeta}{\Hilbert} + \inner{aDb\xi - abD\xi}{\zeta}{\Hilbert} \\
    &= \inner{[D,a] b\xi}{\zeta}{\Hilbert} + \inner{a[D,b]\xi}{\zeta}{\Hilbert}
  \end{align*}
  and thus, as operators on $\dom{D}$, we conclude
  $[D,ab]=a[D,b] + [D,a]b$. Therefore, for all $a,b \in \dom{\Lip_D}$:
  \begin{align*}
    \opnorm{[D,ab]}{}{\Hilbert}
    &= \opnorm{[D,a]b + a[D,b]}{}{\Hilbert} \\
    &\leq \opnorm{[D,a]}{}{\Hilbert} \norm{b}{\A} + \norm{a}{\A}\opnorm{[D,b]}{}{\Hilbert} \\
    &= \Lip_D(a) \norm{b}{\A} + \norm{a}{\A}\Lip_D(b) \text{.}
  \end{align*}
  Therefore, we conclude, for all $a,b \in \dom{\Lip_D}$:
  \begin{align*}
    \Lip\left(\frac{ab+ba}{2}\right)
    &= \opnorm{\left[D,\frac{ab+ba}{2}\right]}{}{\Hilbert} \\
    &\leq \frac{1}{2}\left( \opnorm{[D,ab]}{}{\Hilbert} + \opnorm{[D,ba]}{}{\Hilbert}  \right) \\
    &\leq \Lip_D(a) \norm{b}{\A} + \norm{a}{\A} \Lip_D(b) \text.
  \end{align*}
  A similar argument shows that $\Lip_D\left(\frac{ab-ba}{2i}\right) \leq \Lip_D(a)\norm{b}{\A} + \norm{a}{\A} \Lip_D(b)$. It follows that $(\A,\Lip_D)$ is a {\qcms}.
\end{proof}

For our construction to be coherent and move toward our project of
applying the theory of the propinquity to metric spectral triples, it
is very important that the basic notion of two metric spectral triples
and two {\qcms s} being ``the same'', i.e. isomorphic, are
compatible. We propose the following strong notion of equivalence for
spectral triples.

\begin{definition}\label{equiv-spectral-triple-def}
  Two spectral triples $(\A,\Hilbert_\A,D_\A)$ and
  $(\B,\Hilbert_\B,D_\B)$ are \emph{equivalent} when there exists a
  unitary $U$ from $\Hilbert_\A$ to $\Hilbert_\B$ and a *-isomorphism
  $\theta : \A \rightarrow \B$, such that
  \begin{equation*}
    U\dom{D_\A} = \dom{D_\B}\text{ and }    D_\B = U D_\A U^\ast \text{over $\dom{D_\B}$,}
  \end{equation*}
  and
  \begin{equation*}
    \forall \omega \in \Hilbert_\B, a \in \A \quad \theta(a)\omega = (U a U^\ast)\omega \text{.}
  \end{equation*}
\end{definition}

We remark, using the notation of Definition
(\ref{equiv-spectral-triple-def}) that $U^\ast\dom{D_\B}=\dom{D_\A}$:
indeed,
\begin{equation*}
  U^\ast \dom{D_\B} = U^\ast \left(U\dom{D_\A}\right) = \dom{D_\A}
\end{equation*}
since $U$ is a bijection with $U^{-1}=U^\ast$. It then follows
immediately that $U^\ast D_\B U = D_\A$ as operators on $\dom{D_\A}$.

Equivalence, thus defined, is indeed an equivalence relation on the
class of spectral triples. Moreover, equivalence of spectral triples
preserves the typical constructions based on spectral triples in the
literature \cite{Connes}.

On the other hand, there is a natural notion of isomorphism for
{\qcms}, called full quantum isometries
\cite{Rieffel00,Latremoliere13,Latremoliere13b}. To motivate the following definition, note that if $j : (X,d_X) \rightarrow (Y,d_Y)$ is an isometry between two compact metric spaces $(X,d_X)$ and $(Y,d_Y)$, then $f\in C(Y) \mapsto f\circ j \in C(X)$ is a surjective *-morphism (since $j$ is continuous and injective; also note the reversing of the arrow) such that, by McShane's extension theorem \cite{McShane34}, for any Lipschitz function $g \in \sa{C(X)}$, there exists a Lipschitz function $h \in \sa{C(Y)}$, with the same Lipschitz constant as $g$, such that $g = h\circ j$ --- of course, if $g = k \circ j$ for $k \in \sa{\C(Y)}$ then, as $j$ is an isometry, the Lipschitz constant of $k$ (which is possibly infinite) is at least the Lipschitz constant of $g$. We are thus led to the following definition.

\begin{definition}\label{quantum-iso-def}
  Let $(\A,\Lip_\A)$ and $(\B,\Lip_\B)$ be two {\qcms s}. A
  \emph{quantum isometry}
  $\pi : (\A,\Lip_\A) \twoheadrightarrow (\B,\Lip_\B)$ is a
  *-epimorphism $\pi : \A\twoheadrightarrow\B$ such that for all
  $b \in \sa{\B}$:
  \begin{equation}\label{quantum-iso-eq}
    \Lip_\B(b) = \inf\left\{ \Lip_\A(a) : a\in\dom{\Lip_\A}, \pi(a) = b  \right\}\text{.}
  \end{equation}
  A \emph{full quantum isometry}
  $\pi : (\A,\Lip_\A) \rightarrow (\B,\Lip_\B)$ is a *-isomorphism
  $\pi : \A\rightarrow\B$ such that $\Lip_\B\circ\pi = \Lip_\A$.
\end{definition}

Rieffel proved in \cite{Rieffel00} that quantum isometries can be
chosen as morphisms of a category over the {\qcms s}, and full quantum
isometries are indeed the morphisms whose inverse is also a morphism
in this category. We also note that, using the notation of Definition
(\ref{quantum-iso-def}), with
$\pi:(\A,\Lip_\A)\rightarrow(\B,\Lip_\B)$ a quantum isometry, if
$b\in\dom{\Lip_\B}$, then there exists $a\in\dom{\Lip_\A}$ such that
$\pi(a)=b$, so $\dom{\Lip_\B}\subseteq\pi(\dom{\Lip_\A})$. Of course,
if $b\in\pi(\dom{\Lip_\A})$, so that there exists $a\in\dom{\Lip_\A}$
such that $\pi(a)=b$, then Definition (\ref{quantum-iso-def}) implies
that $\Lip_\B(b) \leq \Lip_\A(a) < \infty$ and thus,
$b\in\dom{\Lip_\B}$. So, for any quantum isometry,
$\pi(\dom{\Lip_\A})=\dom{\Lip_\B}$. We could replace Equation
(\ref{quantum-iso-eq}) with the conditions that
$\pi(\dom{\Lip_\A})\subseteq\dom{\Lip_\B}$ and
\begin{equation}\label{quantum-iso-eq-2}
  \forall b \in \dom{\Lip_\B} \quad \Lip_\B(b) = \inf\left\{\Lip_\A(a):a\in\dom{\Lip_\A},\pi(a)=b \right\}
\end{equation}
since in that case, if $b\notin\dom{\Lip_\B}$, then
$\Lip_\B(b) = \infty$ and $b\notin\pi(\dom{\Lip_\A})$, and thus
$\{\Lip_\A(a):a\in\dom{\Lip_\A},\pi(a)=b\}$ is empty (and by
convention, has infinite infimum), so Equation (\ref{quantum-iso-eq})
holds as stated. Last, replacing $a\in\dom{\Lip_\A}$ by $a\in\sa{\A}$
in Equation (\ref{quantum-iso-eq}) or Equation
(\ref{quantum-iso-eq-2}) does not change anything, since
$\Lip_\A(a) = \infty$ whenever $a\in\sa{\A}\setminus\dom{\Lip_\A}$. We
will use these observations whenever convenient.

Now, if $\pi : (\A,\Lip_\A)\rightarrow(\B,\Lip_\B)$ is a full quantum
isometry between two {\qcms s} $(\A,\Lip_\A)$ and $(\B,\Lip_\B)$, then
we first note that if $a\in\dom{\Lip_\A}$, then
$\Lip_\B(\pi_\A(a)) = \Lip_\A(a) < \infty$, so
$\pi_\A(a) \in \dom{\Lip_\B}$. Similarly, if $b \in\dom{\Lip_\B}$, and
if $a=\pi^{-1}(b)$, then
$\Lip_\A(a) = \Lip_\B(\pi(a)) = \Lip_\B(\pi(\pi^{-1}(b))) = \Lip_\B(b)
< \infty$ and thus $a\in\dom{\Lip_\A}$. So
$\pi(\dom{\Lip_\A}) = \dom{\Lip_\B}$. Moreover, it is also immediate
that $\Lip_\A\circ\pi^{-1} = \Lip_\B\circ\pi\circ\pi^{-1} = \Lip_\B$,
so $\pi^{-1}$ is also a full quantum isometry. Last, for all
$b \in \dom{\Lip_\B}$, we have
$\Lip_\B(b) = \Lip_\A\circ\pi^{-1}(b) = \inf \Lip_\A(\pi^{-1}(\{b\}))$
since $\pi^{-1}(\{b\})=\{\pi^{-1}(b)\}$, as $\pi$ is a
bijection. Thus, full quantum isometries are, indeed, quantum
isometries, and so are their inverse.

There is a more general notion of Lipschitz morphisms between {\qcms
  s} \cite{Latremoliere16b} which will be important for us later on:
given two {\qcms s} $(\A,\Lip_\A)$ and $(\B,\Lip_\B)$, a *-morphism
$\pi:\A\rightarrow\B$ is a \emph{Lipschitz morphism} from
$(\A,\Lip_\A)$ to $(\B,\Lip_\B)$ when we require that
$\pi(\dom{\Lip_\A})\subseteq\dom{\Lip_\B}$ \emph{without} requiring
Equation (\ref{quantum-iso-eq-2}).

\medskip

We now check that equivalent metric spectral triples naturally give
rise to fully quantum isometric quantum metric spaces.

\begin{proposition}\label{equivalence-prop}
  If $(\A,\Hilbert_\A,D_\A)$ and $(\B,\Hilbert_\B,D_\B)$ are two
  equivalent metric spectral triples, then $(\A,\Lip_{D_\A})$ and
  $(\B,\Lip_{D_\B})$ are fully quantum isometric.
\end{proposition}

\begin{notation}
  If $T$ is an invertible operator on a Hilbert space $\Hilbert$, then
  $\AdRep{T}(A) = T A T^{-1}$ for all operators $A$ (bounded or not, up
  to adjusting the domain).
\end{notation}

\begin{proof}
  Let $U : \Hilbert_\A \rightarrow \Hilbert_\B$ be unitary and
  $\theta : (\A,\Lip_\A) \rightarrow (\B,\Lip_\B)$ be a *-isomorphism
  such that $\AdRep{U}D_\A = D_\B$ (including the fact that
  $U\dom{D_\A}=\dom{D_\B}$), and $U a U^\ast = \theta(a)$ for all
  $a\in\A$. If $a\in \dom{\Lip_{D_\A}}$ then
  $a\cdot\dom{D_\A}\subseteq\dom{D_\A}$, and $[D_\A,a]$ is bounded. Now, if
  $\xi \in \dom{D_\B}$, then $U^\ast\xi \in \dom{D_\A}$, and
  therefore, $U a U^\ast\xi \in \dom{D_\B}$. Moreover:
  \begin{align*}
    \Lip_{D_\A}(a) 
    &= \opnorm{[D_\A,a]}{}{\Hilbert_\A} \\
    &= \opnorm{[D_\A,a]}{}{\dom{D_\A}} \\
    &= \opnorm{[U^\ast D_\B U, a]}{}{\dom{D_A}} \\
    &= \opnorm{U^\ast D_\B U a - a U^\ast D_\B U}{}{\dom{D_\A}} \\
    &= \opnorm{U^\ast\left( D_\B U a U^\ast - U a U^\ast D_\A \right) U}{}{\dom{D_A}} \\
    &= \opnorm{D_\B \theta(a) - \theta(a) D_\B}{}{\dom{D_\B}} \\
    &= \opnorm{[D_\B,\theta(a)]}{}{\Hilbert_\B} \\
    &= \Lip_{D_\B}\circ\theta(a) \text{.}
  \end{align*}
  Thus $\theta(a) \in \dom{\Lip_{D_\B}}$ and
  $\Lip_{D_\B}\circ\theta(a) = \Lip_{D_\A}(a)$. In particular,
  $\theta(\dom{\Lip_{D_\A}}) \subseteq\dom{\Lip_{D_\B}}$.
  
  By symmetry, if $b\in\dom{\Lip_{D_\B}}$, then
  $\theta^{-1}(b) \in \Lip_{D_\A}$ with
  $\Lip_{D_\A}\circ\theta^{-1}(b) = \Lip_{D_\B}(b)$.

  If $a\notin\dom{\Lip_{D_\A}}$, yet
  $\theta(a) \in \dom{\Lip_{D_\B}}$, then we would have, by the
  observation above, that
  $a = \theta^{-1}(\theta(a)) \in \dom{\Lip_{D_\A}}$, an obvious
  contradiction. So
  $\theta(\sa{\A}\setminus\dom{\Lip_\A})\subseteq
  \sa{\B}\setminus\dom{\Lip_{D_\B}}$. Therefore,
  $\theta(\dom{\Lip_{D_\A}}) = \dom{\Lip_{D_\B}}$.
  
  Thus $\theta$ is a full quantum isometry from $(\A,\Lip_{D_\A})$ to
  $(\B,\Lip_{D_\B})$.
\end{proof}

It is nontrivial to determine whether or not two fully quantum isometric quantum metric spaces arising from metric spectral triples must come from metric spectral triples that are equivalent. This matter will be one of the points we address in this work.

\medskip

Our main contribution to noncommutative metric geometry is the
discovery and study of the Gromov-Hausdorff propinquity, a family of
metrics on the class of {\Qqcms{F}s}, for any permissible function
$F$, which are analogues of the Gromov-Hausdorff distance
\cite{Latremoliere13,Latremoliere13b,Latremoliere14,Latremoliere15,Latremoliere15b}. The
distance between spectral metric triples, introduced in this paper, is
constructed from the propinquity. We now summarize the construction of
the propinquity, starting with the notion of a tunnel between a pair
of {\qcms s}.

\begin{definition}\label{tunnel-def}
  Let $F$ be a permissible function, and let $(\A_1,\Lip_1)$ and
  $(\A_2,\Lip_2)$ be two {\Qqcms{F}s}. An \emph{$F$-tunnel}
  $\tau = (\D,\Lip,\pi_1,\pi_2)$ from $(\A_1,\Lip_1)$ to
  $(\A_2,\Lip_2)$ is a {\Qqcms{F}} $(\D,\Lip)$ and two quantum
  isometries $\pi_1 : (\D,\Lip)\twoheadrightarrow (\A_1,\Lip_1)$ and
  $\pi_2 : (\D,\Lip)\twoheadrightarrow(\A_2,\Lip_2)$. The
  \emph{domain} $\dom{\tau}$ of $\tau$ is $(\A_1,\Lip_1)$ while the
  \emph{codomain} $\codom{\tau}$ of $\tau$ is $(\A_2,\Lip_2)$.
\end{definition}
In particular, tunnels give rise to isometric embeddings of the state
spaces, though the isometries are of a very special kind, as dual maps
to *-epimorphisms, as illustrated in Figure (\ref{tunnel-fig}). Fixing
a permissible function $F$ and two {\Qqcms{F}s} $(\A,\Lip_\A)$ and
$(\B,\Lip_\B)$, the set of all $F$-tunnels from $(\A,\Lip_\A)$ to
$(\B,\Lip_\B)$ is denoted by:
\begin{equation*}
  \tunnelsetltd{(\A,\Lip_\A)}{(\B,\Lip_\B)}{F} \text{.}
\end{equation*}
We note that the set of $F$-tunnels between any two $F$-Leibniz {\qcms
  s} is never empty.

\begin{figure}[t]\label{tunnel-fig}
  \begin{equation*}
    \xymatrix{
      & (\StateSpace(\D),\Kantorovich{\Lip_\D})  & \\
      & (\D,\Lip_\D) \ar@{>>}[ldd]^{\pi_\A} \ar@{>>}[rdd]_{\pi_\B} \ar@{-->}[u] & \\
      (\StateSpace(\A),\Kantorovich{\Lip_\A}) \ar@/^/@{^{(}->}[ruu]_{\pi_\A^\ast} & & (\StateSpace(\B),\Kantorovich{\Lip_\B}) \ar@/_/@{_{(}->}[luu]^{\pi_\B^\ast} \\
      (\A,\Lip_\A) \ar@{-->}[u] & & (\B,\Lip_\B) \ar@{-->}[u]
    }
  \end{equation*}
  \caption{A tunnel and the dual isometric embeddings of state spaces}
  \begin{tabular}{lp{4cm}}
    $\hookrightarrow$ & isometry \\
    $\twoheadrightarrow$ & quantum isometry \\
    dotted arrows & duality relations \\
    $\pi^\ast : \varphi \mapsto \varphi\circ\pi$ & dual map\\
    $(\A,\Lip_\A)$, $(\B,\Lip_\B)$, $(\D,\Lip_\D)$ & {\Qqcms{F}s}
  \end{tabular}
\end{figure}
There is a natural quantity associated with any tunnels which, in
essence, measures how far apart the domain and codomain of a tunnel
are for this particular choice of embedding.
\begin{notation}
  If $(X,d)$ is a metric space, then the Hausdorff distance
  \cite{Hausdorff} on the class of all bounded, closed subsets of $(X,d)$ is
  denoted by $\Haus{d}$. If $X$ is a vector space and $d$ is induced
  by a norm $\norm{\cdot}{X}$, then $\Haus{d}$ is also denoted
  $\Haus{\norm{\cdot}{X}}$.
\end{notation}

\begin{definition}
  Let $(\A_1,\Lip_1)$ and $(\A_2,\Lip_2)$ be two {\qcms s}. The
  \emph{extent} $\tunnelextent{\tau}$ of a tunnel
  $\tau = (\D,\Lip,\pi_1,\pi_2)$ from $(\A_1,\Lip_1)$ to
  $(\A_2,\Lip_2)$ is the nonnegative number:
  \begin{equation*}
    \tunnelextent{\tau} = \max_{j \in \{1,2\}} \Haus{\Kantorovich{\Lip}}\left(\left\{ \varphi\circ\pi_j : \varphi \in \StateSpace(\A_j) \right\}, \StateSpace(\D)\right)\text{.}
  \end{equation*}
\end{definition}
We note that the extent of a tunnel is always finite. The propinquity
is then defined as follows:
\begin{definition}
  Let $F$ be a permissible function. For any two {\Qqcms{F}s}
  $(\A,\Lip_\A)$ and $(\B,\Lip_\B)$, the \emph{dual Gromov-Hausdorff
    $F$-propinquity} $\dpropinquity{F}((\A,\Lip_\A),(\B,\Lip_\B))$ is
  the nonnegative number:
  \begin{equation*}
    \dpropinquity{F}\left((\A,\Lip_\A),(\B,\Lip_\B)\right) = \inf\left\{ \tunnelextent{\tau} : \tau \in \tunnelsetltd{(\A,\Lip_\A)}{(\B,\Lip_\B)}{F} \right\}\text{.}
  \end{equation*}
\end{definition}

The propinquity enjoys the properties which a noncommutative analogue
of the Gromov-Hausdorff distance ought to possess, as seen in the
following theorem.
\begin{convention}
  Let $\sim$ be an equivalence relation on a class $C$. We call a
  pseudo-metric $d : C\times C \rightarrow [0,\infty)$ a \emph{metric,
    up to $\sim$}, when
  \begin{equation*}
    \forall x,y \in C \quad x\sim y \iff d(x,y) = 0 \text.
  \end{equation*}
\end{convention}

\begin{theorem}\label{prop-thm}
  Let $F$ be a continuous permissible function. The $F$-propinquity
  $\dpropinquity{F}$ is a complete metric up to full quantum isometry
  on the class of {\Qqcms{F}s}. Moreover, the class map which
  associates, to any compact metric space $(X,d)$, its canonical
  {\Lqcms} $(C(X),\Lip_d)$, where $C(X)$ is the C*-algebra of
  continuous, $\C$-valued functions over $X$, and
  \begin{equation*}
    \forall f \in C(X) \quad \Lip_d(f) = \sup\left\{\frac{|f(x)-f(y)|}{d(x,y)}:x,y\in X,x\neq y\right\} \in [0,\infty]\text,
  \end{equation*}
  is an homeomorphism onto its range, when its domain is endowed with
  the Gromov-Hausdorff distance topology and its codomain is endowed
  with the topology induced by the dual propinquity.
\end{theorem}

\begin{remark}
  The additional assumption that $F$ should be continuous in Theorem (\ref{prop-thm}) is only used in the proof of the completeness of the dual propinquity; without it, all other properties listed in Theorem (\ref{prop-thm}) still hold.
\end{remark}

Examples of interesting convergences for the propinquity include fuzzy
tori approximations of quantum tori \cite{Latremoliere13c}, continuity
for certain perturbations of quantum tori \cite{Latremoliere15c},
unital AF algebras with faithful tracial states
\cite{Latremoliere15d}, continuity for noncommutative solenoids
\cite{Latremoliere16}, and Rieffel's work on approximations of spheres
by full matrix algebras \cite{Rieffel15}, among other
examples. Moreover, we prove in \cite{Latremoliere15b} an analogue of
Gromov's compactness theorem. The canonical image of the class of
compact metric spaces is closed and actually nowhere dense for the
propinquity \cite{Latremoliere17b} (the space of classical compact metric spaces is known to be path connected for the Gromov-Hausdorff distance \cite{Ivanov16}, hence also for the propinquity).

We may put restrictions on the class of tunnels under consideration,
so we can adapt the construction of the propinquity to smaller classes
of {\qcms s} with additional properties. In many applications, tunnels
are built from a structure called a bridges \cite{Latremoliere13}.

We prove in this paper that we can construct a distance on the class
of metric spectral triples based upon our construction of the
propinquity.  Our metric, which we will call the spectral propinquity,
will be zero exactly between equivalent spectral triples, and it will
be stronger than the propinquity. To reach our goal, we make the
following observations. First, metric spectral triples give rise, in a
completely natural manner, to {\gMVB s}, whose definition we recall
below. This is an important proof-of-concept for our work on the
modular propinquity \cite{Latremoliere16c,Latremoliere18d}, which
extends the construction of a metric between {\qcms s} to a class of
modules over {\qcms s}.

Secondly, we want to encode more than the metric property for metric
spectral triples. Our project has given us the idea on how to proceed
from there. As is well-known, spectral triples give rise to natural
actions of $\R$ by unitaries on the underlying Hilbert space of the
spectral triple. The propinquity is well-behaved with respect to
group, or even monoid actions. In fact, we have defined a covariant
version of the propinquity. In this paper, we introduce the covariant
version of the modular propinquity in the same spirit as
\cite{Latremoliere17c,Latremoliere18b,Latremoliere18c}. This is a
contribution to our project on its own, so we develop it in its full
generality --- other applications of the construction found in this
paper could be, for instance to the study of the geometry of certain
spaces of actions on modules, such as the class of the actions of the
Heisenberg group actions on Heisenberg modules over quantum tori
\cite{Connes80,Rieffel93,Latremoliere16c,Latremoliere17a,Latremoliere18a}. Now,
applying the covariant modular propinquity to the {\gMVB s} defined by
metric spectral triples and their canonical unitary actions of $\R$ is
our spectral propinquity.

\section{D-norms from Metric Spectral Triples}

Proposition (\ref{spectral-metric-prop}) shows that metric spectral
triples give rise to {\qcms s}. We now see that in fact, these triples
give rise to more structure: they define {\gMVB s}, i.e. a particular
type of module structure over {\qcms s}. The importance of this
observation is that we have constructed a complete metric on {\gMVB s}
--- the metrical propinquity (up to a small change in convention which
we will explain below). Thus, we immediately have a pseudo-metric on
metric spectral triples. We recall from {\cite[Definition
  2.12]{Latremoliere18d}} the following notion, with a small change
explained in a following remark.

\begin{definition}
  Let $\A$ and $\B$ be two unital C*-algebras. An $\A$-$\B$
  C*-correspondence $\module{M}$ is a right Hilbert $\B$-module (whose
  $\B$-valued inner product is denoted by
  $\inner{\cdot}{\cdot}{\module{M}}$), together with a unital
  *-morphism from $\A$ to the C*-algebra of adjoinable $\B$-linear
  operators on $\module{M}$.

  We will not introduce any notation for the *-morphism from $\A$ to
  adjoinable $\B$-linear operators on $\module{B}$, and simply use the
  left module notation instead.
\end{definition}

Metric C*-correspondences are C*-correspondences over {\qcms s}, and endowed with a norm whose properties are inspired by the noncommutative theory of connections \cite{Mesland09,Kaad13,Kaad17}, as explained in \cite{Latremoliere16c}.

\begin{definition}\label{metrical-bundle-def}
  A \emph{metrical C*-correspondence}
  \begin{equation*}
    \left( \module{M}, \CDN, \A, \Lip_\A, \B, \Lip_\B \right)
  \end{equation*}
  is given by the following:
  \begin{enumerate}
  \item $(\A,\Lip_\A)$ and $(\B,\Lip_\B)$ are {\Qqcms{F}s},
  \item $\module{M}$ is a $\A$-$\B$ C*-correspondence,
  \item $\CDN$ is a norm defined on a dense $\A$-left submodule
    $\dom{\CDN}$ of $\module{M}$ such that:
    \begin{enumerate}
    \item for all $\omega\in\dom{\CDN}$ we have
      $\norm{\omega}{\module{M}} \leq \CDN(\omega)$,
    \item the set
      $\left\{ \omega \in \module{M} : \CDN(\omega)\leq 1 \right\}$ is
      compact for $\norm{\cdot}{\module{M}}$,
    \item for all $\omega,\eta \in \dom{\CDN}$, if
      $b = \inner{\omega}{\eta}{\module{M}}$, then
      \begin{equation*}
        \max\left\{ \Lip_\B\left(\frac{b+b^\ast}{2}\right), \Lip_\B\left(\frac{b-b^\ast}{2 i}\right) \right\} \leq F_{\mathsf{inner}}(\CDN(\omega),\CDN(\eta))\text{,}
      \end{equation*}
      where $F_{\mathsf{inner}} : [0,\infty)^2\rightarrow[0,\infty)$ is weakly
      increasing for the product order, and such that
      $F_{\mathsf{inner}}(x,y) \geq 2 x y$ for all $x,y \geq 0$,
    \item for all $\omega\in\dom{\CDN}$ and $a\in \dom{\Lip_\A}$, we have:
      \begin{equation*}
        \CDN(a \omega) \leq F_{\mathsf{mod}}(\norm{a}{\A},\Lip_\A(a), \CDN(\omega)) \text{,}
      \end{equation*}
      where $F_{\mathsf{mod}} : [0,\infty)^3\rightarrow[0,\infty)$ is weakly
      increasing for the product order and such that
      $F_{\mathsf{mod}}(x,y,z) \geq (x+y)z$.
    \end{enumerate}
  \end{enumerate}
  A triple of functions $(F,F_{\mathsf{inner}},F_{\mathsf{mod}})$ as above is called
  \emph{permissible}.

  A \emph{Leibniz} {\gMVB} is a {\MVB{F}} where, for all $x,y,z,t\geq 0$, we have $F(x,y,z,t) = x z + y t$,  $F_{\mathsf{mod}}(x,y,z) = (x+y)z$ and $F_{\mathsf{inner}}(x,y) = 2 x y$.
\end{definition}

\begin{remark}
  We note that we do not require any inequality on
  $\CDN(\omega\cdot b)$ for $b\in \B$ and $\omega\in\module{M}$ in
  Definition (\ref{metrical-bundle-def}), using the notation in that
  definition. Indeed, as explained in \cite{Latremoliere18d}, it is
  not needed to define our metric: Condition (3c) does suffice.
\end{remark}

\begin{remark}
  We made a change to \cite{Latremoliere18d} where we introduced the
  similar notion of a ``metrical quantum vector bundles.'' to our
  notion of metrical C*-correspondence. The change is that a metric
  C*-correspondence is indeed a C*-correspondence, and involves both a
  right and a left action. Moreover, we reversed the order of the two
  quantum compact metric spaces in our notation. We will comment when
  these changes would require some modifications to the proofs in
  \cite{Latremoliere18d}, which are, as we shall see, very simple and
  minor.

  We also will work with right modules, instead of left modules, when
  discussing {\gQVB s}, using the following definition.
\end{remark}

\begin{definition}
  A {\MVB{F}} of the form
  \begin{equation*}
    (\module{M},\CDN,\C,0,\A,\Lip)
  \end{equation*}
  simply denoted by $(\module{M},\CDN,\A,\Lip)$, is called a (right)
  \emph{\QVB{F}}, and $(F,F_{\mathsf{inner}})$ is called a \emph{permissible pair}.
\end{definition}

Quantum metrized vector bundles are modeled after Hermitian vector
bundles endowed with a choice of a metric connection, which is used to
define the D-norms \cite{Latremoliere16c} --- however, we do not require {\gQVB s} to be projective in general. The introduction of the
more general {\gMVB s} is actually motivated by spectral triples.

The following theorem, upon which our present work relies, brings
together our work on modules in noncommutative metric geometry and
noncommutative differential geometry.

\begin{convention}
  A Hilbert space $\Hilbert$ is canonically a $\C$-right Hilbert
  module, by setting $\xi\cdot z = z\xi$ for all $\xi \in \Hilbert$
  and $z\in\C$ (since $\C$ is Abelian). To minimize notations, we will
  typically continue to write our scalars on the left when working
  with Hilbert spaces (but not when working with right Hilbert
  modules), with the understanding, when needed, that we mean this
  canonical right action.
\end{convention}

\begin{theorem}\label{Dnorm-thm}
  Let $(\A,\Hilbert,D)$ be a metric spectral triple. If for all
  $a\in\A$ such that $a\,\dom{D}\subseteq D$ and $[D,a]$ is bounded on
  $\dom{D}$, we set:
  \begin{equation*}
    \Lip_D(a) = \opnorm{[D,\pi(a)]}{}{\Hilbert} \text{,}
  \end{equation*}
  and, for all $\xi \in \dom{D}$, we set:
  \begin{equation*}
    \CDN(\xi) = \norm{\xi}{\Hilbert} + \norm{D\xi}{\Hilbert} \text{,}
  \end{equation*}
  then $(\Hilbert,\CDN,\A,\Lip_D,\C,0)$ is a {\LMVB}, which we denote
  by $\mvb{\A}{\Hilbert}{D}$.
\end{theorem}

\begin{proof}
  For any $a\in\dom{\Lip_D}$ and $\xi \in \dom{D}$, we compute:
  \begin{align*}
    \inner{D a \xi}{D a \xi}{\Hilbert} 
    &= \inner{D a \xi - a D \xi}{D a \xi}{\Hilbert} + \inner{a D \xi}{D a \xi}{\Hilbert} \\
    &= \inner{[D,a]\xi}{D a \xi}{\Hilbert} + \inner{a D \xi}{D a \xi}{\Hilbert} \\
    &= \inner{[D,a]\xi}{[D,a]\xi}{\Hilbert} + \inner{[D,a]\xi}{a D \xi}{\Hilbert} \\
    &\quad + \inner{a D \xi}{D a \xi}{\Hilbert} \\
    &= \inner{[D,a]\xi}{[D,a]\xi}{\Hilbert} + \inner{[D,a]\xi}{a D \xi}{\Hilbert} \\
    &\quad + \inner{a D \xi}{ [ D, a ] \xi}{\Hilbert} + \inner{a D \xi}{a D \xi}{\Hilbert} \\
    &= \inner{[D,a]\xi}{[D,a]\xi}{\Hilbert} + 2\Re \inner{[D,a]\xi}{a D \xi}{\Hilbert} \\
    &\quad + \inner{a D\xi}{a D\xi}{\Hilbert} \\
    &\leq \norm{[D,a]\xi}{\Hilbert}^2 + 2\norm{[D,a]\xi}{\Hilbert}\norm{a}{\A}\norm{D\xi}{\Hilbert} + \norm{a}{\A}^2\norm{D\xi}{\Hilbert}^2\\
    &= \left( \norm{[D,a]\xi}{\Hilbert} + \norm{a}{\A}\norm{D\xi}{\Hilbert}  \right)^2\\
    &\leq \left( \Lip_D(a) \norm{\xi}{\Hilbert} + \norm{a}{\A}\CDN(\xi) \right)^2 \text{.}
  \end{align*}
  Hence,
  $\norm{D a \xi}{\Hilbert} \leq \Lip_D(a)\norm{\xi}{\Hilbert} +
  \norm{a}{\A}\norm{D\xi}{\Hilbert}$. Now, since
  $\norm{a\xi}{\Hilbert}\leq\norm{a}{\A}\norm{\xi}{\Hilbert}$, we
  conclude that
  $\CDN(a\xi)\leq \Lip_D(a)\norm{\xi}{\Hilbert} +
  \norm{a}{\A}\CDN(\xi) \leq (\Lip_D(a) + \norm{a}{\A})\CDN(\xi)$.

  Now, $\Hilbert$ is a Hilbert $\C$-module, and $(\C,0)$ is a {\Lqcms}
  (the only possible one with C*-algebra $\C = C(\{0\})$) . Therefore,
  $\left(\Hilbert,\CDN,\A,\Lip,\C,0\right)$
  has all the properties of a {\LMVB}, as long as we prove the
  compactness of the unit ball of $\CDN$.
  
  Let $\xi \in \dom{D}$ with $\CDN(\xi)\leq 1$. By construction,
  $\norm{(D+i)\xi}{\Hilbert} \leq \norm{D\xi}{\Hilbert} +
  \norm{\xi}{\Hilbert} \leq 1$. By definition, $D+i$ has a compact
  inverse, which we denote by $K$. We then have:
  \begin{align*}
    \left\{ \xi \in \Hilbert : \CDN(\xi)\leq 1 \right\} 
    &= K \left\{ (D+i)\xi : \xi \in \Hilbert, \CDN(\xi)\leq 1 \right\} \\
    &\subseteq K \left\{ \xi \in \Hilbert : \norm{\xi}{\Hilbert} \leq 1 \right\}
  \end{align*}
  and, as $K$ is compact, the set
  $K \left\{ \xi \in \Hilbert : \norm{\xi}{\Hilbert} \leq 1 \right\}$,
  and therefore, the unit ball of $\CDN$, are totally bounded in
  $\Hilbert$.

  It remains to show that $\CDN$ is lower semicontinuous. We thus now
  prove that the unit ball of $\CDN$ is closed in
  $\norm{\cdot}{\Hilbert}$.

  Let $(\xi_n)_{n\in\N}$ be a sequence in $\dom{D}$ converging to
  $\xi$ in $\Hilbert$ and with $\CDN(\xi_n) \leq 1$ for all
  $n\in\N$. Let $\eta\in\dom{D}$. We compute:
  \begin{align*}
    \left|\inner{\xi}{D\eta}{\Hilbert}\right|
    &= \lim_{n\rightarrow\infty} \left|\inner{\xi_n}{D\eta}{\Hilbert}\right| \\
    &= \lim_{n\rightarrow\infty} \left|\inner{D\xi_n}{\eta}{\Hilbert}\right| \\
    &\leq \limsup_{n\rightarrow\infty} \norm{D\xi_n}{\Hilbert} \norm{\eta}{\Hilbert} \\
    &\leq \limsup_{n\rightarrow\infty}\left(1 - \norm{\xi_n}{\Hilbert}\right) \norm{\eta}{\Hilbert}\\
    &= \left(1-\norm{\xi}{\Hilbert}\right)\norm{\eta}{\Hilbert} \text{.}
  \end{align*}
  Therefore, the map
  $\eta\in\dom{D} \mapsto \inner{\xi}{D\eta}{\Hilbert}$ is
  continuous. Hence $\xi \in \dom{D^\ast} = \dom{D}$, and thus for all
  $\eta\in\dom{D}$:
  \begin{equation*}
    \left|\inner{D\xi}{\eta}{\Hilbert}\right| = \left|\inner{\xi}{D\eta}{\Hilbert}\right|  \leq \left(1-\norm{\xi}{\Hilbert}\right)\norm{\eta}{\Hilbert} \text{.}
  \end{equation*}

  Thus $\eta\in\dom{D}\mapsto \inner{D\xi}{\eta}{\Hilbert}$ is
  uniformly continuous (as a
  $\left(1-\norm{\xi}{\Hilbert}\right)$-Lipschitz function) linear map
  on the dense subset $\dom{D}$, and thus extends uniquely to
  $\Hilbert$, where it has norm $1-\norm{\xi}{\Hilbert}$. Therefore
  $\norm{D\xi}{\Hilbert} \leq 1-\norm{\xi}{\Hilbert}$ and thus
  $\CDN(\xi) \leq 1$ as desired.

  Thus $\CDN$ is indeed a D-norm.

  Hence, if $(\A,\Lip)$ is a {\qcms}, we conclude that:
  \begin{equation*}
    \mvb{\A}{\Hilbert}{D} = \left( \Hilbert, \CDN, \A, \Lip, \C, 0 \right)
  \end{equation*}
  is a {\LMVB}.
\end{proof}

\begin{remark}
  If $(F,F_{\mathsf{inner}},F_{\mathsf{mod}})$ is any permissible triple, then by definition, $\mvb{A}{\Hilbert}{D}$ is a {\MVB{F}} for any metric spectral triple $(\A,\Hilbert,D)$.
\end{remark}

As we know how to construct {\LMVB s} from metric spectral triples, it
is only natural to apply the metrical propinquity to them, as
introduced in \cite{Latremoliere18d} (with the minor adjustments
below). We now review the construction of the modular and metrical
propinquity, and we refer to \cite{Latremoliere18d} for details; we
will only indicate where we make minor changes to deal with the
changes from left to right modules. We do recall from
\cite{Latremoliere16c,Latremoliere18d} the notions of module morphisms
and modular quantum isometry which we will now use.

\begin{remark}
  The term \emph{modular}, in this paper, is always used as the adjective for module, and \emph{not} in the sense of Tomita-Takesaki theory.
\end{remark}

\begin{definition}[\cite{Latremoliere16c,Latremoliere18d}]
  If $\module{M}$ is a right $\A$-module, and if $\module{N}$ is a
  right $\B$-module for two unital C*-algebras $\A$ and $\B$, then a
  \emph{module morphism} $(\Pi,\pi)$ from $\module{M}$ to $\module{N}$
  is a *-morphism $\pi : \A\rightarrow \B$ and a $\C$-linear
  map $\Pi : \module{M} \rightarrow \module{N}$ such that
  for all $a\in\A$ and $\omega\in\module{M}$, we have
  $\Pi(\omega\cdot a) = \Pi(\omega)\pi(a)$.

  The definition of a left module morphism is similar.

  If moreover $\module{M}$ and $\module{N}$ are right Hilbert modules
  over, respectively, $\A$ and $\B$, then $(\Pi,\pi)$ is a Hilbert
  module morphism when it is a right module morphism such that
  $\inner{\Pi(\omega)}{\Pi(\eta)}{\module{N}} =
  \pi(\inner{\omega}{\eta}{\module{M}})$ for all
  $\omega,\eta\in\module{M}$.

  Last, if $\module{M}$ is an $\A_1$-$\B_1$ C*-correspondence and
  $\module{N}$ is a $\A_2$-$\B_2$ C*-correspondence, then a
  C*-correspondence morphism $(\Pi,\pi,\theta)$ from $\module{M}$ to
  $\module{N}$ is given by a right Hilbert module morphism $(\Pi,\theta)$ from
  $\module{M}$ to $\module{N}$, seen respectively as $\B_1$ and $\B_2$
  right Hilbert modules, and a left module morphism $(\Pi,\pi)$ from
  $\module{M}$ and $\module{N}$, seen respectively as $\A_1$ and
  $\A_2$ left modules.
\end{definition}

\begin{definition}[\cite{Latremoliere16c,Latremoliere18d}]
  If $\mathds{A} = (\module{M},\CDN_\A,\A,\Lip_\A)$ and
  $\mathds{B} =(\module{N},\CDN_\B,\B,\Lip_\B)$ are two {\gQVB s},
  then a \emph{modular quantum isometry}
  $(\Pi,\pi) : \mathds{A} \rightarrow\mathds{B}$ is a right
  Hilbert module morphism from $\module{M}$ to $\module{N}$ such that $\pi : (\A,\Lip_\A) \rightarrow (\B,\Lip_\B)$ is a quantum isometry, $\Pi$ is surjective, and for all $\omega\in\module{N}$:
  \begin{equation*}
    \CDN_\B(\omega) = \inf\left\{ \CDN_\A(\eta) : \eta\in\dom{\CDN_\A}, \Pi(\eta)=\omega \right\}
  \end{equation*}
  (with $\inf\emptyset=\infty$).
 
  A modular quantum isometry $(\Pi,\pi)$ is a \emph{full module
    quantum isometry} when both $\pi$ and $\Pi$ are bijections, $\pi$
  is a full quantum isometry, and $\CDN_\B\circ\Pi = \CDN_\A$.
\end{definition}

\begin{remark}
  If $(\Pi,\pi)$ is a modular quantum isometry from $(\module{M},\CDN_\A,\A,\Lip_\A)$ onto $(\module{N},\CDN_\B,\B,\Lip_\B)$, then by definition, $\CDN_\B$ is the quotient norm of $\CDN_\A$ via the linear map $\Pi : \module{M} \rightarrow \module{N}$.
\end{remark}

As with quantum isometries, we note that if $(\Pi,\pi)$ is a module
quantum isometry from $(\module{M},\CDN_\A,\A,\Lip_\A)$ to
$(\module{N},\CDN_\B,\B,\Lip_\B)$, then
$\Pi(\dom{\CDN_\A})\subseteq\dom{\CDN_\B}$ --- if
$\omega\in\dom{\CDN_\A}$ then
$\CDN_\B(\Pi(\omega)) \leq \CDN_\A(\omega) < \infty$ by definition, so
$\Pi(\omega)\in\dom{\CDN_\B}$. Moreover, if $(\Pi,\pi)$ is a full
modular quantum isometry, then $\Pi(\dom{\CDN_\A})=\dom{\CDN_\B}$ by
symmetry.

From our perspective, two {\gQVB s} are isomorphic when there exists a
full metrical quantum isometry between them. Putting all these
ingredients together, we get the following notion for quantum
isometries and isomorphism of {\gMVB s}:

\begin{definition}[\cite{Latremoliere18d}]\label{mod-quantum-isometry-def}
  If
  \begin{equation*}
    \mathds{A}_1 = (\module{M}_1,\CDN_1,\A_1,\Lip_1,\B_1,\Lip'_1)\text{ and }\mathds{A}_2 = (\module{M}_2,\CDN_2,\A_2,\Lip_2,\B_2,\Lip'_2)\text{,}
  \end{equation*}
  are {\gMVB s}, then $(\Pi,\pi,\theta) : \mathds{A}_1 \rightarrow\mathds{A}_2$ is a
  \emph{metrical quantum isometry} when:
  \begin{enumerate}
  \item $(\Pi,\pi,\theta)$ is a C*-correspondence morphism,
  \item $(\Pi,\theta)$ is a modular quantum isometry from
    $(\module{M}_1,\CDN_1,\B_1,\Lip'_1)$ to
    $(\module{M}_2,\CDN_2,\B_2,\Lip'_2)$,
  \item $\pi : (\A_1,\Lip_1)\rightarrow(\A_2,\Lip_2)$ is a quantum
    isometry.
  \end{enumerate}

  Moreover, $(\Pi,\pi,\theta)$ is a \emph{full metrical quantum
    isometry} when $(\Pi,\pi)$ is a full modular quantum isometry, and
  $\theta$ is a full quantum isometry.
\end{definition}

\begin{remark}\label{reached-T-norm-rmk}
  We use the notation of Definition (\ref{mod-quantum-isometry-def}). Let $\omega\in\dom{\CDN_2}$. For all $n\in\N$, by definition of a modular quantum isometry, there exists $\eta_n\in\module{M}_1$ such that $\CDN_1(\eta_n)\leq \CDN_2(\omega)\left(1 + \frac{1}{n+1}\right)$. Set $\eta'_n = \frac{1}{1+\frac{1}{n+1}} \eta_n$ for all $n\in\N$: note that $\CDN_1(\eta'_n) \leq \CDN_2(\omega)$. Now, $\{ \xi \in \dom{\CDN_1} : \CDN_1(\xi)\leq \CDN_2(\omega) \}$ is compact, so there exists a convergent subsequence $(\eta'_{n_k})_{k\in\N}$ of $(\eta'_n)_{n\in\N}$ with limit $\eta\in\module{M}_1$ such that $\CDN_1(\eta)\leq \CDN_2(\omega)$. By construction, $\Pi(\eta) = \lim_{k\rightarrow\infty} \frac{1}{1+\frac{1}{n_k+1}}\Pi(\eta_{n_k}) = \omega$. So, again by definition of quantum isometries, we must have $\CDN_2(\omega)\leq \CDN_1(\eta)$ and thus, $\CDN_2(\omega) = \CDN_1(\eta)$.

  So, in short, given a modular quantum isometry $(\Pi,\pi,\theta)$, for all $\omega\in\dom{\CDN_2}$, there exists $\eta\in\dom{\CDN_1}$ such that $\Pi(\eta)=\omega$ and $\CDN_1(\eta)=\CDN_2(\omega)$.
\end{remark}

We now discuss the definition and basic properties of the metrical
propinquity, which defines a topology on the class of {\gMVB s}. We
begin by working with {\gQVB s}. As with the propinquity, we introduce
a notion of tunnel between {\gMVB s}.

\begin{definition}[{\cite{Latremoliere18d}}]\label{modular-tunnel-def}
  Let $(F,F_{\mathsf{inner}})$ be an permissible pair. Let $\mathds{A}_1$,
  $\mathds{A}_2$ be two {\QVB{F}s}. A \emph{modular tunnel}
  $\tau = (\mathds{P},\Theta_1,\Theta_2)$ from $\mathds{A}_1$ to
  $\mathds{A}_2$ is given by a {\QVB{F}} $\mathds{P}$, and, for
  each $j\in\{1,2\}$, a modular quantum isometry
  $\Theta_j : \mathds{P} \rightarrow \mathds{A}_j$.
\end{definition}

The extent of a modular tunnel is actually defined in the same manner
as for tunnels between {\qcms s}.

\begin{definition}[{\cite{Latremoliere18d}}]
  Let $(F,F_{\mathsf{inner}})$ be an permissible pair.  Let
  $\mathds{A}_j = (\module{M}_j,\CDN_j,\A_j,\Lip_j)$, for
  $j\in\{1,2\}$, be two {\QVB{F}s}. The \emph{extent} of a modular
  tunnel
  \begin{equation*}
    \tau = (\mathds{P},(\Theta_1,\theta_1),(\Theta_2,\theta_2))\text{,}
  \end{equation*}
  where $\mathds{P} = (\module{P},\CDN,\D,\Lip_\D)$, is the extent of
  the tunnel $(\D,\Lip_\D,\theta_1,\theta_2)$ from $(\A_1,\Lip_1)$ to
  $(\A_2,\Lip_2)$.
\end{definition}

The modular propinquity is then defined as the usual propinquity,
albeit using modular tunnels:

\begin{definition}[{\cite{Latremoliere18d}}]
  We fix a permissible pair $(F,F_{\mathsf{inner}})$. The \emph{modular
    $(F,F_{\mathsf{inner}})$-propinquity} is defined between any two {\QVB{F}s}
  $\mathds{M}_1$ and $\mathds{M}_2$ as:
  \begin{equation*}
    \dmodpropinquity{F,F_{\mathsf{inner}}}(\mathds{M}_1,\mathds{M}_2) = \inf\left\{ \tunnelextent{\tau} : \text{$\tau$ is a $(F,F_{\mathsf{inner}})$-modular tunnel from $\mathds{M}_1$ to $\mathds{M}_2$} \right\} \text{.}
  \end{equation*}
\end{definition}

We were able to establish that:

\begin{theorem}[{\cite{Latremoliere18d}}]\label{modular-prop-thm}
  Let $(F,F_{\mathsf{inner}})$ be an permissible pair of continuous functions.  The
  modular propinquity is a complete metric on the class of
  {\QVB{F}s} up to full modular quantum isometry.
\end{theorem}

We now make a few necessary comments about the proof of Theorem
(\ref{modular-prop-thm}). In \cite{Latremoliere18d}, our {\gQVB s} are
defined as left Hilbert modules, while now, our {\gQVB s} are right Hilbert 
modules. This only requires very trivial changes to
\cite{Latremoliere18d}. We simply have to write our scalars on the
right in \cite[Theorem 3.11]{Latremoliere18d} when defining the module
$\mathscr{B}$ (using the notation in that paper). Another, very minor,
change in the proof of \cite[Theorem 3.22]{Latremoliere18d} is that we
simply write the action on the right in \cite[Eq. (3.1) of Proof of
Theorem 3.22]{Latremoliere18d}. Similar trivial changes apply in the
proof of \cite[Theorem 5.3]{Latremoliere18d}. Nothing else needs
change.

The metrical propinquity then adds the data needed to work with
metrical C*-correspondences.

\begin{definition}\label{metrical-tunnel-def}
  Let $(F,F_{\mathsf{inner}},F_{\mathsf{mod}})$ be a permissible triple. Let $\mathds{A}_1$ and
  $\mathds{A}_2$ be two $(F,F_{\mathsf{inner}},F_{\mathsf{mod}})$-metrical C*-correspondences for
  $j\in\{1,2\}$.

  A \emph{metrical tunnel} $\tau = (\mathds{P},\Theta_1,\Theta_2)$ is
  given by an $(F,F_{\mathsf{inner}},F_{\mathsf{mod}})$ metrical C*-corres\-pondence $\mathds{P}$ and,
  for each $j\in\{1,2\}$, a metrical quantum isometry
  $\Theta_j : \mathds{P} \twoheadrightarrow \mathds{A}_j$.
\end{definition}

\begin{notation}\label{metrical-tunnel-notation}
  There is an equivalent description of metrical tunnels which,
  sometimes, proves helpful, and also motivates our definition for the
  extent of a metrical tunnel. Let
  $\tau = (\mathds{P},(\Pi_1,\pi_1,\theta_1),(\Pi_2,\pi_2,\theta_2))$
  be a metrical tunnel from $\mathds{M}_1$ to $\mathds{M}_2$, where,
  for all $j\in\{1,2\}$, the metrical C*-correspondence $\mathds{M}_j$ is given as $\mathds{M}_j =
  (\module{M}_j,\CDN_j,\A_j,\Lip_j,\B_j,\TLip_j)$. Moreover, we write
  the metrical C*-correspondence $\mathds{P}$ as
  $(\module{P},\TDN,\D,\Lip_\D,\alg{E},\Lip_{\alg{E}})$.

  Let us now set
  $\rho =
  \big((\module{P},\TDN,\alg{E},\Lip_{\alg{E}}),(\Pi_1,\theta_1),(\Pi_2,\theta_2)\big)$
  and $\mu = \big(\D,\Lip_{\D},\pi_1,\pi_2)$. We then note
  that, by Definitions (\ref{tunnel-def}),(\ref{modular-tunnel-def}),
  and (\ref{metrical-tunnel-def}):
  \begin{enumerate}
  \item $\rho$ is a modular tunnel from
    $(\module{M}_1,\CDN_1,\B_1,\TLip_1)$ to
    $(\module{M}_2,\CDN_2,\B_2,\TLip_2)$,
  \item $\mu$ is a tunnel from $(\A_1,\Lip_1)$ to $(\A_2,\Lip_2)$,
  \item $\module{P}$ is an $\D$-$\alg{E}$-C*-correspondence,
  \item
    $\forall e \in \alg{E} \quad \forall \omega \in \module{P} \quad
    \TDN(e\omega)\leq
    F_{\mathsf{mod}}(\norm{e}{\alg{E}},\Lip_{\alg{E}}(e),\TDN(\omega))$.
  \end{enumerate}
  In the rest of this paper, we will denote $\rho$ by
  $\tau_{\mathsf{mod}}$, and we will denote $\mu$ by
  $\tau_{\mathsf{base}}$, whenever needed.

  Conversely, if
  $\rho=(
  (\module{P},\TDN,\alg{E},\Lip_{\alg{E}}),(\Pi_1,\theta_1),(\Pi_2,\theta_2))$ and
  $\mu = (\D,\Lip_\D,\pi_1,\pi_2)$ satisfy (1)--(4) above,
  then
  $\tau =
  ((\module{P},\TDN,\D,\Lip_\D,\alg{E},\Lip_{\alg{E}}),(\Pi_1,\pi_1,\theta_1),(\Pi_2,\pi_2,\theta_2))$
  is a metrical tunnel from $\mathds{M}_1$ to $\mathds{M}_2$. Thus, it
  may sometimes be convenient to work with the pair $(\rho,\mu)$ in
  place of $\tau$.
\end{notation}

\begin{definition}[{\cite{Latremoliere18d}}]
  The \emph{extent}, $\tunnelextent{\tau}$, of a metrical tunnel
  \begin{equation*}
    \tau = \left( \left(\module{M},\TDN,\D,\Lip_\D,\alg{E},\Lip_{\alg{E}}\right), (\Pi_1,\pi_1,\theta_1),(\Pi_2,\pi_2,\theta_2) \right)
  \end{equation*}
  is given by
  \begin{equation*}
    \tunnelextent{\tau} = \max\left\{\tunnelextent{\left(\D,\Lip_\D,\pi_1,\pi_2\right)},\tunnelextent{\left(\alg{E},\Lip_{\alg{E}},\theta_1,\theta_2 \right)}\right\} \text{.}
  \end{equation*}
\end{definition}

\begin{definition}[{\cite{Latremoliere18d}}]\label{metrical-prop-def}
  Let $(F,F_{\mathsf{inner}},F_{\mathsf{mod}})$ be a permissible triple. The \emph{metrical
    propinquity}, $\dmetpropinquity{F}(\mathds{A},\mathds{B})$,
  between two $(F,F_{\mathsf{inner}},F_{\mathsf{mod}})$ metrical C*-correspondences $\mathds{A}$ and
  $\mathds{B}$ is the nonnegative number given by
  \begin{equation*}
    \dmetpropinquity{F}(\mathds{A},\mathds{B}) = \inf\left\{\tunnelextent{\tau} : \text{ $\tau$ is a metrical $(F,F_{\mathsf{inner}},F_{\mathsf{mod}})$-tunnel from $\mathds{A}$ to $\mathds{B}$} \right\}\text{.}
  \end{equation*}
\end{definition}

\begin{theorem}[{\cite{Latremoliere18d}}]\label{metrical-propinquity-thm}
  Let $(F,F_{\mathsf{inner}},F_{\mathsf{mod}})$ be a permissible triple of continuous functions. The metrical propinquity $\dmetpropinquity{F}$ is a complete metric, up to full quantum isometry, on the class of $(F,F_{\mathsf{inner}},F_{\mathsf{mod}})$--metrical
  C*-correspondences.
\end{theorem}

\begin{notation}
  When the context makes it clear, we will omit the permissible triple from the notation of the metrical propinquity.
\end{notation}

There is no additional changes needed in the proof of \cite[Theorem
4.9]{Latremoliere18d}, besides what we discussed after Theorem
(\ref{modular-prop-thm}). The only change in the proof of
\cite[Theorem 5.4]{Latremoliere18d} about the completeness of
$\dmetpropinquity{F}$ is just to verify that the limit is indeed a bimodule, and this follows immediately from the construction of this limit as a quotient of a bimodule.

A subclass of metrical C*-correspondences is given by metric spectral
triples via Theorem (\ref{Dnorm-thm}). Of interest is the meaning of
distance zero for the metrical propinquity, when applied to metric
spectral triples.

\begin{proposition}\label{spectral-mvb-zero-prop}
  We fix a permissible triple $(F,F_{\mathsf{inner}},F_{\mathsf{mod}})$.

  Let $(\A,\Hilbert_\A,D_\A)$ and $(\B,\Hilbert_\B,D_\B)$ be two
  metric spectral triples. The following assertions are equivalent:
  \begin{enumerate}
  \item
    $\dmetpropinquity{F}(\mvb{\A}{\Hilbert_\A}{D_\A},
    \mvb{\B}{\Hilbert_\B}{D_\B}) = 0$,
  \item there exists a full quantum isometry
    $\rho : (\A,\Lip_{D_\A}) \rightarrow (\B,\Lip_{D_\B})$ and a
    unitary $U : \Hilbert_\A \rightarrow \Hilbert_\B$ such that
    $U\dom{D_\A}=\dom{D_\B}$, and $\rho = \AdRep{U}$, while
    \begin{equation*}
      \norm{D_\A\xi}{\Hilbert_\A} = \norm{D_\B U\xi}{\Hilbert_\B} \text{.}
    \end{equation*}
  \end{enumerate}
\end{proposition}

\begin{proof}
  We identify $\A$ as its image acting on $\Hilbert_\A$ for the
  spectral triple $(\A,\Hilbert_\A,D_\A)$, and similarly with
  $\B$. Moreover, let
  $\CDN_\A:\xi\in\dom{D_\A}\mapsto \norm{\xi}{\Hilbert_\A} +
  \norm{D_\A\xi}{\Hilbert}$, and similarly with $\CDN_\B$.

  By Theorem (\ref{metrical-propinquity-thm}), since:
  \begin{equation*}
    \dmetpropinquity{F}(\mvb{\A}{\Hilbert_\A}{D_\A}),\mvb{\B}{\Hilbert_\B}{D_\B})) = 0
  \end{equation*}
  the {\gMVB s} $\mvb{\A}{\Hilbert_\A}{D_\A}$ and 
  $\mvb{\B}{\Hilbert_\B}{D_\B}$ are metrically isomorphic. Thus, there
  exists a full quantum isometry
  $\rho: (\A,\Lip_{D_\A}) \rightarrow (\B,\Lip_{D_\B})$ and a surjective linear
  isometry, i.e. a unitary $U : \Hilbert_\A \rightarrow \Hilbert_\B$
  such that
  \begin{enumerate}
  \item $U\dom{\CDN_\A}=\dom{\CDN_\B}$, i.e. $U\dom{D_\A}=\dom{D_\B}$,
  \item $\CDN_{\B}\circ U = \CDN_{\A}$ on $\dom{D_\A}$,
  \item $(\rho,U)$ is a module morphism from $\Hilbert_\A$ to
    $\Hilbert_\B$ (as modules over, respectively, $\A$ and $\B$)\text.
  \end{enumerate}
  There is also a full quantum isometry $\mathrm{\iota}$ from $(\C,0)$
  to itself such that $(\iota,U)$ is a Hilbert $\C$-module map, but of
  course, $\iota$ is the identity.

  Thus to begin with, if $a\in\A$ and $\xi\in\Hilbert_\A$, then, since
  $(\rho,U)$ is a module morphism:
  \begin{equation*}
    \rho(a)U\xi = U(a\xi) \text{ so }\forall \eta \in \Hilbert_\B \quad \rho(a)\eta = UaU^\ast\eta \text{.}
  \end{equation*}

  Moreover, since $\CDN_{\B}\circ U = \CDN_{\A}$ (including when
  either of these norms take the value $\infty$), we conclude, first,
  that $U$ maps $\dom{D_\A}$ onto $\dom{D_\B}$, and then, for all
  $\xi \in \dom{D_\A}$:
  \begin{equation*}
    \norm{\xi}{\Hilbert_\A} + \norm{D_\A \xi}{\Hilbert_\A} = \norm{U\xi}{\Hilbert_\B} + \norm{D_\B U \xi}{\Hilbert_\B}
  \end{equation*}
  and since $U$ is an isometry,
  $\norm{\xi}{\Hilbert_\A}=\norm{U\xi}{\Hilbert_\B}$, and therefore we
  conclude for all $\xi \in \dom{D_\A}$:
  \begin{equation}\label{norm-eq}
    \begin{split}
      \norm{D_\A \xi}{\Hilbert_\A}
      &= \norm{D_\B U \xi}{\Hilbert_\B} \\
      &=\norm{U^\ast D_\B U \xi}{\Hilbert_\A} \text{.}
    \end{split}
  \end{equation}

  This concludes our proof.
\end{proof}

While the metrical propinquity allows to recover some metric
information and domain information about metric spectral triples, we
aim at a stronger result in this paper, where we want to define a
distance on metric spectral triples, up to equivalence of spectral
triples. To this end, we propose to account for the natural quantum
dynamics given by a spectral triple on its underlying Hilbert space,
which is a particular case of an action of a monoid on a metrical
C*-correspondence. We therefore augment our previous construction
of the metrical propinquity to incorporate monoid actions. The next
section presents the construction at a higher level of generality than
needed for spectral triples, but the proofs are not any more involved
(in fact, the higher generality makes the exposition clearer), and
this construction can be used for other examples, such as dealing with
the action of the Heisenberg group on Heisenberg modules over quantum
tori, for example.

\section{The covariant Metrical Propinquity}

We begin by constructing the covariant modular propinquity, defined on
the class of objects consisting of {\gQVB s} endowed with a proper
monoid action, appropriately defined as follows.

\begin{definition}[{\cite{Latremoliere18b}}]
  A \emph{proper metric monoid} $(G,\delta)$ is a monoid $G$ and a left invariant metric $\delta$ on $G$ which induces a topology of a proper metric space on $G$ (i.e. a topology for which all closed balls are compact) for which the multiplication on $G$ is continuous.
\end{definition}

Lipschitz dynamical systems are actions of proper metric monoids on
{\qcms s}. While we developed the covariant propinquity between such
systems which acts by positive linear maps \cite{Latremoliere18b}, for
our current purpose, we will focus on actions by *-endomorphisms.

\begin{definition}[{\cite{Latremoliere18b}}]\label{Lipschitz-dynamical-system-def}
  A \emph{Lipschitz dynamical system} $(\A,\Lip,\alpha,H,\delta_H)$ is
  a {\qcms} $(\A,\Lip)$, a proper metric monoid $H$ and a
  monoid morphism $\alpha$ from $H$ to the monoid of *-endomorphisms
  of $\A$, such that:
  \begin{enumerate}
  \item $\alpha$ is strongly continuous: for all $a\in\A$ and
    $g \in H$, we have
    \begin{equation*}
      \lim_{h\rightarrow g} \norm{\alpha^g(a)-\alpha^h(a)}{\A} = 0\text{, }
    \end{equation*}
  \item for all $h \in H$, the *-endomorphisms $\alpha^h$ satisfies
    $\alpha^h(\dom{\Lip})\subseteq\dom{\Lip}$,
  \item there exists a locally bounded function
    $K: H \rightarrow [0,\infty)$ such that, for all $h\in H$, we have
    $\Lip\circ\alpha^h \leq K(h)\Lip$.
  \end{enumerate}
\end{definition}
Condition (2) in Definition (\ref{Lipschitz-dynamical-system-def}) is
actually one of several equivalent definitions of a \emph{Lipschitz
  morphism} \cite{Latremoliere16b}, and in particular, Condition (2)
implies that, for all $h \in H$, there indeed exists
$K(h)\in[0,\infty)$ such that $\Lip\circ\alpha^h \leq K(h)\Lip$;
Condition (3) adds a minimum regularity on such a function $K$.

\begin{definition}\label{covariant-system-def}
  Let $(F,F_{\mathsf{inner}})$ be a permissible pair. A \emph{covariant modular $(F,F_{\mathsf{inner}})$-system} $\CMS{\module{M}}{\CDN}{\beta}{(G,\delta_G,q)}{\A}{\Lip}{\alpha}{(H,\delta_H)}$ is given by:
  \begin{enumerate}
  \item a {\QVB{F}} $(\module{M},\CDN,\A,\Lip)$,
  \item a Lipschitz dynamical system $(\A,\Lip,\alpha,H,\delta_H)$,
  \item a proper metric monoid $(G,\delta_G)$,
  \item a continuous monoid morphism $q$ from $(G,\delta_G)$ to
    $(H,\delta_H)$,
  \item for each $g \in G$, we have an $\A$-linear endomorphism $\beta^g$ of $\module{M}$ such that:
    \begin{enumerate}
    \item $g\in G \mapsto \beta^g$ is a monoid morphism,
    \item the pair $(\beta^g,\alpha^{q(g)})$ is a Hilbert module map.
    \item for all $\omega\in\module{M}$ and $g\in G$, we have:
      \begin{equation*}
        \lim_{h\rightarrow g}\norm{\beta^h(\omega)-\beta^g(\omega)}{\module{M}} = 0\text{,}
      \end{equation*}
    \item there exists a locally bounded function
      $K : G\rightarrow [0,\infty)$ such that for all $g \in G$, we
      have $\CDN\circ\beta^g \leq K(g) \CDN$.
    \end{enumerate}
  \end{enumerate}
\end{definition}

\begin{remark}
  Using the notation of Definition (\ref{covariant-system-def}), we note that, for all $g \in G$, and for all $\xi \in \module{M}$, the following inequality holds:
  \begin{equation*}
    \norm{\inner{\beta^g\xi}{\beta^g\xi}{\module{M}}}{\A} = \norm{\alpha^{q(g)}(\inner{\xi}{\xi}{\module{M}})}{\A} \leq \norm{\inner{\xi}{\xi}{\module{M}}}{\A} \text,
  \end{equation*}
  and thus $\opnorm{\beta^g}{}{\module{M}} \leq 1$.
\end{remark}

We recall from \cite{Latremoliere18b} how to define a covariant
version of the Gromov-Hausdorff distance between proper metric
monoids. The key ingredient is an approximate notion of an almost
isometric isomorphism, defined as follows.
    
\begin{notation}
  If $(G,d)$ is a metric monoid, then the closed ball centered at the
  unit of $G$, and of radius $r \geq 0$, is denoted by $G[r]$.
\end{notation}

\begin{definition}[{\cite{Latremoliere18b}}]
  Let $(G_1,\delta_1)$ and $(G_2,\delta_2)$ be two proper metric monoids.
  
  A $r$-local $\varepsilon$-almost isometric isomorphism
  $(\varsigma_1,\varsigma_2)$ from $(G_1,\delta_1)$ to
  $(G_2,\delta_2)$ is a pair of maps
  $\varsigma_1 : G_1 \rightarrow G_2$ and
  $\varsigma_2 : G_2 \rightarrow G_1$ such that for all
  $\{j,k\} = \{1,2\}$:
  \begin{enumerate}
  \item $\varsigma_j$ maps the unit of $G_j$ to the unit of $G_k$,
  \item for all $g,g'\in G_j[r]$ and $h \in G_k[r]$:
    \begin{equation*}
      \left| \delta_k(\varsigma_j(g)\varsigma_j(g'),h) - \delta_j(g g', \varsigma_k(h))  \right| \leq \varepsilon\text{.}
    \end{equation*}
  \end{enumerate}
  The set of all $r$-local $\varepsilon$-almost isometric isomorphisms
  is denoted by:
  \begin{equation*}
    \UIso{\varepsilon}{(G_1,\delta_1)}{(G_2,\delta_2)}{r} \text{.}
  \end{equation*}
\end{definition}

Local, almost isometries enjoy a natural composition property, which
is the reason why the covariant Gromov-Hausdorff distance they define
is indeed a metric:

\begin{theorem}[{\cite{Latremoliere18b}}]\label{composition-thm}
  Let $(G_1,\delta_1)$, $(G_2,\delta_2)$ and $(G_3,\delta_3)$ be three proper metric monoids.

  Let
  $\varepsilon_1,\varepsilon_2 \in \left( 0 ,
    \frac{\sqrt{2}}{2}\right)$. If
  $\varsigma = (\varsigma_1,\varsigma_2) \in
  \UIso{\varepsilon_1}{G_1}{G_2}{\frac{1}{\varepsilon_1}}$ and
  $\varkappa = (\varkappa_1,\varkappa_2) \in
  \UIso{\varepsilon_2}{G_2}{G_3}{\frac{1}{\varepsilon_2}}$ then:
  \begin{equation*}
    (\varkappa_1\circ\varsigma_1,\varsigma_2\circ\varkappa_2) \in \UIso{\varepsilon_1 + \varepsilon_2}{G_1}{G_3}{\frac{1}{\varepsilon_1 + \varepsilon_2}} \text{.}
  \end{equation*}
  We denote
  $(\varkappa_1\circ\varsigma_1,\varsigma_2\circ\varkappa_2)$ by
  $\varkappa\circ\varsigma$.
\end{theorem}

If $(G,\delta_G)$ and $(H,\delta_H)$ are two proper metric monoids,
then we define their covariant Gromov-Hausdorff distance
$\Upsilon((G,\delta_G),(H,\delta_H))$ as:
\begin{equation*}
  \min\left\{\frac{\sqrt{2}}{2}, \inf\left\{\varepsilon > 0 : \UIso{\varepsilon}{(G,\delta_G)}{(H,\delta_H)}{\frac{1}{\varepsilon}} \not= \emptyset \right\} \right\} \text{,}
\end{equation*}
and we proved in \cite{Latremoliere18b} that $\Upsilon$ is a metric up
to isometric isomorphism of metric monoids on the class of proper
metric monoids; moreover we study conditions on classes of proper
metric monoids to be complete in \cite{Latremoliere18c}. For our
purpose, we will focus on how to use these ideas to construct a
covariant version of $\dmodpropinquity{}$.

We begin with a simple observation. The modular propinquity does not
involve the computation of any quantity directly involving the modules
--- the extent of the basic tunnel is all that is needed. Thus, the
various requirements placed on modular tunnels, regarding maps being
quantum isometries, are sufficient to ensure that the basic tunnel's
extent encodes information about the distance between
modules. However, for our current effort, we introduce another
numerical quantity associated with modular tunnels. This quantity
generalizes the notion of the reach of a tunnel \cite{Latremoliere13b}
--- we will see, in particular, why this quantity is redundant for the
modular quantity.

We begin by defining a form of the {\MongeKant} on the (topological) dual of the underlying module of any {\gQVB}.

\begin{notation}\label{MongeKantAlt-def}
  Let $(\module{M},\CDN,\A,\Lip)$ be a {\gQVB}. For any two continuous linear functionals $\mu:\module{M}\rightarrow\C$ and $\nu:\module{M}\rightarrow\C$ over $\module{M}$, we define:
  \begin{equation*}
    \KantorovichAlt{\CDN}(\mu , \nu) =  \sup_{\substack{ \zeta\in\module{M} \\ \CDN(\zeta)\leq 1 }} \left|\mu(\zeta)-\nu(\zeta) \right|\text.
  \end{equation*}
\end{notation}

Since $\CDN$ dominates the norm, $\KantorovichAlt{\CDN}$ is always finite. Since the closed unit ball of a D-norm is a total set by definition (it has dense $\C$-linear span), $\KantorovichAlt{\CDN}$ is always a metric on the topological dual of $\module{M}$.

However, Hilbert modules need not be self-dual in general, and for our purpose, we will work with a specific subset of continuous linear functionals, which is particularly relevant to our constructions.

\begin{notation}\label{pseudo-state-notation}
  Let $(\module{M},\CDN,\A,\Lip)$ be a {\gQVB}. For any continuous linear functional $\varphi : \A\rightarrow\C$, and for any $\omega\in\module{M}$, we write $\varphi\odot\omega$ for the continuous linear functional $\eta\in\module{M} \mapsto \varphi\left(\inner{\omega}{\eta}{\module{M}}\right)$ over $\module{M}$. We denote the set of all such continuous linear functionals over $\module{M}$ by $\FalseDual{\module{M}}$,i.e.
  \begin{equation*}
    \FalseDual{\module{M}} = \left\{ \varphi\odot\omega : \varphi \in \A', \omega\in\module{M} \right\} \text.
  \end{equation*}

  In particular, the \emph{set of pseudo-states} $\ModStateSpace(\module{M},\CDN)$ of $\module{M}$ is the following subset of $\FalseDual{\module{M}}$:
  \begin{equation*} 
    \left\{ \varphi\odot\omega \in \FalseDual{\module{M}} : \varphi \in \StateSpace(\A), \omega\in\module{M}, \CDN(\omega)\leq 1 \right\}\text{.}
  \end{equation*}
\end{notation}

\begin{remark}
  The set $\ModStateSpace(\module{M})$ is not convex in general. Thus, it
  may be that future applications will prefer to work with the convex
  hull of $\ModStateSpace(\module{M})$, though for our purpose, such a
  change would not affect our work, and the present choice is quite
  natural and easier to handle.
\end{remark}

The topology induced by our new {\MongeKant} on the set of pseudo-states of a {\gQVB} is the weak* topology.

\begin{proposition}
  Let $(\module{M},\CDN,\A,\Lip)$ be a {\gQVB}. The topology induced
  by $\KantorovichAlt{\CDN}$ on $\ModStateSpace(\module{M})$ is the weak*
  topology; moreover the set $\ModStateSpace(\module{M})$ of pseudo-states of 
  $\module{M}$ is weak* compact.
\end{proposition}

\begin{proof}
  Let $(\varphi_n)_{n\in\N}$ be a sequence in $\StateSpace(\A)$ and
  let $(\omega_n)_{n\in\N}$ be a sequence in
  $\{\omega\in\module{M}:\CDN(\omega)\leq 1\}$.

  Assume first that $(\varphi_n\odot\omega_n)_{n\in\N}$ converges
  weakly-* to some linear functional $\mu$ over $\module{M}$. By
  compactness of both $\StateSpace(\A)$, for the weak* topology, and
  of $\{\omega\in\module{M}:\CDN(\omega)\leq 1\}$, for the norm
  topology, there exists a subsequence $(\varphi_{f(n)})_{n\in\N}$ of
  $(\varphi_n)_{n\in\N}$ weak* converging to some
  $\varphi \in \StateSpace(\A)$, and there exists a subsequence
  $(\omega_{g(n)})_{n\in\N}$ of $(\omega_n)_{n\in\N}$ converging to
  some $\omega$ in norm. Note that by lower semicontinuity of $\CDN$,
  we have $\CDN(\omega)\leq 1$. Up to extracting further subsequences,
  we assume $f = g$ without loss of generality.

  Let $\zeta\in\module{M}$ and let $\varepsilon > 0$. Since
  $(\varphi_{f(n)})_{n\in\N}$ weak* converges to $\varphi$, there
  exists $N_1\in\N$ such that if $n\geq N_1$ then
  $|\varphi_{f(n)}(\inner{\omega}{\zeta}{\module{M}}) -
  \varphi(\inner{\omega}{\zeta}{\module{M}})|<\frac{\varepsilon}{3}$. Moreover,
  since $(\omega_{f(n)})_{n\in\N}$ converges to $\omega$ in norm,
  there exists $N_2 \in \N$ such that if $n\geq N_2$ then
  $\norm{\omega-\omega_{f(n)}}{\module{M}} \leq \frac{\varepsilon}{3
    (\norm{\zeta}{\module{M}}+1)}$. Last, as $\mu$ is the weak* limit
  of $(\varphi_n\odot\omega_n)_{n\in\N}$, there exists $N_3 \in \N$
  such that if $n\geq N_3$ then
  $|\mu(\zeta) - \varphi_n\odot\omega_n(\zeta)| <
  \frac{\varepsilon}{3}$.

  If $n\geq \max\{N_1,N_2,N_3\}$ then:
  \begin{align*}
    \left|\mu(\zeta) - \varphi\odot\omega(\zeta)\right| 
    &\leq \left|\mu(\zeta)-\varphi_{f(n)}\odot\omega_{f(n)}(\zeta)\right| + \left|\varphi_{f(n)}\odot\omega_{f(n)}(\zeta) - \varphi\odot\omega(\zeta) \right| \\
    &\leq \frac{\varepsilon}{3} + \left|\varphi_{f(n)}(\inner{\omega_{f(n)}}{\zeta}{\module{M}}) - \varphi_{f(n)}(\inner{\omega}{\zeta}{\module{M}}) \right| \\
    &\quad+ \left|\varphi_{f(n)}(\inner{\omega}{\zeta}{\module{M}})-\varphi(\inner{\omega}{\zeta}{\module{M}})\right| \\
    &\leq \frac{\varepsilon}{3} + \norm{\inner{\omega-\omega_{f(n)}}{\zeta}{\module{M}}}{\A} + \frac{\varepsilon}{3} \\
    &\leq \frac{\varepsilon}{3} + \norm{\omega-\omega_{f(n)}}{\module{M}}\norm{\zeta}{\module{M}} + \frac{\varepsilon}{3} \text{ using Cauchy-Schwarz,}\\
    &\leq \varepsilon \text{.}
  \end{align*}
  Therefore $\mu(\zeta) = \varphi\odot\omega(\zeta)$, since
  $\varepsilon > 0$ is arbitrary. As $\zeta\in\module{M}$ is arbitrary
  as well, we conclude $\mu = \varphi\odot\omega$. Thus
  $\ModStateSpace(\module{M})$ is weak* closed. As it is a subset of the
  unit ball of the dual of $\module{M}$, we conclude that
  $\ModStateSpace(\module{M})$ is weak* compact, by the Banach-Alaoglu
  Theorem.

  Now, the rest of our proof is a standard argument --- see, for instance, \cite[Theorem 1.8]{Rieffel98a}, though we do not quite fit that theorem (because condition (1.3d) is not met here).
  
  We now prove that if a sequence
  $\left(\varphi_n\odot\omega_n\right)_{n\in\N}$ in
  $\ModStateSpace(\module{M})$ converges to $\mu = \varphi\odot\omega$
  for the weak* topology, then it converges to $\mu$ for
  $\KantorovichAlt{\CDN}$. Let $\varepsilon > 0$. Since
  $\left\{\omega\in\module{M}:\CDN(\omega)\leq 1\right\}$ is compact
  for $\norm{\cdot}{\module{M}}$, there exists a finite
  $\frac{\varepsilon}{3}$-dense subset $F$ of
  $\left\{\omega\in\module{M}:\CDN(\omega)\leq 1\right\}$ for the
  norm. As $F$ is finite, and since
  $(\varphi_n\odot\omega_n)_{n\in\N}$ converges to
  $\varphi\odot\omega$ for the weak* topology, there exists $N\in\N$
  such that if $n\geq N$ then
  $|\varphi_n\odot\omega_n(\eta)-\varphi\odot\omega(\eta)| <
  \frac{\varepsilon}{3}$ for all $\eta\in F$.

  Let now $\zeta\in \{\omega\in\module{M}:\CDN(\omega)\leq 1\}$. By
  construction, there exists $\eta\in F$ such that
  $\norm{\zeta-\eta}{\module{M}} < \frac{\varepsilon}{3}$. Since
  $\norm{\omega_n}{\module{M}}\leq\CDN(\omega_n)\leq 1$ for all
  $n\in\N$ and similarly since $\norm{\omega}{}\leq 1$, we then have,
  for all $n\geq N$, that
  \begin{align*}
    \left|\varphi\odot\omega(\zeta)-\varphi_n\odot\omega_n(\zeta)\right| 
    &\leq \left|\varphi\odot\omega(\zeta)-\varphi\odot\omega(\eta)\right| + \left|\varphi\odot\omega(\eta)-\varphi_n\odot\omega_n(\eta)\right| \\
    &\quad + \left|\varphi_n\odot\omega_n(\eta)-\varphi_n\odot\omega_n(\zeta)\right| \\
    &\leq \norm{\inner{\omega}{\zeta - \eta}{}}{\A} + \frac{\varepsilon}{3} + \norm{\inner{\omega_n}{\zeta - \eta}{}}{\A} \\
    &\leq \norm{\omega}{\module{M}} \norm{\zeta-\eta}{\module{M}} + \frac{\varepsilon}{3} \\
    &\quad + \norm{\omega_n}{\module{M}}\norm{\zeta-\eta}{\module{M}} \text{ by Cauchy-Schwarz,}\\
    &\leq \varepsilon \text{.}
  \end{align*}
  Therefore, if $n\geq N$, then
  $\KantorovichAlt{\CDN}\left(\varphi\odot\omega,\varphi_n\odot\omega_n\right)
  \leq \varepsilon$. In conclusion, if
  $\left(\varphi_n\odot\omega_n\right)_{n\in\N}$ is a weak* convergent
  sequence in $\ModStateSpace(\module{M})$, with limit
  $\varphi\odot\omega$, then
  $\lim_{n\rightarrow\infty}\KantorovichAlt{\CDN}(\varphi_n\odot\omega_n,\varphi\odot\omega)
  = 0$.

  We now turn to the converse: we assume that a sequence
  $(\varphi_n\odot\omega_n)_{n\in\N}$ in $\ModStateSpace(\module{M})$
  converges to some $\varphi\odot\omega \in \ModStateSpace(\module{M})$
  for $\KantorovichAlt{\CDN}{}$. This part of our proof is similar to \cite[Proposition 1.4]{Rieffel98a}. Let $\zeta\in\module{M}$ and
  $\varepsilon > 0$. By density of the domain of $\CDN$, there exists
  $\eta \in \dom{\CDN}$ such that
  $\norm{\zeta-\eta}{\module{M}} < \frac{\varepsilon}{3}$. Then, there
  exists $N\in\N$ such that if $n\geq N$ then
  $\KantorovichAlt{\CDN}\left(\varphi_n\odot\omega_n,\varphi\odot\omega\right)
  < \frac{\varepsilon}{3(\CDN(\eta)+1)}$. Thus in particular,
  $|\varphi_n\odot\omega_n(\eta) - \varphi\odot\omega(\eta)| <
  \frac{\varepsilon}{3}$ if $n\geq N$.

  Hence if $n\geq N$ then, as above:
  \begin{align*}
    |\varphi_n\odot \omega_n(\zeta) - \varphi\odot\omega(\zeta)|
    &\leq |\varphi_n\odot \omega_n(\zeta) - \varphi_n\odot\omega_n(\eta)| + |\varphi_n\odot \omega_n(\eta) - \varphi\odot\omega(\eta)| \\
    &\quad + |\varphi\odot \omega(\eta) - \varphi\odot\omega(\zeta)| \\
    &\leq \norm{\inner{\omega_n}{\zeta-\eta}{}}{\A} + \frac{\varepsilon}{3} + \norm{\inner{\omega}{\zeta-\eta}{}}{\A} \\
    &\leq 2\norm{\zeta-\eta}{\module{M}} + \frac{\varepsilon}{3} \leq \varepsilon \text{.}
  \end{align*}
  Hence, $\left(\varphi_n\odot\omega_n\right)_{n\in\N}$ weak*
  converges to $\varphi\odot\omega$ as desired.
\end{proof}

The dual propinquity between {\qcms s} is defined using the extent of
tunnels, though originally \cite{Latremoliere13b} we used a somewhat
different construction using quantities called reach and height. The
relevance of this observation is that while the extent has better
properties, the reach is helpful in defining the covariant version of
the propinquity between Lipschitz dynamical systems.

For this construction, we will use the notion of target sets defined
by tunnels.  As explained in
\cite{Latremoliere13,Latremoliere13b,Latremoliere14,Latremoliere16c,Latremoliere17c,Latremoliere18b,Latremoliere18c,Latremoliere18d},
tunnels are a form of ``almost morphisms'' which induce set-valued
maps which behave as morphisms, using the following definitions:

\begin{definition}[{\cite{Latremoliere13b,Latremoliere18d}}]
  Let $(\A,\Lip_\A)$ and $(\B,\Lip_\B)$ be two {\qcms s}. If $\tau = (\D,\Lip_\D,\pi_\A,\pi_\B)$ is a tunnel, if $a\in\dom{\Lip_\A}$ and if $l\geq\Lip_\A(a)$ then the \emph{target $l$-set} $\targetsettunnel{\tau}{a}{l}$ of $a$ is:
  \begin{equation*}
    \targetsettunnel{\tau}{a}{l} = \left\{ \pi_\B(d) : d\in\dom{\Lip_\D} \text{ such that }\pi_\A(d)=a,  \Lip_\D(d)\leq l \right\} \text{.}
  \end{equation*}
  Let $\mathds{A} = (\module{M},\CDN_\A,\A,\Lip_\A)$ and $\mathds{B} = (\module{N},\CDN_\B,\B,\Lip_\B)$ be two {\gQVB s}. If $\tau = (\mathds{P},(\Pi_\A,\pi_\A),(\Pi_\B,\pi_\B))$ is a modular tunnel with $\mathds{P} = (\module{P},\CDN,\D,\Lip_\D)$, if $\omega\in\module{M}$ and if $l\geq \CDN_\A(\omega)$, then the
  \emph{target $l$-set} $\targetsettunnel{\tau}{\omega}{l}$ of $\omega$ is:
  \begin{equation*}
    \targetsettunnel{\tau}{\omega}{l} = \left\{ \Pi_\B(\zeta) : \zeta\in\dom{\CDN} \text{ such that }\Pi_\A(\zeta)=\omega,  \CDN(\zeta)\leq l \right\} \text{.}
  \end{equation*}
  Moreover, if $a\in\dom{\Lip_\A}$ and $l\geq\Lip_\A(a)$ then we write $\targetsettunnel{\tau}{a}{l}$ for $\targetsettunnel{\tau_\flat}{a}{l}$ where
  $\tau_\flat = (\D,\Lip_\D,\pi_\A,\pi_\B)$.
\end{definition}

\begin{proposition}\label{reach-prop}
  Let $\mathds{A} = (\module{M},\CDN_\A,\A,\Lip_\A)$ and
  $\mathds{B} = (\module{N},\CDN_\B,\B,\Lip_\B)$ be two {\QVB{F}s},
  and let
  $\tau = (\mathds{P},(\Theta_\A,\theta_\A),(\Theta_\B,\theta_\B))$ be
  a modular $(F,F_{\mathsf{inner}})$-tunnel from $\mathds{A}$ to $\mathds{B}$, with
  $\mathds{P} = (\module{P},\CDN,\D,\Lip_\D)$. We have:
  \begin{equation*}
    \Haus{\KantorovichAlt{\CDN}}\left(\left\{\mu\circ\Theta_\A :\mu\in\ModStateSpace(\module{M}) \right\}, \left\{ \nu\circ\Theta_\B:\nu\in\ModStateSpace(\module{N}) \right\} \right) \leq 2 H \tunnelextent{\tau}
  \end{equation*}
  where $H = H(1,1)$.
\end{proposition}

\begin{proof}
  Let $\omega\in\module{M}$ with $\CDN(\omega)\leq 1$ and
  $\varphi \in\StateSpace(\A)$. By definition of the extent of $\tau$,
  there exists $\psi \in \StateSpace(\B)$ such that
  $\Kantorovich{\Lip_\D}(\varphi,\psi) \leq \tunnelextent{\tau}$. Let
  $\xi \in \module{P}$ with $\CDN(\xi)\leq 1$ such that
  $\Theta_\A(\xi) = \omega$ (which exists by Remark (\ref{reached-T-norm-rmk})), and write $\eta = \Theta_\B(\xi)$, so that $\eta\in\targetsettunnel{\tau}{\omega}{1}$.

  We now compute the distance
  $\KantorovichAlt{\CDN}(\varphi\odot\omega,\psi\odot\eta)$. Let
  $\zeta\in\module{P}$ with $\CDN(\zeta)\leq 1$. We note first that:
  \begin{align*}
    \Lip_\A\left(\Re\inner{\omega}{\Theta_\A(\zeta)}{\module{M}}\right)
    &= \Lip_\A\left(\Re\inner{\Theta_\A(\xi)}{\Theta_\A(\zeta)}{\module{M}}  \right) \\
    &= \Lip_\A\circ\theta_\A(\Re\inner{\xi}{\zeta}{\module{M}}) \\
    &\leq \Lip_\D(\Re\inner{\xi}{\zeta}{\module{M}}) \leq H (\CDN(\xi),\CDN(\zeta)) \leq H\text{.}
  \end{align*}
  Similarly $\Lip_\A\left(\Im\inner{\omega}{\Theta_\A(\zeta)}{\module{M}}\right) \leq H(\CDN(\xi),\CDN(\zeta))$, and also
  \begin{equation*}
    \max\left\{\Lip_\B\left(\Re\inner{\eta}{\Theta_\B(\zeta)}{\module{N}}\right), \Lip_\B\left(\Im\inner{\eta}{\Theta_\B(\zeta)}{\module{N}}\right) \right\} \leq H\text.
  \end{equation*}
  Therefore:
  \begin{align*}
    |\varphi\odot\omega(\Theta_\A(\zeta)) - \psi\odot\eta(\Theta_\B(\zeta))|
    &=|\varphi(\inner{\omega}{\Theta_\A(\zeta)}{}) - \psi(\inner{\eta}{\Theta_\B(\zeta)}{})| \\
    &\leq|\varphi(\Re\inner{\omega}{\Theta_\A(\zeta)}{}) - \psi(\Re\inner{\eta}{\Theta_\B(\zeta)}{})| \\
    &\quad + |\varphi(\Im \inner{\omega}{\Theta_\A(\zeta)}{}) - \psi(\Im\inner{\eta}{\Theta_\B(\zeta)}{})| \\
    &\leq 2 H \Kantorovich{\Lip_\D}(\varphi,\psi) \leq 2 H \tunnelextent{\tau} \text{.}
  \end{align*}

  Therefore,
  $\KantorovichAlt{\CDN}(\varphi\odot\omega,\psi\odot\eta) \leq 2 H
  \tunnelextent{\tau}$ as desired. This computation is symmetric in
  $\mathds{A}$ and $\mathds{B}$, so our proposition is now proven.
\end{proof}

We note, in passing, that the Hausdorff distance in Proposition (\ref{reach-prop}) is, in fact, taken between two subsets of pseudo-states.
\begin{remark}\label{Theta-pseudo-state-rmk}
  Let $\mathds{M}_1=(\module{M}_1,\CDN_1,\A,\Lip_\A)$ and $\mathds{M}_2=(\module{M}_2,\CDN_2,\B,\Lip_\B)$ be two {\gQVB s}. Let $(\Pi,\pi)$ be a modular quantum isometry from $\mathds{M}_1$ onto $\mathds{M}_2$.
  Let $\mu = \varphi\odot\omega$ with $\omega\in\module{M}_2$ and $\varphi$ a continuous linear functional over $\B$. Since $\Pi$ is a surjection from $\module{M}_1$ onto $\module{M}_2$, there exists $\eta\in\module{M}_1$ such that $\Pi(\eta) = \omega$. Therefore:
\begin{align*}
  \varphi\odot\omega\circ\Pi
  &= \varphi(\inner{\omega}{\Pi(\cdot)}{\module{M}_2})
  = \varphi(\inner{\Pi(\eta)}{\Pi(\cdot)}{\module{M}_2}) \\
  &= \varphi(\pi(\inner{\eta}{\cdot}{\module{M}_1}))
  = (\varphi\circ\pi)\odot\eta \text,
\end{align*}
noting that, since $\pi$ is a unital *-morphism and thus continuous and linear, $\varphi\circ\pi$ is a continuous linear functional of $\A$. So $\varphi\odot\omega\circ\Pi$ lies in $\FalseDual{\module{M}_1}$. In short, if $\mu \in \FalseDual{\module{M}_2}$ then $\mu\circ\Pi \in \FalseDual{\module{M}_1}$. 

 In addition, we record by Remark (\ref{reached-T-norm-rmk}), that if $\omega\in\dom{\CDN_j}$ then we can choose $\eta$ such that $\CDN(\eta) = \CDN_j(\omega)$. Keeping the notation above, we also note that $\varphi\circ\pi$ is a state of $\A$ if $\varphi$ is a state of $\B$ since $\pi$ is a unital *-morphism. So, if $\mu\in\ModStateSpace(\module{M}_2)$ then $\mu\circ\Pi \in\ModStateSpace(\module{M}_1)$.
\end{remark}

While the expression in our previous proposition seems redundant, it
takes a new importance when trying to define a covariant version of
the dual-modular propinquity, following our ideas from
\cite{Latremoliere18b} by using the expression given in Proposition
(\ref{reach-prop}), known as the reach in \cite{Latremoliere13b}. We
now define covariant tunnels between covariant modular systems. We
emphasize that we do not require any monoid action on the elements of
the tunnels themselves: our covariant tunnels are build by bringing
together modular tunnels and local almost isometries, with one small
additional condition.

\begin{definition}\label{covariant-tunnel-def}
  Let $(F,F_{\mathsf{inner}})$ be a permissible pair.  Let
  \begin{equation*}
    \mathds{M}_j = \CMS{\module{M}_j}{\CDN_j}{\beta_j}{(G_j,\delta_{G_j},q_j)}{\A_j}{\Lip_j}{\alpha_j}{(H_j,\delta_{H_j})}\text{,}
  \end{equation*}
  for each $j\in\{1,2\}$, be a covariant modular $(F,F_{\mathsf{inner}})$-system.

  An \emph{$\varepsilon$-covariant tunnel}
  $\tau =
  (\mathds{P},(\Theta_1,\theta_1),(\Theta_2,\theta_2),\varsigma,\varkappa)$
  from $\mathds{M}_1$ to $\mathds{M}_2$, for some $\varepsilon> 0$, is given by:
  \begin{enumerate}
  \item a {\QVB{F}} $\mathds{P} = (\module{P},\CDN,\D,\Lip_\D)$,
  \item for each $j\in\{1,2\}$, a quantum module isometry
    $\Theta_j : \mathds{P} \twoheadrightarrow \mathds{M}_j$,
  \item a local almost isometry
    $\varsigma = (\varsigma_1,\varsigma_2) \in
    \UIso{\varepsilon}{G_1}{G_2}{\frac{1}{\varepsilon}}$,
  \item a local almost isometry
    $\varkappa = (\varkappa_1,\varkappa_2) \in
    \UIso{\varepsilon}{H_1}{H_2}{\frac{1}{\varepsilon}}$,
  \end{enumerate}
 such that, for all $\{j,k\} = \{1,2\}$, $\varkappa_j\circ q_j = q_k\circ\varsigma_j$ on $G_j\left[\frac{1}{\varepsilon}\right]$.
\end{definition}

The covariant reach of a covariant tunnel is inspired by Proposition
(\ref{reach-prop}), and includes the actions and local almost
isometries. For reference and comparison, we also include the reach of
a covariant tunnel following \cite{Latremoliere18b}.

As our notation involves a lot of data, we will take the liberty to
invoke the notations used in Definition ({\ref{covariant-tunnel-def}})
repeatedly below.

\begin{definition}[{\cite{Latremoliere18b}}]
  We use the notations of Definition (\ref{covariant-tunnel-def}). The
  \emph{$\varepsilon$-covariant reach}
  $\tunnelreach{\tau}{\varepsilon}$ of $\tau$ is:
  \begin{equation*}
    \max_{\{j,k\} = \{1,2\}} \sup_{\mu\in\StateSpace(\A_j)}\inf_{\nu\in\StateSpace(\A_k)} \left[  \sup_{g\in H_j\left[\frac{1}{\varepsilon}\right]} \Kantorovich{\Lip_\D}(\mu\circ\alpha_j^g\circ\theta_j,\nu\circ\alpha_k^{\varkappa_j(g)}\circ\theta_k) \right] \text{.}
  \end{equation*}
\end{definition}

Our new definition now reads:

\begin{definition}\label{covariant-reach-def}
  We use the notations of Definition (\ref{covariant-tunnel-def}). The
  \emph{$\varepsilon$-modular covariant reach}
  $\tunnelmodreach{\tau}{\varepsilon}$ of $\tau$ is:
  \begin{equation*}
    \max_{\{j,k\} = \{1,2\}} \sup_{\mu\in\ModStateSpace(\module{M}_j)}\inf_{\nu\in\ModStateSpace(\module{M}_k)} \left[  \sup_{g\in G_j\left[\frac{1}{\varepsilon}\right]} \KantorovichAlt{\CDN}(\mu\circ\beta_j^g\circ\Theta_j,\nu\circ\beta_k^{\varsigma_j(g)}\circ\Theta_k) \right] \text{.}
  \end{equation*}
\end{definition}

\begin{remark}
We continue using the notation of Definition (\ref{covariant-tunnel-def}). Let $\{j,k\} = \{1,2\}$. If $\mu\in\ModStateSpace(\module{M}_j)$, and $g \in G_j$, and if $\beta_j^g$ is adjoinable, then $\mu=\varphi\odot\omega$ for some $\varphi\in\StateSpace(\A_j)$ and $\omega\in\dom{\CDN_j}$, and therefore,
\begin{equation*}
  \mu\circ\beta_j^g = \varphi(\inner{\omega}{\beta_j^g(\cdot)}{\module{M}_j}) = \varphi(\inner{(\beta_1^g)^\ast\omega}{\cdot}{\module{M}_j}) = \varphi\odot((\beta_j^g)^\ast\omega) \text.
\end{equation*}
If $\omega\in\dom{\CDN_j}$, then $\CDN_j(\beta_j^g\omega)\leq K(g) \CDN_j(\omega)$ for some constant $K(g)$, so $\mu\circ\beta_j^g$ is then as scalar multiple of a pseudo-state of $\module{M}_j$. Therefore, $\mu\circ\beta_j^g\circ\Theta_j$ is a scalar multiple of a pseudo-state of $\module{P}$.
\end{remark}

We now follow the pattern identified in \cite{Latremoliere18b} and
synthesize our various numerical quantities attached to a covariant
modular tunnel into a single number:

\begin{definition}\label{magnitude-def}
  We use the notations of Definition (\ref{covariant-tunnel-def}). The
  $\varepsilon$-modular magnitude of $\tau$ is:
  \begin{equation*}
    \tunnelmodmagnitude{\tau}{\varepsilon} = \max\left\{ \tunnelextent{\tau}, \tunnelreach{\tau}{\varepsilon}, \tunnelmodreach{\tau}{\varepsilon} \right\} \text{.}
  \end{equation*}
\end{definition}

\begin{remark}
  To any modular covariant tunnel corresponds a covariant tunnel
  between the underlying Lipschitz dynamical systems formed by the
  base spaces, and the modular magnitude dominates the magnitude of
  this tunnel. Using the notations of Definition
  (\ref{covariant-tunnel-def}), this covariant tunnel is simply
  $\tau_\flat(\D,\Lip_\D,\theta_1,\theta_2,\varkappa)$, and by
  construction,
  $\tunnelmagnitude{\tau_\flat}{\varepsilon} \leq
  \tunnelmodmagnitude{\tau}{\varepsilon}$.
\end{remark}

We verify that, roughly speaking, covariant tunnels can be composed, which is the reason why, ultimately, our covariant modular metric will satisfy the
triangle inequality. The proof follows the idea of \cite{Latremoliere14}.

\begin{theorem}\label{triangle-thm}
  Let $(F,F_{\mathsf{inner}})$ be a permissible pair. Let
  $\varepsilon_1, \varepsilon_2 \in
  \left(0,\frac{\sqrt{2}}{2}\right)$. Let $\mathds{M}_1$,
  $\mathds{M}_2$ and $\mathds{M}_3$ be three covariant modular
  $(F,F_{\mathsf{inner}})$-systems. Let $\tau_1$ be $\varepsilon_1$-covariant tunnel
  from $\mathds{M}_1$ to $\mathds{M}_2$ and let $\tau_2$ be a
  $\varepsilon_2$-covariant tunnel from $\mathds{M}_2$ to
  $\mathds{M}_3$.

  For all $\varepsilon > 0$, there exists a
  $(\varepsilon_1 + \varepsilon_2)$-covariant $(F,F_{\mathsf{inner}})$-tunnel $\tau$
  from $\mathds{M}_1$ to $\mathds{M}_3$ with:
  \begin{equation*}
    \tunnelmodmagnitude{\tau}{\varepsilon_1 + \varepsilon_2} \leq \tunnelmodmagnitude{\tau_1}{\varepsilon_1}+\tunnelmodmagnitude{\tau_2}{\varepsilon_2} + \varepsilon \text{.}
  \end{equation*}
\end{theorem}

\begin{proof}
  Let $\varepsilon_1,\varepsilon_2 \in \left(0,\frac{\sqrt{2}}{2}\right)$ and let $\varepsilon > 0$. Let
  $\mathds{M}_j =
  \CMS{\module{M}_j}{\CDN_j}{\beta_j}{(G_j,\delta_j,q_j)}{\A_j}{\Lip_j}{\alpha_j}{(H_j,d_j)}$
  for each $j\in\{1,2,3\}$. Let $\tau_1$ be a
  $\varepsilon_1$-covariant tunnel from $\mathds{M}_1$ to $\mathds{M}_2$ with:
  \begin{equation*}
    \tau_1 = \left( \mathds{P}_1, (\Theta_1,\theta_1), (\Theta_2,\theta_2), \varsigma_1, \varkappa_1 \right)
  \end{equation*}
  with $\mathds{P}_1 = (\module{P}_1,\TDN_1,\D_1,\TLip_1)$ and let
  $\tau_2$ be a $\varepsilon_2$-covariant tunnel from $\mathds{M}_2$
  to $\mathds{M}_3$ with:
  \begin{equation*}
    \tau_2 = \left( \mathds{P}_2, (\Pi_1,\pi_1), (\Pi_2,\pi_2), \varsigma_2, \varkappa_2 \right)
  \end{equation*}
  with $\mathds{P}_2 = (\module{P}_2,\TDN_2,\D_2,\TLip_2)$.

  By definition, and to fix our notation:
  \begin{itemize}
  \item $\varsigma_1 = (\varsigma_1^1,\varsigma_1^2) \in \UIso{\varepsilon_1}{G_1}{G_2}{\frac{1}{\varepsilon_1}}$,
  \item $\varsigma_2 = (\varsigma_2^1,\varsigma_2^2) \in \UIso{\varepsilon_2}{G_2}{G_3}{\frac{1}{\varepsilon_2}}$,
  \item $\varkappa_1 = (\varkappa_1^1,\varkappa_1^2) \in \UIso{\varepsilon_1}{H_1}{H_2}{\frac{1}{\varepsilon_1}}$,
  \item $\varkappa_2 = (\varkappa_2^1,\varkappa_2^2) \in \UIso{\varepsilon_2}{H_2}{H_3}{\frac{1}{\varepsilon_2}}$,
  \end{itemize}
  
  Let:
  \begin{equation*}
    \mathds{P} = \left(\module{P}_1\oplus\module{P}_2,\CDN,\D_1\oplus\D_2,\Lip\right)
  \end{equation*}
  where, for all $(d_1,d_2) \in \sa{\D_1\oplus\D_2} = \sa{\D_1}\oplus\sa{\D_2}$, we set
  \begin{equation*}
    \Lip(d_1,d_2) = \max\left\{ \TLip_1(d_1),\TLip_2(d_2),\frac{1}{\varepsilon} \norm{\theta_2(d_1) - \pi_1(d_2)}{\A_2} \right\}
  \end{equation*}
  and, for all
  $(\omega_1,\omega_2) \in \module{P}_1\oplus\module{P}_2$, we
  similarly set
  \begin{equation*}
    \CDN(\omega_1,\omega_2) = \max\left\{ \TDN_1(\omega_1), \TDN_2(\omega_2), \frac{1}{\varepsilon}\norm{\Theta_2(\omega_1) - \Pi_1(\omega_2)}{\module{M}_2}  \}{} \right\}\text{.}
  \end{equation*}
  Let
  $\Xi_1 : (\omega_1,\omega_2)\in\module{P}_1\oplus\module{P}_2
  \mapsto \Theta_1(\omega_1)$, and
  $\Xi_2 : (\omega_1,\omega_2)\in\module{P}_1\oplus\module{P}_2
  \mapsto \Pi_2(\omega_2)$. Similarly, let
  $\xi_1 : (d_1,d_2) \in \D_1\oplus\D_2\mapsto \theta_1(d_1)$ and
  $\xi_2 : (d_1,d_2) \in \D_1\oplus\D_2 \mapsto \pi_2(d_2)$.

  By \cite{Latremoliere18d}, we conclude that
  $(\mathds{P},(\Xi_1,\xi_1),(\Xi_2,\xi_2))$ is a modular
  $(F,F_{\mathsf{inner}})$-tunnel from $(\module{M}_1,\CDN_1,\A_1,\Lip_1)$ to
  $(\module{M}_3,\CDN_3,\A_3,\Lip_3)$ of extent at most
  $\tunnelextent{\tau_1} + \tunnelextent{\tau_2} + \varepsilon$.

  Using Theorem (\ref{composition-thm}), we also have:
  \begin{equation*}
    \varsigma \coloneqq (\varsigma_2^1\circ\varsigma_1^1,\varsigma_1^2\circ\varsigma_2^2) \in \UIso{\varepsilon_1 + \varepsilon_2}{G_1}{G_3}{\frac{1}{\varepsilon_1 + \varepsilon_2}} \text,
  \end{equation*}
  and
  \begin{equation*}
    \varkappa \coloneqq (\varkappa_2^1\circ\varkappa_1^1,\varkappa_1^2\circ\varkappa_2^2) \in \UIso{\varepsilon_1 + \varepsilon_2}{H_1}{H_3}{\frac{1}{\varepsilon_1 + \varepsilon_2}} \text.
  \end{equation*}

  Let now $g \in H_1\left[\frac{1}{\varepsilon_1+\varepsilon_2}\right]$. By Definition (\ref{metrical-tunnel-def}), we have
  \begin{equation*}
    q_3\circ\varsigma_2^1\circ\varsigma_1^1 = \varkappa_2^1\circ q_2 \circ \varsigma_1^1 = \varkappa_2^1\circ\varkappa_1^1\circ q_1 \text.
  \end{equation*}
  A similar computation holds with $G_1$ and $G_3$ roles switched.
  
  Therefore,
  $\tau =
  (\mathds{P},(\Xi_1,\xi_1),(\Xi_2,\xi_2),(\varsigma,\varkappa))$ is
  an $\varepsilon_1+\varepsilon_2$ covariant tunnel. Moreover, by
  \cite[Proposition 3.14]{Latremoliere18b}, we also know that
  $(\D_1\oplus\D_2,\Lip,\xi_1,\xi_2,\varsigma,\varkappa)$ is an
  $(\varepsilon_1+\varepsilon_2)$-covariant tunnel from
  $(\A_1,\Lip_1)$ to $(\A_3,\Lip_3)$ with:
  \begin{align*}
    \tunnelmagnitude{(\D_1\oplus\D_2,\Lip,\xi_1,\xi_2,\varsigma,\varkappa)}{\varepsilon_1+\varepsilon_2}
    &\leq \tunnelmagnitude{(\D_1,\TLip_1,\theta_1,\theta_2)}{\varepsilon_1} \\
    &\quad+ \tunnelmagnitude{(\D_2,\TLip_2,\pi_1,\pi_2)}{\varepsilon_2} + \varepsilon \\
    &\leq \tunnelmodmagnitude{\tau_1}{\varepsilon_1} + \tunnelmodmagnitude{\tau_2}{\varepsilon_2} + \varepsilon \text.
  \end{align*}

  We conclude by computing the $(\varepsilon_1+\varepsilon_2)$-modular
  reach of $\tau$. Let now $\mu \in \ModStateSpace(\module{M}_1)$. By
  definition of the modular reach, there exists
  $\nu \in \ModStateSpace(\module{M}_2)$ such that, for all
  $g \in G\left[\frac{1}{\varepsilon_1}\right]$, we have
  $\KantorovichAlt{\TDN_1}(\mu\circ\beta_1^g\circ\Theta_1,\nu\circ\beta_2^{\varsigma^1_1(g)}\circ\Theta_2)
  \leq \tunnelmodreach{\tau_1}{\varepsilon_1}$. Similarly, there
  exists $\eta\in\ModStateSpace(\module{M}_3)$ such that, if
  $g \in G\left[\frac{1}{\varepsilon_2}\right]$, then
  \begin{equation*}
    \KantorovichAlt{\TDN_2}(\nu\circ\beta_2^{g}\circ\Pi_1,\eta\circ\beta_3^{\varsigma^2_1(g)}\circ\Pi_2)\leq \tunnelmodreach{\tau_2}{\varepsilon_2}\text.
  \end{equation*}

  Now, let
  $\zeta \coloneqq (\zeta_1,\zeta_2) \in
  \module{P}_1\oplus\module{P}_2$ with $\CDN(\zeta)\leq 1$. In
  particular, $\TDN_1(\zeta_1)\leq 1$ and $\TDN_2(\zeta_2)\leq
  1$. Moreover,
  $\norm{\Theta_2(\zeta_1) - \Pi_1(\zeta_2)}{\module{M}_2} <
  \varepsilon$.

  Let $g \in G\left[\frac{1}{\varepsilon_1+\varepsilon_2}\right]$. As
  shown in \cite[Lemma 2.11]{Latremoliere18b}, we have
  $\varepsilon_1 + \frac{1}{\varepsilon_1+\varepsilon_2} \leq
  \frac{1}{\varepsilon_2}$.
  
  Then, in particular, $\varsigma^1_1(g)\in G_2\left[\frac{1}{\varepsilon_1+\varepsilon_2} + \varepsilon_1\right]\subseteq G\left[\frac{1}{\varepsilon_2}\right]$. We thus conclude, writing $\varsigma = (\varsigma_1,\varsigma_2)$, that:
  \begin{align*}
    \Big|\mu&\circ\beta_1^g\circ\Xi_1(\zeta_1,\zeta_2) - \eta\circ\beta_3^{\varsigma_2^1\circ\varsigma_1^1(g)}\circ\Xi_2(\zeta_1,\zeta_2)\Big| \\
            &= \left|\mu\circ\beta_1^g\circ\Theta_1(\zeta_1) - \eta\circ\beta_3^{\varsigma_2^1\circ\varsigma_1^1(g)}\circ\Pi_2(\zeta_2)\right| \\
            &\leq \left|\mu\circ\beta_1^g\circ\Theta_1(\zeta_1) - \nu\circ\beta_2^{\varsigma^1_1(g)}(\Theta_2(\zeta_1))\right| + \left|\nu\circ\beta_2^{\varsigma^1_1(g)}(\Theta_2(\zeta_1)) - \nu\circ\beta_2^{\varsigma^1_1(g)}(\Pi_1(\zeta_2))\right|\\
            &\quad + \left|\nu\circ\beta_2^{\varsigma^1_1(g)}(\Pi_1(\zeta_2)) - \eta\circ\beta_3^{\varsigma^2_1\circ\varsigma^1_1(g)}\circ\Pi_2(\zeta_2)\right| \\ 
            &\leq \KantorovichAlt{\TDN_1}(\mu\circ\beta_1^g\circ\Theta_1,\nu\circ\beta_2^{\varsigma^1_1(g)}\circ\Theta_2) + \norm{\Theta_2(\zeta_1) - \Pi_1(\zeta_2)}{\module{M}_2} \\
            &\quad + \KantorovichAlt{\TDN_2}(\nu\circ\beta_2^{\varsigma^1_1(g)}\circ\Pi_1,\eta\circ\beta_3^{\varsigma^2_1(\varsigma^1_1(g))}\circ\Pi_2) \\
            &\leq \tunnelmodreach{\tau_1}{\varepsilon_1} + \varepsilon + \tunnelmodreach{\varepsilon_2}{\varepsilon_2} 
              \leq \tunnelmodmagnitude{\tau_1}{\varepsilon_1} + \tunnelmodmagnitude{\tau_2}{\varepsilon_2} + \varepsilon \text{.}
  \end{align*}

  A similar computation can be done switching the roles of $G_1$ and $G_3$.
  
  Therefore, the $(\varepsilon_1+\varepsilon_2)$-covariant reach of
  $\tau$ is bounded above by $\tunnelmodmagnitude{\tau_1}{\varepsilon_1} +
  \tunnelmodmagnitude{\tau_2}{\varepsilon_2} +
  \varepsilon$. Altogether, by Definition (\ref{magnitude-def}), we
  thus have shown that $\tau$ is a
  $(\varepsilon_1 + \varepsilon_2)$-covariant tunnel with:
  \begin{equation*}
    \tunnelmodmagnitude{\tau}{\varepsilon_1+\varepsilon_2} \leq \tunnelmodmagnitude{\tau_1}{\varepsilon_1}  + \tunnelmodmagnitude{\tau_2}{\varepsilon_2} + \varepsilon
  \end{equation*}
  as desired.
\end{proof}

We now have the tools to define the covariant modular propinquity.

\begin{notation}
  For any permissible pair $(F,F_{\mathsf{inner}})$, for any $\varepsilon > 0$, and any two covariant modular
  systems $\mathds{A}$ and $\mathds{B}$, the set of all
  $\varepsilon$-covariant $(F,F_{\mathsf{inner}})$-tunnels from $\mathds{A}$ to
  $\mathds{B}$ is denoted by:
  \begin{equation*}
    \tunnelset{\mathds{A}}{\mathds{B}}{F,F_{\mathsf{inner}}}{\varepsilon} \text{.}
  \end{equation*}
\end{notation}

\begin{definition}\label{modcovprop-def}
  Fix a permissible pair $(F,F_{\mathsf{inner}})$. The \emph{covariant modular
    propinquity} between any two covariant modular $(F,F_{\mathsf{inner}})$-systems
  $\mathds{A}$ and $\mathds{B}$ is the nonnegative number:
  \begin{multline*}
    \dcovmodpropinquity{F}\left( \mathds{A}, \mathds{B} \right) \\
    = \min\left\{ \frac{\sqrt{2}}{2}, \inf\left\{\varepsilon > 0 :
        \exists\tau \in
        \tunnelset{\mathds{A}}{\mathds{B}}{F,F_{\mathsf{inner}}}{\varepsilon} \quad
        \tunnelmodmagnitude{\tau}{\varepsilon} \leq \varepsilon
      \right\} \right\} \text{.}
  \end{multline*}
\end{definition}

We record that the covariant modular propinquity is indeed a
pseudo-metric:
\begin{proposition}
  For any permissible pair $(F,F_{\mathsf{inner}})$, the covariant modular propinquity is a pseudo-metric on the class of covariant modular
  $(F,F_{\mathsf{inner}})$-systems.
\end{proposition}

\begin{proof}
  The proof that the covariant modular propinquity satisfies the
  triangle inequality is now identical to \cite{Latremoliere18b} with
  the use of Theorem (\ref{triangle-thm}) (it is helpful to point out that if $\tau$ is an $\varepsilon$-covariant tunnel of magnitude at most $m$, then $\tau$ is also a $(\varepsilon+t)$-covariant tunnel of magnitude at most $m$, for any $t\geq 0$, by definition).
  
  We also note that if $(\mathcal{P},\Theta,\Pi,\varsigma,\varkappa)$ is a
  $\varepsilon$-covariant modular tunnel, then so is
  $(\mathcal{P},\Pi,\Theta,\varkappa,\varsigma)$, and these two
  tunnels have the same $\varepsilon$-magnitude. So the covariant
  modular propinquity is symmetric as well.

  Last, if
  $\mathds{M} =
  \CMS{\module{M}}{\CDN}{\beta}{(G,\delta,q)}{\A}{\Lip}{\alpha}{(H,d)}$
  is an $(F,F_{\mathsf{inner}})$ covariant modular system, then
  \begin{equation*}
    \mathds{M} = \left(\mathds{M},(\mathrm{id}_{\module{M}},\mathrm{id}_\A),(\mathrm{id}_{\module{M}},\mathrm{id}_\A),(\mathrm{id}_G,\mathrm{id}_G),(\mathrm{id}_H,\mathrm{id}_H)\right)
  \end{equation*}
  where $\mathrm{id}_X$ is the identity map of $X$, is a
  $\varepsilon$-covariant modular tunnel of $\varepsilon$-magnitude 0,
  for all $\varepsilon > 0$. Thus
  $\dcovmodpropinquity{F}(\mathds{M},\mathds{M}) = 0$. This
  concludes our proof.
\end{proof}

We now check that we have indeed defined a metric up to the following
notion of equivalence.

\begin{definition}
  For each $j\in\{1,2\}$, let
  \begin{equation*}
    \mathds{M}_j =
    \CMS{\module{M}_j}{\CDN_j}{\beta_j}{(G_j,\delta_j,q_j)}{\A_j}{\Lip_j}{\alpha_j}{(H_j,d_j)}
  \end{equation*}
  be a covariant modular $(F,F_{\mathsf{inner}})$-system.

  A \emph{full equivariant modular quantum isometry}
  $(\Pi,\pi,\varsigma,\varkappa)$ from $\mathds{M}_1$ to
  $\mathds{M}_2$ is given by a full modular quantum isometry
  $(\Pi,\pi)$ from $(\module{M}_1,\CDN_1,\A_1,\Lip_1)$ to
  $(\module{M}_2,\CDN_2,\A_2,\Lip_2)$, an isometric monoid isomorphism
  $\varsigma : G_1 \rightarrow G_2$ and an isometric proper monoid
  isomorphism $\varkappa : H_1 \rightarrow H_2$ such that:
  \begin{enumerate}
  \item $\varkappa\circ q_1 = q_2 \circ \varsigma$,
  \item for all $h \in H_1$, we have
    $\pi\circ\alpha_1^h = \alpha_2^{\varkappa(h)}\circ\pi$,
  \item for all $g \in G_1$, we have
    $\Pi\circ\beta_1^g = \beta_1^{\varsigma(g)}\circ\Pi$.
  \end{enumerate}
\end{definition}

The study of convergence for modules seems to benefit
\cite{Latremoliere16c,Latremoliere18d} from the introduction of the
following distance on modules, which is naturally related to the
distance of Notation (\ref{MongeKantAlt-def}):
\begin{proposition}[{\cite[Definition 3.24, Proposition
    3.25]{Latremoliere16c}}]\label{MongeKantMod-def}
  Let $(\module{M},\CDN,\A,\Lip)$ be a {\gQVB}.  For
  $\omega,\eta \in \module{M}$, if we set:
  \begin{equation*}
    \KantorovichMod{\CDN}(\omega,\eta) = \sup\left\{ \norm{\inner{\omega-\eta}{\zeta}{\module{M}}}{\A}: \zeta\in\module{M}, \CDN(\zeta)\leq 1 \right\} \text{,}
  \end{equation*}
  then $\KantorovichMod{\CDN}$ is a metric on $\module{M}$ which, on
  bounded subsets of $\module{M}$, induces the $\A$-weak topology,
  i.e. the locally convex topology induced on $\module{M}$ by the
  family of seminorms:
  \begin{equation*}
    \forall \zeta \in \module{M} \quad \omega\in\module{M}\longmapsto\norm{\inner{\omega}{\zeta}{\module{M}}}{\A} \text.
  \end{equation*}
  Moreover, if $B\subseteq \module{M}$ is bounded for the D-norm
  $\CDN$, then the topology induced by $\KantorovichMod{\CDN}$ and by
  $\norm{\cdot}{\module{M}}$ on $B$ are equal.
\end{proposition}

We now have two analogues of the {\MongeKant} for {\gQVB s}, and we
will understand their relationship during this section. We first
observe that the metric introduced in Proposition
(\ref{MongeKantMod-def}) over a metrical C*-correspondence is indeed related to the set of pseudo-states of the underlying module.
\begin{proposition}\label{equiv-prop}
  Let $(\module{M},\CDN,\A,\Lip)$ be a {\gQVB}. If
  $\omega,\eta\in\module{M}$ then:
  \begin{equation*}
    \sup\left\{ \left|\mu(\omega-\eta)\right| : \mu \in \ModStateSpace(\module{M}) \right\} \leq \KantorovichMod{\CDN}(\omega,\eta) \leq 2 \sup\left\{ \left|\mu(\omega-\eta)\right| : \mu \in \ModStateSpace(\module{M}) \right\} \text{.}
  \end{equation*}
\end{proposition}

\begin{proof}
  First, if $\mu = \varphi\odot\zeta\in\ModStateSpace(\module{M})$, with
  $\varphi\in\StateSpace(\A)$ and $\CDN(\zeta)\leq 1$, then we
  compute:
  \begin{equation*}
    \left| \mu(\omega-\eta) \right| = \left|\varphi(\inner{\zeta}{\omega-\eta}{})\right| \leq \norm{\inner{\zeta}{\omega-\eta}{}}{\A} \leq \KantorovichMod{\CDN}(\omega,\eta) \text{.}
  \end{equation*}

  On the other hand, if $b \in \A$, then (noting
  $\varphi(b^\ast) = \overline{\varphi(b)}$):
  \begin{align*}
    \norm{b}{\A}
    &\leq \norm{\Re b}{\A} + \norm{\Im b}{\A} \\
    &\leq \sup_{\varphi\in\StateSpace(\A)} |\varphi(\Re b)| + \sup_{\varphi\in\StateSpace(\A)} |\varphi(\Im b)| \\
    &\leq 2\sup_{\varphi\in\StateSpace(\A)} |\varphi(b)| \text.
  \end{align*}

  Therefore, for all $\zeta\in\dom{\CDN}$ with $\CDN(\zeta)\leq 1$, we
  conclude
  \begin{equation*}
    \norm{\inner{\zeta}{\omega-\eta}{\module{M}}}{\A} \leq 2 \sup\{|\mu(\omega-\eta)|:\mu\in\ModStateSpace(\module{M})\}\text.
  \end{equation*}
  
  Our proposition follows.
\end{proof}

\begin{remark}
  \cite[Proposition 3.20]{Latremoliere18d} has a minor typo, where all
  occurrences of $\sqrt{2}$ should be a $2$.
\end{remark}

We now conclude that our covariant modular propinquity is indeed a
metric, up to a fully equivariant modular quantum isometry.

\begin{theorem}
  Let $(F,F_{\mathsf{inner}})$ be a permissible pair and let $\mathds{A}$ and
  $\mathds{B}$ be two covariant modular $(F,F_{\mathsf{inner}})$-systems. Then $\dcovmodpropinquity{F}(\mathds{A},\mathds{B}) = 0$ if and only if there exists a full equivariant modular quantum isometry from $\mathds{A}$ to $\mathds{B}$.
\end{theorem}

\begin{proof}
  We need some notation. We write:
  \begin{align*}
    \mathds{A} &= \CMS{\module{M}}{\CDN_\A}{\beta_\A}{(G_\A,\delta_\A,q_\A)}{\A}{\Lip_\A}{\alpha_\A}{(H_\A,d_\A)} \\ \intertext{ and }
    \mathds{B} &= \CMS{\module{N}}{\CDN_\B}{\beta_\B}{(G_\B,\delta_\B,q_\B)}{\B}{\Lip_\B}{\alpha_\B}{(H_\B,d_\B)} \text{.}
  \end{align*}

  Let $K_\A : G_\A \rightarrow [0,\infty)$ and
  $K_\B : G_\B \rightarrow [0,\infty)$ be locally bounded functions
  such that for all $g \in G_\A$, we have
  $\CDN_\A\circ\beta_\A^g \leq K_\A(g)\CDN_\A$ and for all $g\in G_\B$
  we have $\CDN_\B\circ\beta_\B^g\leq K_\B(g)\CDN_\B$.

  By Definition (\ref{modcovprop-def}), for all $n\in\N$, there exists
  a $\frac{1}{n+1}$-covariant modular tunnel
  $(\tau_n,\varsigma_n,\varkappa_n)$ from $\mathds{A}$ to $\mathds{B}$
  with
  $\tunnelmodmagnitude{\tau_n}{\frac{1}{n+1}}\leq\frac{1}{n+1}$. We
  recall from Definition (\ref{covariant-tunnel-def}) that $\tau_n$ is
  a modular tunnel, while:
  \begin{align*}
    \varsigma_n = (\varsigma_n^1,\varsigma_n^2) \in \UIso{\frac{1}{n+1}}{G_\A}{G_\B}{n+1} \intertext{ and }\varkappa_n = (\varkappa_n^1,\varkappa_n^2) \in \UIso{\frac{1}{n+1}}{H_\A}{H_\B}{n+1}\text{.}
  \end{align*}

  We also write, for each $n\in\N$, the tunnel $\tau_n$ as
  $(\mathds{M}_n,(\Theta_n,\theta_n),(\Theta'_n,\theta'_n))$, where
  $\mathds{M}_n=(\module{P}_n,\TDN_n,\D_n,\TLip_n)$ is a {\gQVB},
  $(\Theta_n,\theta_n)$ is a modular quantum isometry from
  $\mathds{M}_n$ onto $(\module{M},\CDN_\A,\A,\Lip_\A)$, and
  $(\Theta'_n,\theta'_n)$ is a modular quantum isometry from
  $\mathds{M}_n$ onto $(\module{N},\CDN_\B,\B,\Lip_\B)$.

  By \cite[Theorem 3.23]{Latremoliere18b}, there exists a full modular quantum
  isometry $(\Pi,\pi)$ from $(\module{M},\CDN_\A,\A,\Lip_\A)$ to
  $(\module{N},\CDN_\B,\B,\Lip_\B)$, and a strictly increasing
  function $f : \N \rightarrow \N$, such that:
  \begin{enumerate}
  \item for all $a\in\dom{\Lip_\A}$ and $l\geq\Lip_\A(a)$, the
    sequence
    $\left(\targetsettunnel{\tau_{f(n)}}{a}{l}\right)_{n\in\N}$
    converges to $\{ \pi(a) \}$ for the Hausdorff distance
    $\Haus{\norm{\cdot}{\B}}$,
  \item for all $b\in\dom{\Lip_\B}$ and $l\geq\Lip_\B(b)$, the
    sequence
    $\left(\targetsettunnel{\tau_{f(n)}^{-1}}{b}{l}\right)_{n\in\N}$
    converges to $\{ \pi^{-1}(b) \}$ for the Hausdorff distance
    $\Haus{\norm{\cdot}{\A}}$,
  \item for all $\omega\in\dom{\CDN_\A}$ and $l\geq\CDN_\A(\omega)$,
    the sequence
    $\left(\targetsettunnel{\tau_{f(n)}}{\omega}{l}\right)_{n\in\N}$
    converges to $\{\Pi(\omega)\}$ for the Hausdorff distance
    $\Haus{\KantorovichMod{\CDN_\B}{}}$,
  \item for all $\eta\in\dom{\CDN_\B}$ and $l\geq\CDN_\B(\eta)$, the
    sequence
    $\left(\targetsettunnel{\tau_{f(n)}^{-1}}{\eta}{l}\right)_{n\in\N}$
    converges to $\{\Pi^{-1}(\eta)\}$ for the Hausdorff distance
    $\Haus{\KantorovichMod{\CDN_\A}}$.
  \end{enumerate}
    
  Furthermore, by \cite[Theorem 2.12, Theorem 3.23]{Latremoliere18b} applied to both
  $(\varsigma_n)_{n\in\N}$ and $(\varkappa_n)_{n\in\N}$ (up to
  extracting further subsequences), there exists a strictly increasing
  function $f_2 : \N\rightarrow\N$, an isometric monoid isomorphism
  $\varsigma : G_\A \rightarrow G_\B$ and an isometric monoid
  isomorphism $\varkappa : H_\A \rightarrow H_\B$ such that:
  \begin{itemize}
  \item for all $g \in G_\A$, we have
    $\lim_{n\rightarrow\infty}\varsigma^1_{f(f_2(n))}(g) =
    \varsigma(g)$, and for all $g \in G_\B$ we have
    $\lim_{n\rightarrow\infty}\varsigma^2_{f(f_2(n))}(g) =
    \varsigma^{-1}(g)$,
  \item for all $g \in H_\A$, we have
    $\lim_{n\rightarrow\infty}\varkappa^1_{f(f_2(n))}(g) =
    \varkappa(g)$, and for all $g \in H_\B$ we have
    $\lim_{n\rightarrow\infty}\varkappa^2_{f(f_2(n))}(g) =
    \varkappa^{-1}(g)$.
  \end{itemize}
  Now, the work in \cite[Theorem 3.23]{Latremoliere18b} shows that $\pi$ is, in
  fact, full equivariant, in the sense that for all $g \in H_\A$ we
  have $\pi\circ\alpha_\A^g = \alpha_\B^{\varkappa(g)}\circ\pi$. We
  now prove that the same method can be used here to show that $(\Pi,\pi)$
  is indeed equivariant as well. To ease notation, we rename
  $f\circ f_2$ simply as $f$.
    
  Let $\omega\in\dom{\CDN_\A}$ and $l = \CDN_\A(\omega)$. Let
  $\mu = \varphi \odot \xi \in \ModStateSpace(\module{N})$, where
  $\varphi\in\StateSpace(\B)$ and $\xi\in\dom{\CDN_\B}$ with
  $\CDN_\B(\xi)\leq 1$. Let $g \in G_\B$ and choose $N\in\N$ so that
  $g \in G_\B[N+1]$. To ease our notations, let
  $\varpi = \varsigma^{-1}$, so that
  $(\varsigma_{f(n)}^2(g))_{n\geq N}$ converges to
  $\varpi(g) = \varsigma^{-1}(g) \in G_\A$.
    
  By Definition (\ref{covariant-reach-def}) of the modular reach, for
  each $n\in\N$, $n > N$, there exists 
  $\nu_n = \psi_n\odot\rho_n \in \ModStateSpace(\module{M})$, with
  $\psi_n\in\StateSpace(\A)$ and
  $\rho_n\in\dom{\CDN_\A}$,$\CDN_\A(\rho_n)\leq 1$, such that
  \begin{equation*}
    \forall h \in G_\B[N+1] \quad \KantorovichAlt{\CDN_n}(\nu_n\circ\alpha_\A^{\varsigma_{f(n)}^2(h)}\circ\Theta_n,\mu\circ\alpha_\B^{h}\circ\Theta'_n) \leq \frac{1}{f(n+1)} \leq\frac{1}{n+1}\text.
  \end{equation*}

  Since $(\varsigma^2_{f(n)}(g))_{n\in\N}$ converges to $\varpi(g)$,
  and since $K_\A$ is locally bounded, there exists $K > 0$ and
  $N'\in\N$ such that for all $n\in\N$, $n\geq N'$ we have
  $\CDN_\A\circ\beta_\A^{\varsigma^2_{f(n)}}\leq K \CDN_\A$ and
  $\CDN_\A\circ\beta_\A^{\varpi(g)}\leq K\CDN_\A$.
    
  For each $n\in\N$, $n\geq\max\{N,N'\}$, let:
  \begin{itemize}
  \item $o_n \in \targetsettunnel{\tau_{f(n)}}{\omega}{l}$,
  \item
    $\eta_n \in
    \targetsettunnel{\tau_{f(n)}}{\beta_\A^{\varpi(g)}(\omega)}{K l}$,
  \item
    $\gamma_n \in
    \targetsettunnel{\tau_{f(n)}}{\beta_\A^{\varsigma^2_{f(n)}(g)}(\omega)}{K
      l}$.
  \end{itemize}

  Now:
  \begin{align*}
    \left|\mu\left(\eta_n - \beta_\B^g(o_n)\right)\right|
    &\leq \left|\mu(\eta_n - \gamma_n)\right| + \left|\mu\left(\gamma_n\right) - \nu_n\left(\beta_\A^{\varsigma^2_{f(n)}(g)}(\omega)\right)\right| \\
    &\quad + \left|\nu_n\left(\beta_\A^{\varsigma^2_{f(n)}(g)}(\omega)\right) - \mu(\beta_\B^g(o_n)) \right| \\
    &\leq \KantorovichMod{\CDN_\B}(\eta_n ,\gamma_n) + K l \cdot \KantorovichAlt{\TDN_n}(\mu\circ\Theta_{f(n)},\nu_n\circ\Theta'_{f(n)}) \\
    &\quad+ l \cdot \KantorovichAlt{\TDN_n}(\nu_n\circ\beta_\A^{\varsigma_{f(n)}^2}\circ\Theta_{f(n)},\mu\circ\beta_\A^g\circ\Theta'_{f(n)}) \\
    &\leq \KantorovichMod{\CDN_\B}(\eta_n ,\gamma_n) + \frac{(K+1) l}{n+1}\text{.}
  \end{align*}

  Since $\beta_\A$ is strongly continuous and
  $(\varsigma_{f(n)}^2(g))_{n\in\N}$ converges to $\varpi(g)$, using
  \cite[Proposition 3.20]{Latremoliere18d}:
  \begin{multline*}
    \limsup_{n\rightarrow\infty} \KantorovichMod{\CDN_\B}(\eta_n , \gamma_n) \\ \leq 2\limsup_{n\rightarrow\infty} \left(\KantorovichMod{\CDN_\A}\left(\beta_\A^{\varsigma^2_{f(n)}(g)}(\omega) ,\beta_\A^{\varpi(g)}(\omega)\right) + 2 K F_{\mathsf{mod}}(2l,1) \tunnelextent{\tau_{f(n)}}\right) \\ = 0 \text{.}
  \end{multline*}

  Therefore, by continuity of $\mu$, by continuity of $\beta_\B^g$, and by construction of $\Pi$:
  \begin{equation*}
    \left| \mu(\Pi(\beta_\A^{\varpi(g)}(\omega)) - \beta_\B^{g}\circ\Pi(\omega)) \right| = \lim_{n\rightarrow\infty}\left|\mu\left(\eta_n - \beta_\B^g(o_n)\right)\right| = 0 \text{.}
  \end{equation*}
  Since $\mu\in\ModStateSpace(\module{M})$ is arbitrary, we conclude
  $\KantorovichMod{\CDN_\B}(\Pi(\beta_\A^{\varpi(g)}(\omega)) ,
  \beta_\B^{g}\circ\Pi(\omega)) = 0$, by Proposition
  (\ref{equiv-prop}). Therefore
  $\Pi(\beta_\A^{\varpi(g)}(\omega)) = \beta_\B^{g}\circ\Pi(\omega)$,
  as desired.

  By continuity, since $\dom{\CDN_\A}$ is norm dense in $\module{M}$,
  we conclude that $\Pi\circ\beta_\A^{\varpi(g)} = \beta_\B^g\circ\Pi$
  for all $g\in G_\B$, which is of course equivalent to
  $\Pi\circ\beta_\A^g = \beta_\B^{\varsigma(g)}\circ\Pi$ for all
  $g\in G_\A$.

  Last, by Definition (\ref{covariant-tunnel-def}), we note that $q_\B\circ\varsigma = \varkappa\circ q_\A$, since $q_\A$ and $q_\B$ are continuous.
  Similarly, $q_\A\circ\varsigma^{-1} = \varkappa^{-1}\circ q_\B$.  This
  concludes the proof of our theorem.
\end{proof}

Our object for this section is technically to define a covariant
\emph{metrical} propinquity. However, this is now simple, except possibly for notational issues.
\begin{definition}
  Let $(F,F_{\mathsf{inner}},F_{\mathsf{mod}})$ be a permissible triple. A \emph{covariant metrical
    $(F,F_{\mathsf{inner}},F_{\mathsf{mod}})$-system} is given as a pair $(\mathds{M}, (\A,\Lip_\A))$  of a covariant modular $(F,Q$)-system
  $\mathds{M} = \CMS{\module{M}}{\CDN}{\beta}{(G,\delta_G,q)}{\B}{\Lip_\B}{\alpha}{(H,\delta_H) }$ and a {\qcms} $(\A,\Lip_\A)$ such that in particular,
  $(\module{M},\CDN,\A,\Lip_\A,\B,\Lip_\B)$ is a {\MVB{F}}.
\end{definition}

Note that we do not require any action on $(\A,\Lip_\A)$. To avoid
drowning in notations, we will not discuss the now easy construction
of a metric where an independent action on $(\A,\Lip_\A)$ is accounted
for: all that is needed will be to replace tunnels by covariant
tunnels in the obvious locations. We work here when no such action is
present.

\begin{definition}
  Let $(F,F_{\mathsf{inner}},F_{\mathsf{mod}})$ be a permissible triple. For each $j\in\{1,2\}$, let
  $(\mathds{M}_j,(\A_j,\Lip_j))$ be a covariant metrical
  $(F,F_{\mathsf{inner}},F_{\mathsf{mod}})$-systems, with
  \begin{equation*}
    \mathds{M}_j = \CMS{\module{M}_j}{\CDN_j}{\beta_j}{(G_j,\delta_{G_j},q_j)}{\B_j}{\Lip_{\B_j}}{\alpha_j}{(H_j,\delta_{H_j})}\text.
  \end{equation*}

  A \emph{$\varepsilon$-covariant metrical $(F,F_{\mathsf{inner}},F_{\mathsf{mod}})$-tunnel} $(\tau,\varsigma,\varkappa)$, for $\varepsilon > 0$, is given by a
  metrical $(F,F_{\mathsf{inner}},F_{\mathsf{mod}})$ tunnel $\tau$, and two local almost isometries
  $\varsigma\in\UIso{\varepsilon}{G_1}{G_2}{\frac{1}{\varepsilon}}$
  and
  $\varkappa\in\UIso{\varepsilon}{H_1}{H_2}{\frac{1}{\varepsilon}}$,
  such that
  \begin{equation*}
    \forall \{j,k\}=\{1,2\} \quad q_k\circ\varsigma_j = \varkappa_j\circ q_j \text.
  \end{equation*}
\end{definition}

The magnitude of a metrical tunnel is easily defined:
\begin{definition}
  If $\tau$ is a $\varepsilon$-covariant metrical tunnel, then the
  \emph{$\varepsilon$-metrical magnitude of $\tau$} is
  $\tunnelmodmagnitude{\tau}{\varepsilon} =
  \max\left\{\tunnelmodmagnitude{\tau_{\mathsf{mod}}}{\varepsilon},
    \tunnelextent{\tau_{\mathsf{base}}} \right\}$, where we used
  Notation (\ref{metrical-tunnel-notation}).
\end{definition}

The covariant metric propinquity is defined similarly to the other
versions of the covariant propinquity:

\begin{notation}
  For any permissible triple $(F,F_{\mathsf{inner}},F_{\mathsf{mod}})$, and any two covariant metrical $(F,F_{\mathsf{inner}},F_{\mathsf{mod}})$-systems $\mathds{A}$ and $\mathds{B}$, the set of all $\varepsilon$-covariant $(F,F_{\mathsf{inner}},F_{\mathsf{mod}})$-tunnels from $\mathds{A}$ to $\mathds{B}$ is denoted by:
  \begin{equation*}
    \tunnelset{\mathds{A}}{\mathds{B}}{F,F_{\mathsf{inner}},F_{\mathsf{mod}}}{\varepsilon} \text{.}
  \end{equation*}
\end{notation}

\begin{definition}
  Let $(F,F_{\mathsf{inner}},F_{\mathsf{mod}})$ be a permissible triple. The \emph{covariant metrical $(F,F_{\mathsf{inner}},F_{\mathsf{mod}})$-propinquity} between two covariant metrical $(F,F_{\mathsf{inner}},F_{\mathsf{mod}})$-systems $\mathds{A}$ and $\mathds{B}$ is:
  \begin{multline*}
    \dcovmetpropinquity{F}\left( \mathds{A}, \mathds{B} \right) = \\
    \min\left\{ \frac{\sqrt{2}}{2}, \inf\left\{\varepsilon > 0 :
        \exists\tau \in
        \tunnelset{\mathds{A}}{\mathds{B}}{F,F_{\mathsf{inner}},F_{\mathsf{mod}}}{\varepsilon} \quad
        \tunnelmodmagnitude{\tau}{\varepsilon} \leq \varepsilon
      \right\} \right\} \text{.}
  \end{multline*}
\end{definition}

Putting all our efforts together, we obtain:
\begin{definition}
  Let $(\mathds{A},\A,\Lip_\A)$ and $(\mathds{B},\B,\Lip_\B)$ be two
  covariant metrical systems. A \emph{full equivariant metrical
    quantum isometry} $(\Theta,\theta,\pi)$ is given by a full
  equivariant modular quantum isometry $(\Theta,\theta)$ from
  $\mathds{A}$ to $\mathds{B}$ and a full quantum isometry
  $\pi : (\A,\Lip_\A)\rightarrow(\B,\Lip_\B)$ such that $(\Theta,\pi)$
  is also a module map.
\end{definition}

\begin{theorem}\label{metrical-main-thm}
  Let $(F,F_{\mathsf{inner}},F_{\mathsf{mod}})$ be a permissible triple. The covariant metrical $(F,F_{\mathsf{inner}},F_{\mathsf{mod}})$-propinquity is a metric up to full equivariant metrical quantum isometry on the class of all covariant metrical
  $(F,F_{\mathsf{inner}},F_{\mathsf{mod}})$-systems.
\end{theorem}

\begin{proof}
  The proof follows from the similar proof for the covariant modular
  propinquity, with the addition of the proofs in
  \cite{Latremoliere18d} about the metrical propinquity. We will use
  Notation (\ref{metrical-tunnel-notation}).

  For instance, given $\tau$ a $\varepsilon_1$-covariant metrical
  $(F,F_{\mathsf{inner}},F_{\mathsf{mod}})$-tunnel from $(\mathds{A}_1,\A_1,\Lip_1)$ to
  $(\mathds{A}_2,\A_2,\Lip_2)$ and $\gamma = (\gamma_1,\gamma_2)$ a
  $\varepsilon_2$-covariant metrical $(F,F_{\mathsf{inner}},F_{\mathsf{mod}})$-tunnel from
  $(\mathds{A}_2,\A_2,\Lip_2)$ to $(\mathds{A}_3,\A_3,\Lip_3)$, then
  as long as $\varepsilon_1,\varepsilon_2 \leq \frac{\sqrt{2}}{2}$,
  Theorem (\ref{triangle-thm}) applies to $\tau_{\mathsf{mod}}$ and
  $\gamma_{\mathsf{mod}}$ to produce, for any $\varepsilon > 0$, a
  $\varepsilon_1 + \varepsilon_2$-covariant modular $(F,F_{\mathsf{inner}})$-tunnel
  $\tau_{\mathsf{mod}}\circ\gamma_{\mathsf{mod}}$ from $\mathds{A}_1$
  to $\mathds{A}_3$, whose magnitude is no more than
  $\tunnelmodmagnitude{\tau_{\mathsf{mod}}}{\varepsilon_1} +
  \tunnelmodmagnitude{\tau_{\mathsf{mod}}}{\varepsilon_2} +
  \varepsilon$, while \cite[Theorem 3.1]{Latremoliere14} shows how to similarly
  construct a tunnel $\tau_{\mathsf{base}}\circ\gamma_{\mathsf{base}}$
  from $(\A_1,\Lip_1)$ to $(\A_3,\Lip_3)$ with
  $\tunnelextent{\tau_{\mathsf{base}}\circ\gamma_{\mathsf{base}}} \leq
  \tunnelextent{\tau_{\mathsf{base}}} +
  \tunnelextent{\gamma_{\mathsf{base}}} + \varepsilon$. The same
  argument as \cite[Proposition 4.4]{Latremoliere18d} then shows that
  the pair
  $\tau'=(\tau_{\mathsf{mod}}\circ\gamma_{\mathsf{mod}},\tau_{\mathsf{base}}\circ\gamma_{\mathsf{base}})$
  defines a $(\varepsilon_1+\varepsilon_2)$-covariant
  $(F,F_{\mathsf{inner}},F_{\mathsf{mod}})$-metrical tunnel (using Notation
  (\ref{metrical-tunnel-notation})) from $(\mathds{A}_1,\A_1,\Lip_1)$
  to $(\mathds{A}_3,\A_3,\Lip_3)$, with
  $(\varepsilon_1+\varepsilon_2)$-magnitude at most
  $\tunnelmodmagnitude{\tau}{\varepsilon_1} +
  \tunnelmodmagnitude{\gamma}{\varepsilon_2} + \varepsilon$ --- and therefore,
  \begin{equation*}
    \tunnelmodmagnitude{\tau'}{\varepsilon_1+\varepsilon_2+\varepsilon} \leq \tunnelmodmagnitude{\tau'}{\varepsilon_1+\varepsilon_1} \leq \tunnelmodmagnitude{\tau}{\varepsilon_1} + \tunnelmodmagnitude{\gamma}{\varepsilon_2} + \varepsilon  \text.
  \end{equation*}
  This then
  can be used to show that the covariant metrical propinquity
  satisfies the triangle inequality as in \cite{Latremoliere18b}.

  Similarly, if
  $\dcovmetpropinquity{F}((\mathds{A},\A,\Lip_\A),
  (\mathds{B},\B,\Lip_\B)) = 0$, then in particular,
  \begin{equation*}
    \dcovmodpropinquity{F}(\mathds{A},\mathds{B}) = 0
  \end{equation*}
  and thus there exists a full equivariant modular quantum isometry
  $(\Theta,\theta) : \mathds{A} \rightarrow \mathds{B}$; while
  $\dpropinquity{F}((\A,\Lip_\A),(\B,\Lip_\B)) = 0$ and thus there
  exists a quantum isometry
  $\pi : (\A,\Lip_\A) \rightarrow (\B,\Lip_\B)$. By the same argument
  as \cite[Theorem 4.9]{Latremoliere18d}, we conclude that
  $(\Theta,\pi)$ is indeed a module morphism (note: the covariant
  metric propinquity dominates the metrical propinquity applied to the
  metrical quantum bundles obtained from forgetting the group actions,
  so \cite{Latremoliere18d} applies to give the metrical isomorphism
  directly).
\end{proof}

\section{The Gromov-Hausdorff Propinquity for Metric Spectral Triples}

Let $(\A,\Hilbert,D)$ be a metric spectral triple. To the canonical
metrical C*-corres\-pondence
$\mvb{\A}{\Hilbert}{D} = (\Hilbert,\CDN,\A,\Lip_D,\C,0)$, we also can
associated a canonical action of $\R$ by unitaries on $\Hilbert$,
setting $U : t \in \R\mapsto \exp(itD)$. Note that for all $t\in \R$,
since $U^t$ is unitary and since it commutes with $D$, we have
$\CDN(U^t \xi) = \CDN(\xi)$ for all $\xi \in \Hilbert$.  Now, in order
to also keep a record of the positive orientation of time flowing, we
actually consider the \emph{restriction} of the action $U$ to the
proper monoid $[0,\infty)$, with addition. We thus define:
\begin{definition}
  If $(\A,\Hilbert,D)$ is a metric spectral triple, the associated
  covariant modular system $\umvb{\A}{\Hilbert}{D}$ is defined as
  $(\mathds{D},\A,\Lip_D)$ where:
  \begin{equation*}
    \mathds{D} = \CMS{\Hilbert}{\CDN}{U}{([0,\infty),d,q)}{\C}{0}{\mathrm{id}}{(\{0\},d)}
  \end{equation*}
  with:
  \begin{equation*}
    U : t \in [0,\infty) \mapsto U_t = \exp( i t D) 
  \end{equation*}
  and
  \begin{equation*}
    (\Hilbert,\CDN,\A,\Lip_D,\C,0) = \mvb{\A}{\Hilbert}{D}\text,
  \end{equation*}
  while $\mathrm{id}$ is the identity map (seen here as an action of
  the trivial group $\{0\}$), $d$ is the usual distance induced by the
  usual metric of $\R$, and $q : t \in [0,\infty)\mapsto 0$.
\end{definition}

We thus can apply the covariant version of our metrical propinquity to
metric spectral triples.

\begin{definition}
  Let $(F,F_{\mathsf{inner}},F_{\mathsf{mod}})$ be a permissible triple. The \emph{spectral
    $(F,F_{\mathsf{inner}},F_{\mathsf{mod}})$-propinquity} between two metric spectral triples
  $(\A,\Hilbert_\A,D_\A)$ and $(\B,\Hilbert_\B,D_\B)$ is:
  \begin{multline*}
    \spectralpropinquity{F}((\A,\Hilbert_\A,D_\A),(\B,\Hilbert_\B,\D_\B))
    \\ =
    \dcovmetpropinquity{F}(\umvb{\A}{\Hilbert_\A}{D_\A},\umvb{\B}{\Hilbert_\B}{D_\B})
    \text{.}
  \end{multline*}
\end{definition}

\begin{remark}
  We should explain why a permissible triple parametrizes the spectral
  propinquity, since, in general, metric spectral triples give rise to
  Leibniz metrical C*-correspondences. The reason is that we allow for
  the covariant, metrical tunnels to be more general than imposing on
  them the usual Leibniz conditions. In particular, the covariant
  metrical tunnels involved in the computation of the spectral
  propinquity between spectral triples are not expected to arise from
  spectral triples. As per our usual convention, if we work only with
  Leibniz covariant tunnels, then we simply write
  $\spectralpropinquity{F}$ for the spectral propinquity.
\end{remark}

The main result of this work is:

\begin{theorem}\label{main-thm}
  Let $(F,F_{\mathsf{inner}},F_{\mathsf{mod}})$ be a permissible triple. The spectral propinquity
  $\spectralpropinquity{F}$ is a metric on the class of metric
  spectral triples up to equivalence of spectral triples.
\end{theorem}

\begin{proof}
  As the covariant metrical propinquity is indeed a pseudo-metric, so
  is the spectral propinquity. It is thus enough to study the distance
  zero question.

  Let $(\A,\Hilbert_\A,D_\A)$ and $(\B,\Hilbert_\B,\D_\B)$ be two
  metric spectral triples with:
  \begin{equation*}
    \spectralpropinquity{F}((\A,\Hilbert_\A,D_\A),(\B,\Hilbert_\B,D_\B)) = 0\text{,}
  \end{equation*}
  and write $U_\A : t\in[0,\infty)\mapsto \exp(itD_\A)$ and
  $U_\B : t\in[0,\infty) \mapsto \exp(it D_\B)$. We also write
  $\CDN_\A : \xi\in\dom{D_\A} \mapsto \norm{\xi}{\Hilbert_\A} +
  \norm{D_\A \xi}{\Hilbert_\A}$ and
  $\CDN_\B : \xi\in\dom{D_\B} \mapsto \norm{\xi}{\Hilbert_\B} +
  \norm{D_\B \xi}{\Hilbert_\B}$. Last, we write
  $\Lip_\A : a\in\dom{\Lip_\A} \mapsto
  \opnorm{[D_\A,a]}{}{\Hilbert_\A}$ and
  $\Lip_\B : a\in\dom{\Lip_\B} \mapsto
  \opnorm{[D_\B,b]}{}{\Hilbert_\B}$.

  By Theorem (\ref{metrical-main-thm}), there exists a C*-correspondence morphism $(\Theta,\theta,\pi)$, and an isometric isomorphism
  $\varsigma : [0,\infty)\rightarrow[0,\infty)$ such that
  $(\Theta,\theta,\varsigma,0)$ is a full equivariant
  modular quantum isometry from
  \begin{equation*}
    \CMS{\Hilbert_\A}{\CDN_\A}{U_\A}{[0,\infty)}{\C}{0}{\mathrm{id}}{\{0\}}
  \end{equation*}
  to
  \begin{equation*}
    \CMS{\Hilbert_\B}{\CDN_\B}{U_\B}{[0,\infty)}{\C}{0}{\mathrm{id}}{\{0\}}\text,
  \end{equation*}
  while $\pi : (\A,\Lip_\A)\rightarrow(\B,\Lip_\B)$ is a full quantum
  isometry, and $(\Theta,\pi)$ is a module morphism. Now, the only
  isometric isomorphism of the monoid $[0,\infty)$ is the identity, so
  we shall dispense with the notation $\varsigma$.

  As $\Theta$ is a surjective linear isomorphism of Hilbert spaces, it is a
  unitary, which we denote by $V$. As in Proposition
  (\ref{spectral-mvb-zero-prop}), since $(\Theta,\pi)$ is a module
  morphism, we conclude that $\pi = \AdRep{V}$ and moreover, $V$ (as
  it preserves the $D$-norms) maps $\dom{D_\A}$ onto $\dom{D_\B}$.

  Moreover, equivariance means that for all $t\in\R$, we have
  $V U_\A^t V^\ast = U_\B^t$. We then observe that, if
  $\xi \in \dom{D_\A}$ then, as $V$ is continuous and
  $V\dom{D_\A}=\dom{D_\B}$,
  \begin{equation*}
    i D_\A\xi = \lim_{\substack{t\rightarrow 0 \\ t>0}} \frac{U_\A^t\xi - \xi}{t} = \lim_{\substack{t\rightarrow 0 \\ t>0}} \frac{V^\ast U_\B^t  V \xi - \xi}{t} = V^\ast \lim_{\substack{t\rightarrow 0 \\ t>0}} \frac{U_\B^t V \xi- V\xi}{t} = i V^\ast D_\B V \xi \text{.}
  \end{equation*}
  Therefore, as desired, $(\A,\Hilbert_\A,D_\A)$ and
  $(\B,\Hilbert_\B,D_\B)$ are equivalent.

  It is immediate that equivalent metric spectral triples are at
  distance zero for our spectral propinquity, concluding our proof.
\end{proof}

We thus have constructed our distance over the space of metric spectral triples (up to a choice of permissible triple). We now include some examples of applications, which, in particular, prove that our metric is not discrete or otherwise too rigid.

\section{Applications}

We conclude with examples of convergence for the spectral propinquity in this paper, two of which are established in two companion papers. Our first
example concerns simple perturbations of metric spectral triples. A
second family of examples concerns spectral triples on fractals. A
third family of examples concern finite dimensional approximations of
spectral triples on quantum tori.

\subsection{Perturbation of Metric Spectral Triples}

Let $(\A,\Hilbert,D)$ be a metric spectral triple and let $T$ be a
bounded self-adjoint linear operator acting on $\Hilbert$. We write,
as before,
\begin{equation*}
  \dom{\Lip_\A} = \left\{ a\in\sa{\A}: a\dom{D}\subseteq\dom{D},[D,a]\text{ bounded} \right\}\text,
\end{equation*}
and
\begin{equation*}
  \forall a \in \dom{\Lip_\A} \quad \Lip(a) = \opnorm{[D,a]}{}{\Hilbert} \;\text{ and }\; \forall \xi\in\dom{D} \quad \CDN(\xi) = \norm{\xi}{\Hilbert} + \norm{D\xi}{\Hilbert} \text{.}
\end{equation*}
  
Our goal is to study the continuity of the family of spectral triples $(\A,\Hilbert,D + T)$ as $T$ varies in some neighborhood of $0$ in the space of self-adjoint operators on $\Hilbert$. We first note that, of course, using the resolvent identity, $D+T$ converges, in the sense of the resolvent convergence, to $D$, as $T$ converges in norm to $0$. In particular, some of the difficulties addressed in this paper, regarding working on different Hilbert spaces, may appear moot here. It is, however, not quite the case. Indeed, of interest to us is also the convergence of the quantum metrics induced on $\A$ by these varying spectral triples. As we shall see, this is very integral to our approach. This relatively simple example will illustrate the basic scheme to establish convergence for the spectral propinquity, which of course gets more complicated in the next two examples.

We first note that $D + T$ is indeed self-adjoint with the same domain
as $D$. Moreover, it has compact resolvent. For all $z \in \C$ not in the spectrum of $D$, we denote the resolvent $(D+z)^{-1}$ of $D$ at $z$ by $\resolvent{D}{z}$. Since
\begin{align*}
  \resolvent{D+T}{i}
  &= \resolvent{D+T}{i} - \resolvent{D}{i} + \resolvent{D}{i} \\
  &= \left( \resolvent{D+T}{i} T - 1 \right) \resolvent{D}{i}
\end{align*}
and since $\resolvent{D}{i}$ is compact, and since the algebra of
compact operators sits as an ideal in $\B(\Hilbert)$, we conclude that
$\resolvent{D+T}{i}$ is compact as well.

Our purpose is to study the perturbed spectral triple
$(\A,\Hilbert,D+T)$. Our first problem is to prove that this triple is
indeed metric. For any $a\in\dom{\Lip_\A}$, we write
$\Lip_T(a) = \opnorm{[D+T,a]}{}{\Hilbert}$ (and let $\Lip_T(a)=\infty$
if $a\in\sa{\A}\setminus\dom{\Lip_\A}$). By construction, $\Lip_T$ is
defined on a dense Jordan-Lie algebra of $\sa{\A}$ and satisfies the
Leibniz inequality. It is also lower semicontinuous as $D+T$ is
self-adjoint.

In general, it may not be true that $\Lip_T(a) = 0$ if and only if
$a\in\R\unit_\A$, so this is the first question we must address.

let $r = \diam{\A}{\Lip}$ be the diameter of
$(\StateSpace(\A),\Kantorovich{\Lip})$. In the rest of this example,
we assume:
\begin{equation*}
  \opnorm{T}{}{\Hilbert} < \frac{1}{2 r}\text{.}
\end{equation*}
  
Let $a\in\dom{\Lip_\A}$, and let $\varphi\in\StateSpace(\A)$ and
$a' = a-\varphi(a)\unit_\A$. First, note that
\begin{equation*}
  \norm{a-\varphi(a)}{\A} = \sup\left\{|\psi(a-\varphi(a))|:\psi\in\StateSpace(\A) \right\} \leq r \Lip(a) \text.
\end{equation*}

We then compute:
\begin{equation}\label{perturbation-ex-eq1}
  \begin{split}
    |\Lip(a)-\Lip_T(a)| &= |\Lip(a')-\Lip_T(a')| \\
    &= \left|\opnorm{[D,a']}{}{\Hilbert} - \opnorm{[D+T,a']}{}{\Hilbert} \right| \\
    &\leq \opnorm{[T,a']}{}{\Hilbert} \\
    &\leq 2\opnorm{T}{}{\Hilbert} \norm{a'}{\A} \\
    &\leq \underbracket[1pt]{2 \opnorm{T}{}{\Hilbert}
      r}_{\text{constant}} \Lip(a) \text{.}
  \end{split}
\end{equation}

Therefore, for all $a\in\dom{\Lip_\A}$, we conclude:
\begin{equation*}
  \left(1-2 r \opnorm{T}{}{\Hilbert}\right)\Lip(a) \leq \Lip_T(a) \leq \left(1 + 2 r \opnorm{T}{}{\Hilbert}\right) \Lip(a) \text,
\end{equation*}
so
$\frac{1}{1+2r\opnorm{T}{}{\Hilbert}} \Lip_T(a) \leq \Lip(a) \leq
\frac{1}{1-2r\opnorm{T}{}{\Hilbert}} \Lip_T(a)$ for all
$a\in\dom{\Lip_\A}$ --- noting that $2r\opnorm{T}{}{\Hilbert}<1$, so
we do divide by strictly positive numbers. From \cite[Lemma
1.10]{Rieffel98a}, we conclude that $(\A,\Lip_T)$ is indeed a {\qcms}.
  
Thus, $(\A,\Hilbert,D+T)$ is a metric spectral triple.

We also note that the diameter $\diam{\A}{\Lip_T}$ of
$(\StateSpace(\A),\Kantorovich{\Lip_T})$ is no more than
\begin{equation}\label{rt-eq}
  r_T = \frac{1}{1-2r\opnorm{T}{}{\Hilbert}}\diam{\A}{\Lip} \geq r\text.
\end{equation}
We then note that for all $a\in\dom{\Lip_\A}$, following the same method
as above, we get:
\begin{equation} \label{perturbation-eq-1}
  \left|\Lip(a)-\Lip_T(a)\right| \leq 2 \opnorm{T}{}{\Hilbert} r_T \Lip_T(a) \text{.}
\end{equation}
These inequalities will prove helpful for our computations.
  
\bigskip
  
If we now set:
\begin{equation*}
  \forall \xi \in \dom{D} \quad \CDN_T(\xi) = \norm{\xi}{\Hilbert} + \norm{(D+T)\xi}{\Hilbert} 
\end{equation*}
then $\mvb{\A}{\Hilbert}{D} = (\Hilbert,\CDN_T,\A,\Lip_T,\C,0)$ is a
metrical vector bundle.

\bigskip

We now estimate how far apart $(\A,\Lip)$ and $(\A,\Lip_T)$ are with respect
to the propinquity.
  
For any $a,b \in \A$, we set:
\begin{equation*}
  \mathsf{S}(a,b) = \max\left\{ \Lip(a), \Lip_T(b), \left(\frac{1 - 2 r_T \opnorm{T}{}{\Hilbert}}{2 r_T^2 \opnorm{T}{}{\Hilbert}}\right) \norm{a-b}{\A} \right\}
\end{equation*}
which is an L-seminorm on $\A\oplus\A$, using techniques from
\cite{Latremoliere14}. Moreover, if $a\in\A$ with $\Lip(a) = 1$, then
setting
$b = \frac{1}{1+2 r_T \opnorm{T}{}{\Hilbert}}a + \frac{2 r_T
  \opnorm{T}{}{\Hilbert}}{1+2 r_T \opnorm{T}{}{\Hilbert}}\mu(a)$ for
some $\mu\in\StateSpace(\A)$, we observe, first, that by Expression (\ref{perturbation-eq-1}),
\begin{equation*}
  \Lip_T(b) = \frac{1}{1+2r_T \opnorm{T}{}{\Hilbert}} \Lip_T(a) \leq \frac{1 - 2 r_T \opnorm{T}{}{\Hilbert}}{1 + 2 r_T \opnorm{T}{}{\Hilbert}} \Lip(a) \leq 1\text.
\end{equation*}
Moreover, using Expression (\ref{perturbation-ex-eq1}), as well as the following computation:
\begin{align*}
  \frac{1-2 r_T \opnorm{T}{}{\Hilbert}}{2 r_T^2 \opnorm{T}{}{\Hilbert}} \norm{a-b}{\A} 
  &\leq \frac{1-2 r_T \opnorm{T}{}{\Hilbert}}{1 + 2 r_T \opnorm{T}{}{\Hilbert}} \frac{\norm{a-\mu(a)}{\A}}{r_T} \\
  &\leq \frac{\norm{a-\mu(a)}{\A}}{r_T} \leq \frac{r \Lip(a)}{r_T} \leq 1\text,
\end{align*}
that indeed $\mathsf{S}(a,b) = 1\text{.}$

Similarly, if $b\in\A$ with $\Lip_T(b)\leq 1$ then, setting
$a =\frac{1}{1-2r_T \opnorm{T}{}{\Hilbert}}b - \frac{2r_T \opnorm{T}{}{\Hilbert}}{1-2r_T \opnorm{T}{}{\Hilbert}}$, we get again by Expression (\ref{perturbation-eq-1}),
\begin{equation*}
  \Lip(a) \leq (1 + 2 r_T \opnorm{T}{}{\Hilbert}) \Lip_T(a) \leq \frac{1+2 r_T \opnorm{T}{}{\Hilbert}}{1 + 2 r_T \opnorm{T}{}{\Hilbert}} \Lip_T(b) \leq 1 \text,
\end{equation*}
and, with a calculation similar as above, $\mathsf{S}(a,b) = 1$.

In conclusion,
$\tau_T^{\textsf{space}} = (\A\oplus\A,\mathsf{S},\pi_1,\pi_2)$, with
$\pi_j : (a_1,a_2)\in\A\oplus\A\mapsto a_j$ ($j=1,2$), is a Leibniz
tunnel from $(\A,\Lip)$ to $(\A,\Lip_T)$. Its extent is no more than
$\frac{2 r_T^2 \opnorm{T}{}{\Hilbert}}{1 - 2 r_T
  \opnorm{T}{}{\Hilbert}}$ so we get:
\begin{equation*}
  \dpropinquity{}((\A,\Lip),(\A,\Lip_T)) \leq \frac{2 r_T^2 \opnorm{T}{}{\Hilbert}}{1 - 2 r_T \opnorm{T}{}{\Hilbert}}  \text{.}
\end{equation*}
Consequently, since $\lim_{T\rightarrow 0} r_T = r$ by Equation (\ref{rt-eq}), we conclude:
\begin{equation*}
  \operatorname*{\dpropinquity{}-lim}_{\substack{T\rightarrow 0 \\ T \in \sa{\Hilbert} \\ \opnorm{T}{}{\Hilbert}<\frac{1}{2 r} }} (\A,\Lip_T) = (\A,\Lip) \text{.}
\end{equation*}
We record that our tunnel is Leibniz by setting
$F : x,y,l_x,l_y \geq 0 \mapsto x l_y + y l_x$.

Let us define the function
\begin{equation*}
  C_T = \min\left\{ 1 + \frac{1}{\opnorm{T}{}{\Hilbert}}, \frac{1 - 2 r_T \opnorm{T}{}{\Hilbert}}{2 r_T^2 \opnorm{T}{}{\Hilbert}} \right\} \text.
\end{equation*}

Now, for all $\xi,\eta\in\dom{D}$, we set:
\begin{equation*}
  \CDN'(\xi,\eta) = \max\left\{ \CDN(\xi), \CDN_T(\eta), C_T \norm{\xi-\eta}{\Hilbert} \right\} \text{.} 
\end{equation*}

For all $\xi \in \Hilbert$, we note that:
\begin{align*}
  \left|\CDN(\xi) - \CDN_T(\xi)\right| &\leq \norm{T\xi}{\Hilbert} \\
                                       &\leq \opnorm{T}{}{\Hilbert} \norm{\xi}{\Hilbert} \leq \min \left\{\opnorm{T}{}{\Hilbert} \CDN(\xi), \opnorm{T}{}{\Hilbert} \CDN_T(\xi) \right\}
\end{align*}
  
Once more, it is easy to check that $\CDN'$ is a D-norm on
$\Hilbert\oplus\Hilbert$, where $\Hilbert\oplus\Hilbert$ is a module
over $\A\oplus\A$ via the diagonal action:
$(a,b)(\xi,\eta) = (a\xi,b\eta)$. Moreover, if $\xi \in \dom{D}$ with
$\CDN(\xi) = 1$ then, setting
$\eta = \frac{1}{1+\opnorm{T}{}{\Hilbert}}\xi$, we get
$\CDN'(\xi,\eta)=1$, and similarly, if $\eta\in\Hilbert$ with
$\CDN_T(\eta)=1$, then setting
$\xi = \frac{1}{1+\opnorm{T}{}{\Hilbert}}\eta$ we get
$\CDN'(\xi,\eta) = 1$.

Thus, if $\Pi_j : (\xi_1,\xi_2)\in\Hilbert\oplus\Hilbert\mapsto \xi_j$
($j=1,2$), then the quotient norm of $\CDN'$ via $\Pi_1$
(resp. $\Pi_2$) is $\CDN$ (resp. $\CDN_T$).

We now check the Leibniz identity. Let $a,b \in \A$ and
$\xi,\eta\in \Hilbert$. First, we estimate:
  \begin{align*}
    \CDN'(a\xi,b\eta) &= \max\left\{ \CDN(a\xi), \CDN_T(b\eta), C_T \norm{a\xi-b\eta}{\Hilbert} \right\}\\
                      &\leq \max\left\{\begin{array}{l}
                                         (\norm{a}{\A}+\Lip(a))\CDN(\xi)\\
                                         (\norm{b}{\A}+\Lip_T(b))\CDN_T(\eta)\\
                                         C_T\left(\norm{a}{\A}\norm{\xi-\eta}{\Hilbert}+\norm{a-b}{\A}\norm{\eta}{\Hilbert}\right)
                                       \end{array} \right\}\text{.}
  \end{align*}
  In particular,
  \begin{equation*}
    C_T  \left(\norm{a}{\A}\norm{\xi-\eta}{\Hilbert}+\norm{a-b}{\A}\norm{\eta}{\Hilbert}\right) \\
    \leq \norm{a}{\A} \CDN'(\xi,\eta) + \TLip(a,b)\CDN'(\xi,\eta)\text.
  \end{equation*}

  Therefore, $\CDN'$ satisfies the desired Leibniz property. We make an observation: the choice of the value $C_T$, rather than just $1+\frac{1}{\opnorm{T}{}{\Hilbert}}$, is motivated by the Leibniz relation. Thus, the Leibniz relation between the D-norm and the L-seminorm does indeed force some constraint on the D-norm, which, indeed, is why we only need to use the usual extent of the tunnel between base spaces when working with modules (see \cite{Latremoliere18d}). Thus, Leibniz type conditions, which were not used at all early in noncommutative metric geometry \cite{Rieffel00}, then re-introduced by us as a key property to obtain a metric on {\qcms s} up to *-isomorphism in \cite{Latremoliere13,Latremoliere13b}, is now an essential tool encoding some rigidity needed to define our metric between modules.

  To construct our metrical tunnel, we are looking at
  $\Hilbert\oplus\Hilbert$ as a $\C\oplus\C$-Hilbert module for the
  diagonal action, where we set:
  \begin{equation*}
    \mathsf{Q}(z,w) = C_T |z-w| \text{.}
  \end{equation*}
  Of course, $\mathsf{Q}$ is an L-seminorm on $\C\oplus\C$. Moreover,
  it is immediate that $(\C\oplus\C,\mathsf{Q},j_1,j_2)$ is a
  tunnel from $\C$ to $\C$, with $j_1:(z,w)\in\C^2\mapsto z$ and
  $j_2:(z,w)\in\C^2\mapsto w$. Moreover, for all
  $\xi,\xi',\eta,\eta'\in\Hilbert$:
  \begin{align*}
    \mathsf{Q}(\inner{(\xi,\eta)}{(\xi',\eta')}{\C\oplus\C})
    &=C_T \left|\inner{\xi}{\xi'}{\Hilbert} - \inner{\eta}{\eta'}{\Hilbert}\right| \\
    &\leq C_T \left(\inner{\xi-\eta}{\xi'}{\Hilbert} + \inner{\eta}{\xi'-\eta'}{\Hilbert} \right) \\
    &\leq C_T \left(\norm{\xi-\eta}{\Hilbert}\norm{\xi'}{\Hilbert} + \norm{\eta}{\Hilbert}\norm{\xi'-\eta'}{\Hilbert}  \right)\\
    &\leq 2 \CDN'(\xi,\eta) \CDN'(\xi',\eta') \text{.}
  \end{align*}
    
  We set:
  \begin{equation*}
    \mathds{P}_T = \left( \Hilbert\oplus\Hilbert,\CDN',\C\oplus\C,\mathsf{Q},\A\oplus\A,\mathsf{S} \right)
  \end{equation*}
  We then have checked that, for any self-adjoint operator $T$ on
  $\Hilbert$ with $\opnorm{T}{}{\Hilbert} < \frac{1}{2 r}$, we have constructed a Leibniz metrical tunnel:
  \begin{equation*}
    \tau_T = \left(\mathsf{P}_T, (\Pi_1,j_1,\pi_1), (\Pi_2,j_2,\pi_2) \right)
  \end{equation*}
  such that:
  \begin{equation*}
    \tunnelextent{\tau}{} \leq \frac{1}{C_T} \text,
  \end{equation*}
  and thus, we conclude:
  \begin{equation*}
    \dmetpropinquity{F}\left(\mvb{\A}{\Hilbert}{D+T},\mvb{\A}{\Hilbert}{D}\right) \leq \frac{1}{C_T} \text{.}
  \end{equation*}

  We now turn to estimating the covariant propinquity. Under our
  conditions, \cite[IX, Theorem 2.12, p. 502]{Kato} applies with
  $a=\opnorm{T}{}{\Hilbert}$, $\beta = 0$ and $M=1$ (the last two
  quantities following from the fact that the spectrum of $iD$ is
  purely imaginary), so that for all $t \in [0,\infty)$:
  \begin{equation*}
    \opnorm{(\exp(itD) - \exp(it(D+T)))(iD+1)^{-1}}{}{\Hilbert} \leq t \opnorm{T}{}{\Hilbert} \text{.}
  \end{equation*}

  Let $\xi$ with $\CDN(\xi)\leq 1$, so that
  $\norm{D\xi}{} + \norm{\xi}{} \leq 1$, which in particular implies
  that $\norm{(iD+1)\xi}{}\leq 1$. Then for all $t\in[0,\infty)$:
  \begin{multline*}
    \norm{(\exp(itD) - \exp(it(D+T)))\xi}{\Hilbert} \\
    \begin{split}
      &=\norm{(\exp(itD) - \exp(it(D+T)))(iD+1)^{-1} (iD+1) \xi}{\Hilbert} \\
      &\leq \opnorm{(\exp(itD) - \exp(it(D+T)))(iD+1)^{-1}}{}{\Hilbert} \\
      &\leq |t| \opnorm{T}{}{\Hilbert} \text{.}
    \end{split}
  \end{multline*}

  The same reasoning applies to give us, for all $\xi \in \Hilbert$
  with $\CDN_T(\xi)\leq 1$ and $t \in [0,\infty)$:
  \begin{multline*}
    \norm{(\exp(itD) - \exp(it(D+T)))\xi}{\Hilbert} \\
    \begin{split}
      &=\norm{(\exp(it(D+T-T)) - \exp(it(D+T)))(i(D+T)+1)^{-1} (i(D+T)+1) \xi}{\Hilbert} \\
      &\leq \opnorm{(\exp(itD) - \exp(it(D+T)))(i(D+T)+1)^{-1}}{}{\Hilbert} \\
      &\leq |t| \opnorm{T}{}{\Hilbert} \text{.}
    \end{split}
  \end{multline*}

  Let $\varepsilon > 0$ be given. First, there exists
  $C \in \left( 0 , \frac{1}{2r} \right) $ such that if
  $\opnorm{T}{}{\Hilbert} < C$ then $0<\frac{1}{C_T}<\frac{\varepsilon}{3}$.
  
  Now, let $t \in [0,\infty)$ such that $t < \frac{1}{\varepsilon}$. Let $c \in (0,1)$ such that $\frac{c}{1-c}<\frac{\varepsilon}{3}$.

  Let $T \in \sa{\B(\Hilbert)}$ be chosen so that $\opnorm{T}{}{\Hilbert} \leq \delta = \min \left\{ C, \frac{\varepsilon^2}{3}, c \right\}$.

  First, note that for all $\xi \in \dom{D}$:
  \begin{equation*}
    \left| \CDN(\xi) - \CDN_T(\xi) \right| \leq \opnorm{T}{}{\Hilbert} \norm{\xi}{\Hilbert} \leq c \norm{\xi}{\Hilbert} \text,
  \end{equation*}
  so if $\norm{\xi}{\Hilbert}\leq 1$, then $\CDN_T(\xi) \geq \CDN(\xi) - \opnorm{T}{}{\Hilbert} \geq 1 - c$. So $\frac{|\CDN(\xi)-\CDN_T(\xi)|}{\CDN_T(\xi)}\leq \frac{c}{1-c}$ whenever $\norm{\xi}{\Hilbert} \leq 1$.
  
  If $\xi \in \Hilbert$ with $\CDN(\xi)\leq 1$, then for all
  $(\eta,\eta')\in\Hilbert\oplus\Hilbert$ with
  $\CDN'(\eta,\eta')\leq 1$ (so that
  $\norm{\eta-\eta'}{\Hilbert} < \frac{\varepsilon}{3}$):
  \begin{multline*}
    \left|\inner{\eta}{\exp(itD)\xi}{\Hilbert} - \inner{\eta'}{\exp(it(D+T))\frac{1}{\CDN_T(\xi)}\xi}{\Hilbert}\right| \\
    \begin{split}
      &\leq \norm{\eta-\eta'}{\Hilbert} + \norm{(\exp(itD)-\exp(it(D+T))) \xi}{\Hilbert} \\
      &\quad + \norm{\exp(it(D + T))\xi-\exp(it(D+T))\frac{1}{\CDN_T(\xi)}\xi}{\Hilbert} \\
      &< \frac{\varepsilon}{3} + |t|\opnorm{T}{}{\Hilbert} + \norm{(1-\frac{1}{\CDN_T(\xi)})\xi}{\Hilbert} \\
      &= \frac{2 \varepsilon}{3} + \norm{\frac{\CDN(\xi)-\CDN_T(\xi)}{\CDN_T(\xi)}\xi}{\Hilbert} \\
      &\leq \frac{2\varepsilon}{3} + \frac{c}{1-c} \\
      &< \varepsilon \text{.}
    \end{split}
  \end{multline*}
  The same computation can be made for all $\varphi\in\StateSpace(\A)$
  and $\xi \in \Hilbert$ with $\CDN_T(\xi)\leq 1$. Consequently, we
  have shown that:
  \begin{equation*}
    \tunnelmagnitude{\tau}{\varepsilon} \leq \varepsilon \text{.}
  \end{equation*}

  Therefore:

    \begin{multline*}
      \forall\varepsilon > 0 \quad \exists \delta > 0 \quad \forall T \in \sa{\B(\Hilbert)} \\
      \opnorm{T}{}{\Hilbert} < \delta \implies
      \lspectralpropinquity{}((\A,\Hilbert,D), (\A,\Hilbert,D+T)) <
      \varepsilon \text{,}
    \end{multline*}

    where $\lspectralpropinquity = \spectralpropinquity{F}$ where: $F:x,y,z,t \in [0,\infty)^4\mapsto x z + y t$, $F_{\mathsf{inner}} : x,y,z\in [0,\infty)^3\mapsto (x+y)z$, and $F_{\mathsf{mod}} : x,y \in [0,\infty)^2\mapsto 2 x y$, i.e. the usual Leibniz conditions. This concludes our example.

    \subsection{Approximation of Spectral Triples on Fractals}

    The {\SiepG} $\SG{\infty}$ is a fractal, constructed as the
    attractor set of an iterated function system (IFS) of affine
    functions of the plane. Specifically, let
    \begin{equation*}
      v_0 = \begin{pmatrix} 0 \\ 0 \end{pmatrix}\text{, }v_1 = \begin{pmatrix} 1 \\ 0 \end{pmatrix} \text{ and }v_2 = \begin{pmatrix} \frac{1}{2} \\ \frac{\sqrt{3}}{2} \end{pmatrix} \text{.}
    \end{equation*}
    We define three similitudes of the plane by letting for each
    $j\in\{0,1,2\}$,
    \begin{align*}
      T_j &: x \in \R^2 \longmapsto \frac{1}{2} \left( x + v_j \right) \in \R^2 \text{.}
    \end{align*}
    we define various triangles by induction:
    \begin{equation*}
      \begin{cases}
        \Delta_{0,1} = [v_0,v_1]\cup[v_1,v_2]\cup[v_2,v_0]\text,\\
        \Delta_{n+1,j + r 3^n} = T_r\Delta_{n,j}\text{ for all
          $n\in\N$, $j\in\{1,\ldots,3^{n}\}$, and $r\in\{0,1,2\}$.}
      \end{cases}
    \end{equation*}
    For each $n\in\N$, we let
    $\SG{n} = \bigcup_{j=1}^{3^n} \Delta_{n,j}$. The \emph{\SiepG}
    $\SG{\infty}$ is the closure of $\bigcup_{n\in\N}\SG{n}$.

\begin{figure}\label{SiepG-fig}
  \centering
  \begin{tabular}{m{8cm} m{1cm} m{2cm}}
    \includegraphics[height=0.8in]{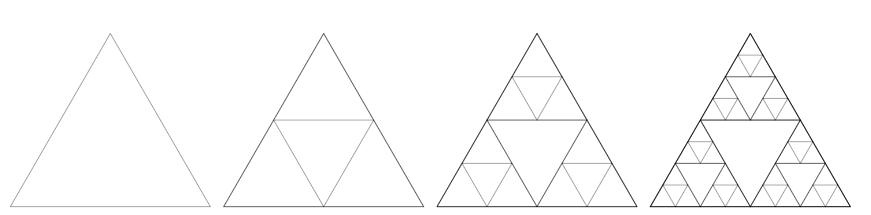} & \vspace*{1.6cm} $\ldots \longrightarrow$ \vspace*{0.3in} & \includegraphics[height=0.8in]{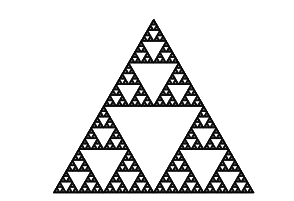}
  \end{tabular}
  \caption{The {\SiepG} is a limit of graphs in the plane for the
    Hausdorff distance}
\end{figure}
The {\SiepG} is a prototype fractal, and has been extensively
studied. In \cite{Lapidus08}, a spectral triple was constructed on the
{\SiepG}. Since the {\SiepG} is, naturally, the limit of the graphs
$\SG{n}$ as $n$ goes to $\infty$, a natural question is whether the
spectral triple of \cite{Lapidus08} is the limit of spectral triples
on $\SG{n}$ as $n$ varies in $\N$.

We answered this question positively in \cite{Latremoliere20a}, using
the metric defined in the present paper. We briefly recall the
construction of the spectral triples involved, which are constructed
using direct sums of spectral triples over the unit interval,
appropriately re-scaled.

Let $\mathrm{CP}$ be the unital Abelian C*-algebra of all $\C$-valued
continuous functions $f$ over $[0,1]$ such that $f(0) = f(1)$:
\begin{equation*}
  \mathrm{CP} = \left\{ f \in C([-1,1]) : f(-1) = f(1) \right\} \text{.}
\end{equation*}
The Gelfand spectrum of $\mathrm{CP}$ is of course homeomorphic to the
unit circle $\T = \{z\in\C : |z|=1\}$ in $\R^2$.

We now define a spectral triple on $\mathrm{CP}$. Let $\mathscr{J}$ be
the Hilbert space closure of $\mathrm{CP}$ for the inner product
$(f,g) \in \mathrm{CP} \mapsto \int_{-1}^1 fg$. As usual, we identify
$f \in \mathrm{CP}$ with the multiplication operator by $f$ on
$\mathscr{J}$.

For each $k \in \Z$, let $e_k : t \in [-1,1] \mapsto \exp(i \pi k t)$
--- of course, $e_k \in \mathscr{J}$. We define $D$ as the closure of
the linear extension of the map defined as:
\begin{equation*}
  \forall k \in \Z \quad D e_k =  \pi \left( k + \frac{1}{2} \right)e_k \text{.} 
\end{equation*}
A standard argument establishes that, indeed,
$\left( \mathrm{CP}, \mathscr{J}, D \right)$ is a spectral triple over
$\mathrm{CP}$ \cite{Lapidus08}. Moreover, the Connes metric induced
by this spectral triple on $\T$, seen as the Gelfand spectrum of
$\mathrm{CP}$, is the usual geodesic distance of $\T$ (i.e. the
distance between two points is the smallest of the lengths of the two
arcs between these points).

\medskip

We now use the spectral triple $(\mathrm{CP},\mathscr{J},D)$ in order
to construct a spectral triple on the unit interval $[0,1]$. If
$f \in C([0,1])$, then the map $t \in [-1,1]\mapsto f(|t|)$ is in
$CP$. Let $\varpi$ be the faithful *-representation of $C([0,1])$ on
$\mathscr{J}$ defined by
\begin{equation*}
  \forall f \in C([0,1]), \quad \forall \xi \in \mathscr{J}, \quad \varpi(f)\xi : t \in [-1,1] \mapsto f(|t|)\xi(t) \text{.}
\end{equation*}
It is easy to check that $\left( C([0,1]), \mathscr{J}, D \right)$ is
a metric spectral triple over $C([0,1])$ which induces the usual
metric on $[0,1]$.

\medskip

We now construct our spectral triple over $\SG{n}$ for all
$n\in\N\cup\{\infty\}$. Let $n \in \N$ and $j \in \{1,\ldots,3^n\}$ be
given. Let $w_0,w_1,w_2 \in V_n$ be the vertices of the triangle
$\Delta_{n,j}$ listed in counter-clockwise order. We parametrize
$\Delta_{n,j}$ by defining the following map:
\begin{equation*}
  r_{n,j} : t \in \left[ 0, 1 \right] \longmapsto
  \begin{cases}
    (1- 3t ) w_0 + 3 t w_1 \text{ if $3 t \in [0,1]$,} \\
    (2- 3 t) w_1 + (3 t-1) w_2 \text{ if $3 t \in [1,2]$,} \\
    (3- 3 t) w_2 + (3 t-2) w_0 \text{ if $3 t \in [2,3]$.}
  \end{cases}
\end{equation*}

We can define a spectral triple
$\left(C(\Delta_{n,j}),\pi_{n,j}, \frac{2^n}{3} D\right)$ where
$\pi_{n,j}$ is the representation of $C(\Delta_{n,j})$ on
$\mathscr{J}$ which sends $f \in C(\Delta_{n,j})$ to the
multiplication operator by $f\circ r_{n,j} \in \mathrm{CP}$ on
$\mathscr{J}$. It is now easy to check that in particular,
$\Delta_{n,j}$ with the induced {\MongeKant} is isometric to
$\Delta_{n,j}$ with the geodesic distance induced on $\Delta_{n,j}$ by
the Euclidean metric $\R^2$.

Now, let $n \in \N\cup\{\infty\}$ and write
$\Hilbert_n = \oplus_{k = 1}^n \oplus_{j = 0}^{3^k} \mathscr{J}$. For
each $k\in\N$ with $k\leq n$ and for all $j\in \{1,\ldots, 3^n\}$, we
define $q_{k,j} : C(\SG{n}) \twoheadrightarrow C(\Delta_{k,j})$ as the
quotient map which restricts a function in $C(\SG{n})$ to
$\Delta_{k,j}\subseteq\SG{n}$. We now define a representation of
$C(\SG{n})$ on $\Hilbert_n$ as the diagonal representation by setting
for all $f \in C(\SG{n})$ and
$\xi = (\xi_{k,j})_{k\in\N_n,j\in\{1,\ldots,3^k\}}$:
\begin{equation*}
  \pi_n(f)\xi = \left( \pi_{k,j}(q_{k,j}(f))\xi_{k,j} \right)_{k\in\N_n, j\in\{1,\ldots,3^k\}} \text{.}
\end{equation*}
Last, for all
$\xi = (\xi_{k,j})_{k\in\N_n,j\in\{1,\ldots,3^k\}} \in \Hilbert_n$, we
set:
\begin{equation*}
  D_n\xi = \left( \frac{2^k}{3} D \xi_{k,j} \right)_{k\in\N_n,j\in\{1,\ldots,3^k\}}\text{.}
\end{equation*}

For all $n\in\N\cup\{\infty\}$, the triple
$\left(C(\SG{n}),\Hilbert_n,D_n\right)$ is a metric spectral triple,
as seen in \cite{Lapidus08}. Moreover, the Connes metric induced on
$C(\SG{n})$ by $\left(C(\SG{n}),\Hilbert_n,D_n\right)$ is the geodesic
distance on $\SG{n}$ (not the restriction of the metric from $\R^2$).

We prove in \cite{Latremoliere20a}:

\begin{theorem}\label{SiepG-thm}
  The following limit holds:
  \begin{equation*}
    \lim_{n\rightarrow\infty} \spectralpropinquity{F}\left(\left(C(\SG{n}),\Hilbert_n,D_n\right), \left(C(\SG{\infty}),\Hilbert_\infty,D_\infty\right)\right) = 0 \text.
  \end{equation*}
\end{theorem}

In \cite{Latremoliere20a}, Theorem (\ref{SiepG-thm}) is actually
established for a larger class of fractals, called piecewise $C^1$
fractal curves, which include the {\SiepG}, as well as its Harmonic
cousin, called the Harmonic gasket, thus providing a long list of
interesting examples of convergence of spectral triples, from the
noncommutative geometric study of fractals.

\subsection{Approximation of Some Spectral Triples on Quantum Tori}

A common example of a finite dimensional, quantum space used in
mathematical physics as an approximation of the $2$-torus $\T^2$,
where $\T = \{ z \in \C : |z| = 1\}$ is the C*-algebra generated by
the so-called clock-and-shift matrices:
\begin{equation}\label{Clock-Shift-eq}
  S_n =
  \begin{pNiceMatrix}
    0      & 1 & 0      & \Cdots &  0 \\
    \Vdots & \Ddots & \Ddots & \Ddots  & \Vdots \\
    &        &        &         & 0\\
    0      & \Cdots &        &         0 & 1 \\
    1 & 0 & \Cdots & & 0
  \end{pNiceMatrix} \text{ and }C_n =
  \begin{pNiceMatrix}
    1 & & & \\
    & \exp\left(\frac{2i\pi}{n}\right) & & \\
    & & \Ddots & \\
    & & & \exp\left(\frac{2i(n-1)\pi}{n}\right)
  \end{pNiceMatrix} \text{,}
\end{equation}
Such approximations are discussed informally in various contexts, from
quantum mechanics in finite dimension \cite{Weyl} to matrix models in
quantum field theory and string theory
\cite{Kimura01,Schreivogl13,Barrett15,Connes97,Seiberg99}. A first
formalization of this heuristics was given in \cite{Latremoliere13c},
where we proved that there indeed exist quantum metrics on
$C^\ast(C_n,S_n)$ such that the sequence $(C^\ast(C_n,S_n))_{n\in\N}$
converges, in the sense of the propinquity, to $C(\T^2)$, when $\T^2$
is endowed with its usual geodesic metric as a subspace of $\R^2$.

\medskip

The C*-algebra $C^\ast(C_n,S_n)$ is *-isomorphic to the C*-algebra of
$n\times n$ matrices. A mean to define a geometry on $C^\ast(C_n,S_n)$
is to define over it a spectral triple. We wish to show that 
spectral triple on $C^\ast(C_n,S_n)$ converges to a natural spectral
triple on $C(\T^2)$ for the spectral propinquity. We describe our
construction in \cite{Latremoliere21a} here for this special case.

\medskip

The C*-algebra $C(\T^2)$ is the universal C*-algebra generated by two
commuting unitaries; for instance, we can set
$U : (z_1,z_2) \in \T^2 \mapsto z_1$ and
$V : (z_1,z_2)\in\T^2\mapsto z_2$ and note that
$C(\T^2)=C^\ast(U,V)$. We define a unique, faithful *-representation
of $C(\T^2)$ on $\Hilbert_\infty = \ell^2(\Z^2)$ by setting, for all
$\xi \in \ell^2(\Z^2)$:
\begin{multline*}
  \pi_\infty(U)\xi : (m_1,m_2)\in\Z^2\mapsto \xi(m_1-1,m_2) \\ \text{
    and }\pi_\infty(V)\xi : (m_1,m_2)\in\Z^2\mapsto \xi(m_1,m_2-1)
  \text.
\end{multline*}
Let
$\mathrm{dom} = \left\{ \xi \in \ell^2(\Z^2) :
  \left((|m_1|+|m_2|)\xi(m_1,m_2)\right) \in \ell^2(\Z^2)
\right\}$. We also define, for all $\xi \in \mathrm{dom}$:
\begin{multline*}
  \partial_U\xi : (m_1,m_2)\in\Z^d \mapsto i m_1 \xi(m_1,m_2) \\
  \text{ and } \partial_V\xi : (m_1,m_2)\in\Z^d \mapsto i m_2
  \xi(m_1,m_2) \text.
\end{multline*}

The canonical moving frame of the compact Lie group $\T^2$ is given by
$[\partial_U,\pi_\infty(\cdot)]$ and $[\partial_V,\pi_\infty(\cdot)]$.

\medskip

Fix $n\in\N\setminus\{0\}$. The C*-algebra $C^\ast(C_n,S_n)$ is the
universal C*-algebra generated by two unitaries whose multiplicative
commutator is $\zeta_n = \exp\left(\frac{2 i \pi}{n}\right)$, and it is *-isomorphic to the C*-algebra $\alg{M}_n$ of $n\times n$ matrices. Let $\Hilbert_n$ be the Hilbert space obtained by endowing $\alg{M}_n = C^\ast(C_n,S_n)$ with the usual inner product $(a,b)\in\alg{M}_n \mapsto \inner{a}{b}{n} \coloneqq \mathrm{trace}(a^\ast b)$, with $\mathrm{trace}$ the usual normalized trace over $\alg{M}_n$. As $C^\ast(C_n,S_n)$ is finite dimensional, it is complete for the Hilbert norm associated with $\inner{\cdot}{\cdot}{n}$. Note that $C^\ast(C_n,S_n)$ acts on $\Hilbert_n$ canonically on both the left and the right, in such a way that $\Hilbert_n$ is a bimodule over $C^\ast(C_n,S_n)$.

As explained, for instance, in \cite{Kimura01,Schreivogl13,Barrett15}, natural substitutes for the derivations $\partial_U$ and
$\partial_V$, when working with fuzzy tori, is given by the commutators:
\begin{equation*}
  \frac{n}{2\pi} [C_n,\cdot] \text{ and } \frac{n}{2 \pi} [S_n^\ast,\cdot]
\end{equation*}
seen as operators on $\Hilbert_n$. However, we will want to work with skew-adjoint operators in order to define a Dirac operator. With this in mind, we are able to prove the following convergence result.

\begin{theorem}[{\cite{Latremoliere21a}}]
  We use the notations developed in this subsection.
  
  Let $\gamma_1,\gamma_2,\gamma_3,\gamma_4$ be $4\times 4$ matrices
  such that
  \begin{equation*}
    \forall j,s \in \{1,\ldots,4\} \quad \gamma_j \gamma_s + \gamma_s \gamma_j = \begin{cases}
      -2 \text{ if $j=s$,} \\
      0  \text{ otherwise.}
    \end{cases}
  \end{equation*}
  
  For all $n\in\N$, we define the operator $D_n$ on
  $\ell^2(H_n)\otimes\C^4$, by
  \begin{multline*}
    D_n = \frac{n}{2\pi} \bigg(
    \left[\frac{C_n+C_n^\ast}{2},\cdot\right]\otimes \gamma_1 +
    \left[\frac{C_n-C_n^\ast}{2i},\cdot\right]\otimes \gamma_2 \\
    + \left[\frac{S_n+S_n^\ast}{2},\cdot\right]\otimes \gamma_3 +
    \left[\frac{S_n^\ast - S_n}{2i},\cdot\right]\otimes \gamma_4
    \bigg)\text.
  \end{multline*}

  We also define the operator $D_\infty$ from
  $\mathrm{dom}\otimes\C^4$ to $\ell^2(\Z^2)\otimes\C^4$ by setting:
  \begin{multline*}
    D_\infty =
    \pi_\infty\left(\frac{U+U^\ast}{2}\right)\partial_V\otimes
    \gamma_1 +
    \pi_\infty\left(\frac{U-U^\ast}{2}\right)\partial_V\otimes
    \gamma_2 \\ +
    \pi_\infty\left(\frac{V+V^\ast}{2}\right)\partial_U\otimes
    \gamma_3 +
    \pi_\infty\left(\frac{V^\ast-V}{2}\right)\partial_U\otimes
    \gamma_4 \text.
  \end{multline*}

  For each $n\in\N\cup\{\infty\}$, the triple
  $(\A_n,\mathscr{J}_n,D_n)$ is a metric spectral triple over $\A_n$,
  and moreover:
  \begin{equation*}
    \lim_{n\rightarrow\infty} \spectralpropinquity{5}((\A_n,\mathscr{J}_n,D_n),(\A_\infty,\mathscr{J}_\infty,D_\infty)) = 0\text,
  \end{equation*}
  where $\spectralpropinquity{5}$ is the spectral propinquity for the
  admissible triple $(F,F_{\mathsf{inner}},F_{\mathsf{mod}})$, with
  \begin{multline*}
    F:x,y,l_x,l_y\in\R_+\mapsto x l_y + y l_x\text{, }G : x,y,z\in\R_+
    \mapsto 5(x+y)z \\ \text{ and }H:x,y\in\R_+\mapsto 2 x^2 y^2
    \text.
  \end{multline*}
\end{theorem}

\begin{remark}
  We note that we relax a little bit the Leibniz condition for the
  previous result, to accommodate some of the tunnels constructed in
  \cite{Latremoliere21a}.
\end{remark}

\medskip

In fact, the above construction can be generalized to noncommutative
limits. In general, we require the introduction of certain auxiliary
unitaries to construct our spectral triples over quantum tori, so the
description is more involved. We establish in \cite{Latremoliere21a}
that, for any quantum torus, there exists a natural spectral triple,
constructed from the dual action of the torus, which is the limit of
spectral triples on fuzzy tori, with the obvious requirement on the
twist of the fuzzy tori (coded in a $2$-cocycle) to converge to the
twist of the quantum torus (similarly coded).

\bibliographystyle{amsplain} \bibliography{../thesis} \vfill

\end{document}